\documentclass[11pt, twoside, leqno]{article}

\usepackage{amssymb}
\usepackage{amsmath}
\usepackage{amsthm}
\usepackage{color}
\usepackage{mathrsfs}

\usepackage{indentfirst}

\usepackage{txfonts}

\allowdisplaybreaks

\def\rr{{\mathbb R}}
\def\rn{{\mathbb{R}^n}}

\def\zz{{\mathbb Z}}

\def\nn{{\mathbb N}}

\def\cp{{\mathcal P}}

\def\fz{\infty }

\def\lz{\lambda}

\def\lf{\left}
\def\r{\right}
\def\lfz{{\lfloor}}
\def\rfz{{\rfloor}}
\def\la{\langle}
\def\ra{\rangle}

\def\gs{\gtrsim}

\def\noz{\nonumber}
\def\wz{\widetilde}

\def\fin{\mathop\mathrm{\,fin\,}}
\def\BMO{\mathop\mathrm{\,BMO\,}}
\def\bmo{\mathop\mathrm{\,bmo\,}}

\def\loc{{\mathop\mathrm{\,loc\,}}}
\def\supp{\mathop\mathrm{\,supp\,}}

\def\Xint#1{\mathchoice
{\XXint\displaystyle\textstyle{#1}}%
{\XXint\textstyle\scriptstyle{#1}}%
{\XXint\scriptstyle\scriptscriptstyle{#1}}%
{\XXint\scriptscriptstyle\scriptscriptstyle{#1}}%
\!\int}
\def\XXint#1#2#3{{\setbox0=\hbox{$#1{#2#3}{\int}$ }
\vcenter{\hbox{$#2#3$ }}\kern-.6\wd0}}

\def\dashint{\Xint-}

 \def\la{{\langle}}
 \def\ra{{\rangle}}
\def\({\left(}
\def \){ \right)}

\def\lz{{\lambda}}

 \def\supp{\operatorname{supp}}

\newtheorem{theorem}{Theorem}[section]
\newtheorem{lemma}[theorem]{Lemma}
\newtheorem{corollary}[theorem]{Corollary}
\newtheorem{proposition}[theorem]{Proposition}

\theoremstyle{definition}
\newtheorem{remark}[theorem]{Remark}
\newtheorem{definition}[theorem]{Definition}
\renewcommand{\appendix}{\par
   \setcounter{section}{0}%
   \setcounter{subsection}{0}%
   \setcounter{subsubsection}{0}%
   \gdef\thesection{\@Alph\c@section}%
   \gdef\thesubsection{\@Alph\c@section.\@arabic\c@subsection}%
   \gdef\theHsection{\@Alph\c@section.}%
   \gdef\theHsubsection{\@Alph\c@section.\@arabic\c@subsection}%
   \csname appendixmore\endcsname
 }

\numberwithin{equation}{section}

\begin{document}

\arraycolsep=1pt

\title{\bf\Large Real-Variable Characterizations of Orlicz-Slice Hardy Spaces
\footnotetext{\hspace{-0.35cm} 2010 {\it
Mathematics Subject Classification}. Primary 42B30;
Secondary 42B25, 42B35, 46E30.
\endgraf {\it Key words and phrases.} Hardy space, slice space, Orlicz space, atom, maximal function,
Calder\'on-Zygmund operator.
\endgraf This project is supported
by the National Natural Science Foundation of China (Grant Nos.
11726621, 11726622, 11761131002, 11571039 and 11471042).}}
\author{Yangyang Zhang, Dachun Yang\,\footnote{Corresponding author
/ March 25, 2018 / newest version.}, \  Wen Yuan and Songbai Wang}
\date{}
\maketitle

\vspace{-0.7cm}

\begin{center}
\begin{minipage}{13cm}
{\small {\bf Abstract}\quad In this article, the authors first introduce a class of Orlicz-slice
spaces which generalize the slice spaces recently studied by P. Auscher et al.
Based on these Orlicz-slice spaces, the authors introduce a new kind of Hardy type spaces,
the Orlicz-slice Hardy spaces, via the radial maximal functions. This new scale of Orlicz-slice Hardy spaces
contains the variant of the Orlicz-Hardy space of A. Bonami and J. Feuto
as well as the Hardy-amalgam space of Z. V. de P. Abl\'e and J. Feuto as special cases.
Their characterizations  via the atom, the molecule,
various maximal functions, the Poisson integral and the
Littlewood-Paley functions are also obtained. As an application of these
characterizations, the authors establish their finite atomic characterizations,
which further induce a description of their dual spaces and a criterion on the boundedness
of sublinear operators from these Orlicz-slice Hardy spaces into a quasi-Banach space.
Then, applying this criterion, the authors obtain the boundedness of $\delta$-type Calder\'on-Zygmund
operators on these Orlicz-slice Hardy spaces. All these results are new even for slice Hardy spaces
and, moreover, for Hardy-amalgam spaces, the Littlewood-Paley function characterizations, the dual spaces
and the boundedness of $\delta$-type Calder\'on-Zygmund
operators on these Hardy-type spaces are also new.
}
\end{minipage}
\end{center}


\section{Introduction\label{s1}}

The Hardy spaces $H^p(\rn)$, with $p\in (0, 1]$, are known to be one of the most important
working spaces on $\rn$ in harmonic analysis and partial differential equations, which play key roles
in many branches of analysis; see, for example,
\cite{fs72,fs82,sto89,em,m94}.
In particular, $H^p(\rn)$, with $p\in (0, 1]$, are good substitutes of Lebesgue spaces $L^p(\rn)$
when studying the boundedness of Calder\'on-Zygmund operators.
In recent decades, in order to meet the requirements arising in the study of the boundedness of operators,
partial differential equations and some other fields, various variants of Hardy spaces have been introduced and
developed, such as weak Hardy spaces (see, for example, \cite{frs74,fs87}), Hardy-Lorentz spaces (see, for example, \cite{at07,abr17,lyy16,lyy17,lyy18}) and
Orlicz-Hardy spaces (see, for example, \cite{l66,l67,sto76,yll}).
Recently, in \cite{af1}, as a generalization of the classical Hardy space and the Lorentz-Hardy space,
Abl\'e and Feuto introduced the Hardy type space  $\mathcal{H}^{(p,q)}(\rn)$ with $p,\,q\in(0,\infty)$
based on the N. Weiner amalgam spaces $(L^p,\ell^q)(\rn)$ and
obtained an atomic decomposition of these Hardy-amalgam spaces when $q\in(0,\fz)$ and $p\in(0,\min\{1,q\})$.
In \cite{bf}, Bonami and Feuto introduced the Hardy type spaces $H_*^\Phi(\rn)$ and $h_*^\Phi(\rn)$
with respect to the amalgam space $(L^\Phi,\ell^1)(\rn)$, where
$\Phi(t):=\frac{t}{\log(e+t)}$ for any $t\in[0,\infty)$ is an Orlicz function,
and applied these Hardy-type spaces
to study the linear decomposition of the product of the
Hardy space $H^1(\rn)$ and its dual space $\BMO(\rn)$ as well as the
local Hardy space $h^1(\rn)$ and its dual space $\bmo(\rn)$. Moreover, very recently,
Cao et al. \cite{cky1} applied $h_*^\Phi(\rn)$ to study the bilinear
decomposition of the product of the
local Hardy space $h^1(\rn)$ and its dual space  $\bmo(\rn)$. Recall that both the Hardy type spaces
$H_*^\Phi(\rn)$ and $h_*^\Phi(\rn)$ were defined in \cite{bf} via the (local) radial maximal
functions, while $h_*^\Phi(\rn)$ in \cite{cky1} was defined via the local grand maximal
function. Moreover, \emph{no other real-variable characterizations} of both the Hardy type spaces
$H_*^\Phi(\rn)$ and $h_*^\Phi(\rn)$ are known so far.

On the other hand, recently, to study the classification of weak solutions in
the natural classes for the boundary value problems of a $t$-independent
elliptic system in the upper plane, Auscher and Mourgoglou \cite{am} introduced the slice spaces $E^q_t(\rn)$.
In \cite{ap}, Auscher and Prisuelos-Arribas further introduced a more general
slice space $(E^q_r)_t(\rn)$ and applied it to study the action of operators,
such as the Hardy-Littlewood maximal operator, Calder\'on-Zygmund operators and Riesz potentials, on tent spaces.

More precisely, recall that the \emph{tent space $T^q_r(\rr_+^{n+1})$}, with $q,\,r\in(0,\fz)$,
consists of all measurable functions $F$ on $\rr_+^{n+1}:=\rn\times(0,\fz)$ such that
$$\|F\|_{T^q_r(\rr_+^{n+1})}:=\lf\|\lf[\int_0^\fz\int_{B(\cdot,t)}|F(y,t)|^r
\,\frac{dy\,dt}{t^{n+1}}\r]^{1/r}\r\|_{L^q(\rn)}<\fz,$$
here and hereafter, for any $(x,t)\in\rr_+^{n+1}$, $B(x,t):=\{y\in\rn:\ |y-x|<t\}$.
It is known (see \cite{cms}) that $T^q_r(\rr_+^{n+1})$ can be represented as $\sum_{i=1}^\fz\lz_i A_i$
with $\{\lz_i\}_{i\in\nn}\in \ell^q$ and $\{A_i\}_{i\in\nn}$ being $T^q_r(\rr_+^{n+1})$-atoms,
that is, for any $i\in\nn$, there exists a
ball $B_i\subset \rn$ such that
$$\supp (A_i)\subset \widehat{B}_i:=\{(x,t)\in\rr_+^{n+1}:\ d(x, \rn\setminus B_i)\ge t\}$$
and $\iint_{\widehat{B}_i}|A_i(x,t)|^r\,\frac{dx\,dt}{t}\le |B_i|^{1-\frac{r}{q}}$,
where $d(x, \rn\setminus B_i):=\inf\{|x-y|:\ y\in\rn\setminus B_i\}$. As a subspace of
$T^q_r(\rr_+^{n+1})$, Auscher and Prisuelos-Arribas \cite{ap} introduced the space
$\mathfrak{T}^q_r(\rr_+^{n+1})$  consisting of all functions $F\in T^q_r(\rr_+^{n+1})$
which can be represented as
$\sum_{i=1}^\fz\lz_i A_i$ with $\{\lz_i\}_{i\in\nn}\in \ell^q$ and $T^q_r(\rr_+^{n+1})$-atoms $\{A_i\}_{i\in\mathbb{N}}$
satisfying the additional moment condition $\int_{\rr_+^{n+1}}A_i(x,t)\,dx=0$ for almost every $t\in (0,\fz)$
and any $i\in\mathbb{N}$. In \cite{ap},
Auscher and Prisuelos-Arribas studied the behaviors of the Hardy-Littlewood
maximal operator, Calder\'on-Zygmund operators and Riesz potentials
on $T^q_r(\rr_+^{n+1})$ and $\mathfrak{T}^q_r(\rr_+^{n+1})$. As Auscher and Prisuelos-Arribas
mentioned in \cite{ap},
``it would be interesting to explore further these spaces (interpolation, etc)
and their applications".

One key tool used in \cite{ap} is the slice space which is defined via slicing the classical
tent space norm at a fixed height. Recall that, for any $q,\ r,\ t\in(0,\fz)$,
the \emph{slice space $(E^q_r)_t(\rn)$} in \cite{ap} is  defined as the space of all locally $r$-integrable
functions $f$ on $\rn$ such that
\begin{equation}\label{slice}
\|f\|_{(E^q_r)_t(\rn)}:=\lf\{\int_\rn\left[t^{-n}\int_{B(x,t)}|f(y)|^r\,dy\right]^{q/r}
\,dx\right\}^{1/q}<\fz.
\end{equation}
In particular, $E^q_t(\rn):=(E^q_2)_t(\rn)$ was introduced in \cite{am}.
A subspace $(\mathfrak{C}_r^q)_t(\rn)$ of $(E^q_r)_t(\rn)$ was also introduced in \cite{ap}
in a way similar to $\mathfrak{T}^q_r(\rr_+^{n+1})$ (see also Definition \ref{auscher} below).
These slice spaces $(E^q_r)_t$ and $(\mathfrak{C}_r^q)_t$  were proved in \cite{ap}
to be the \emph{retracts} of the tent spaces $T^q_r(\rr_+^{n+1})$ and $\mathfrak{T}^q_r(\rr_+^{n+1})$, respectively.
They are also special cases of the  Wiener-amalgam spaces (see \cite{f})
which were first introduced by N. Wiener and
further developed in time-frequency analysis and sampling theory.
Properties of slice spaces such as the duality, the atomic decomposition
and the interpolation were also clarified in \cite{am,ap}. Observe that the Hardy type space
$(\mathfrak{C}_r^q)_t(\rn)$ [and also $\mathfrak{T}^q_r(\rr_+^{n+1})$] was
introduced in \cite{ap} via atoms and \emph{no other real-variable characterizations} of
these Hardy type spaces are known so far.

Let $\Phi$ be an Orlicz function on $[0,\fz)$ and $q,\ t\in (0,\fz)$.
Motivated by the aforementioned works, in this article, we first introduce a class of Orlicz-slice
spaces, $(E_\Phi^q)_t(\rn)$, which generalize the slice spaces [in this case,
$\Phi(\tau):=\tau^r$ for any $\tau\in [0,\fz)$ with $r\in(0,\infty)$] recently defined and studied by Auscher and Mourgoglou \cite{am} (the case $r=2$)
as well as by Auscher and Prisuelos-Arribas \cite{ap}.
Based on these Orlicz-slice spaces, we then introduce a new kind of Hardy-type spaces, the Orlicz-slice Hardy spaces
$(HE_\Phi^q)_t(\rn)$, via the radial maximal functions. This new scale of Orlicz-slice Hardy spaces
contains the variant of the Orlicz-Hardy space, $H_*^\Phi(\rn)$
[in this case, $q=t=1$], of Bonami and Feuto \cite{bf} as well as the Hardy-amalgam space [in this case, $t=1$ and $\Phi(\tau):=\tau^p$ for any $\tau\in [0,\fz)$ with $p\in (0,\fz)$] of Abl\'e and Feuto \cite{af1}
as special cases. Their characterizations  via the atom, the molecule,
various maximal functions, the Poisson integral and the
Littlewood-Paley functions are also obtained. As an application of these
characterizations, we then establish finite atomic characterizations
of Orlicz-slice Hardy spaces, which further induce a description of their dual spaces
and a criterion on the boundedness of sublinear operators from these Orlicz-slice Hardy spaces into a quasi-Banach space.
Then, applying this criterion, we obtain the boundedness of $\delta$-type Calder\'on-Zygmund
operators on these Orlicz-slice Hardy spaces. Moreover, the relations between
the Orlicz-slice space and the Orlicz-slice Hardy space,
or between the Hardy-type space $(\mathfrak{C}_r^q)_t(\rn)$,
with $t\in (0,\fz)$, $r\in (1,\fz)$ and $q\in (\frac n{n+1}, 1]$,
from \cite{ap} and $(HE_\Phi^q)_t(\rn)$ in the case when $\Phi(\tau):=\tau^s$
for any $\tau\in [0,\fz)$ with $s\in (\frac n{n+1},q]$ are also clarified.
All these results of this article are new even for slice Hardy spaces
and, moreover, for Hardy-amalgam spaces, the molecular characterization,
the Littlewood-Paley function characterizations, the dual spaces
and the boundedness of $\delta$-type Calder\'on-Zygmund
operators on these Hardy-type spaces are also new.
Thus, the results obtained in this article essentially complement and
generalize the real-variable theories of the
Hardy-amalgam space in  \cite{af1}
as well as the Hardy-type space $H_*^\Phi(\rn)$ in \cite{bf}.

To be more precise, in Section \ref{s2} of this article, we introduce the notion
of Orlicz-slice spaces $(E_\Phi^q)_t(\rn)$ and then present some basic properties of
$(E_\Phi^q)_t(\rn)$, such as their equivalence relation with the Orlicz-amalgam spaces
(see Proposition \ref{th2} below), the Fefferman-Stein vector-valued inequality
for the Hardy-Littlewood maximal operator on $(E_\Phi^q)_t(\rn)$ (see Theorem \ref{main} below),
the boundedness of the Hardy-Littlewood maximal operator on
$(E_\Phi^q)_t(\rn)$ (see Corollary \ref{main2} below),
and the dual spaces of $(E_\Phi^q)_t(\rn)$ (see Theorem \ref{dual} below).
The boundedness of the Hardy-Littlewood maximal operator on $(E_\Phi^q)_t(\rn)$
is a key tool in this article.
Recall that the boundedness of the  Hardy-Littlewood maximal operator
on the amalgam space $(L^p,\ell^q)(\rn)$ with $p,\,q\in(1,\infty)$
was obtained in \cite{chh}. However, the approach used in  \cite{chh}  for
$(L^p,\ell^q)(\rn)$ is no longer feasible for  $(E_\Phi^q)_t(\rn)$
because the quasi-norm $\|\cdot\|_{L^\Phi(\rn)}$ cannot be represented as an integral
and hence cannot apply the weighted boundedness of the Hardy-Littlewood maximal operator.
To overcome this obstacle, we employ a different method, namely, we first establish a generalization
of \cite[Lemma 4.1]{ap} via replacing the maximal function and $L^r(\rn)$ norm therein,
respectively, by the vector-valued maximal function and $L^\Phi(\rn)$ norm here (see Lemma
\ref{mm} below), which plays a key role in establishing the Fefferman-Stein
vector-valued inequality for the Hardy-Littlewood maximal operator on $(E_\Phi^q)_t(\rn)$
(see the proof of Theorem \ref{main}); from Theorem \ref{main}, we immediately induce
the desired boundedness of the Hardy-Littlewood maximal
operator on $(E_\Phi^q)_t(\rn)$. We also point out that the proof of Theorem \ref{dual}
strongly depends on Proposition \ref{th2} and the well-known dual spaces of Orlicz-amalgam spaces.
Moreover, in Lemma \ref{ballproof} below, we further prove that the Orlicz-slice spaces $(E_\Phi^q)_t(\rn)$
are ball quasi-Banach function spaces considered in \cite{ykds} and hence all results from
\cite{ykds} are applicable to $(E_\Phi^q)_t(\rn)$.

In Section \ref{s3}, based on the Orlicz-slice spaces $(E_\Phi^q)_t(\rn)$, we first introduce
the Orlicz-slice Hardy spaces,
$(HE_\Phi^q)_t(\rn)$, which are defined via the radial maximal functions
(see Definition \ref{dh} below) and then present some fundamental
properties of these Orlicz-slice Hardy spaces $(HE_\Phi^q)_t(\rn)$ including
characterizations via the grand and the non-tangential maximal functions (see Theorem \ref{mdj} below),
the poisson integral (see Theorem \ref{poisson} below), the atom (see Theorem \ref{atom ch} below), the molecule (see Theorem \ref{molecular} below),
the Littlewood-Paley functions (see Theorems \ref{lusin}, \ref{gfunction} and \ref{glamda} below)
and the finite atomic decomposition (see Theorem \ref{finite} below).
We also clarify the relations between $(E_\Phi^q)_t(\rn)$ and $(HE_\Phi^q)_t(\rn)$ in Theorem \ref{dayu1} below.

The proofs of  all main results in Section \ref{s3} are given
in Section \ref{s4}. Recall that a real-variable theory of Hardy spaces related to ball
quasi-Banach function spaces was recently developed in \cite{ykds}.
The results obtained in \cite{ykds} are of so wide generality that, in Section \ref{s4},
we can directly apply them to obtain the atomic and the
molecular characterizations as well as those characterizations via various maximal functions,
the Poisson integral and the Lusin area function of the Orlicz-slice Hardy space  $(HE_\Phi^q)_t(\rn)$
as well as the relation between $(HE_\Phi^q)_t(\rn)$
and $(E_\Phi^q)_t(\rn)$. Then, using the atomic characterization, we further establish
the Littlewood-Paley $g$-function and $g_\lz^\ast$-function characterizations
of $(HE_\Phi^q)_t(\rn)$ and also the finite atomic characterization.

We point out that, in \cite{af1}, Abl\'e and Feuto introduced the
Hardy-amalgam space $\mathcal{H}^{(p,q)}(\rn)$ and obtained the non-tangential
maximal function characterization, the Poisson integral characterization and
the atomic decomposition as well as the finite atomic decomposition of this space.
To the best of our knowledge, this might be the first article to deal with
the real-variable theory of Hardy spaces based on amalgam spaces.
Comparing with \cite{af1}, the approach used in this article via the general theory of \cite{ykds}
for the corresponding characterizations of $(HE_\Phi^q)_t(\rn)$ is much
simpler. Also, comparing with the atomic characterization of $\mathcal{H}^{(p,q)}(\rn)$
obtained in \cite{af1}, the atomic characterization of $(HE_\Phi^q)_t(\rn)$ obtained
in this article holds true on a wider range even when $(HE_\Phi^q)_t(\rn)$ is reduced
to $\mathcal{H}^{(p,q)}(\rn)$ [in this case, $t=1$ and $\Phi(\tau):=\tau^p$
for any $\tau\in [0,\fz)$ with $p\in (0,\fz)$], which improves the related result in \cite{af1}.

In Section \ref{s5},
as an application of both the atomic characterization (Theorem \ref{atom ch}) and
the finite atomic characterization (Theorem \ref{finite}) of
$(HE_\Phi^q)_t(\rn)$ obtained in Section \ref{s3},
we prove that the dual spaces of $(HE_\Phi^q)_t(\rn)$
can be described as certain Campanato spaces related to the Orlicz-slice spaces (see Theorem \ref{dual2} below).

The last section, Section \ref{s6}, is devoted to some
 further applications of the characterizations
obtained in Section \ref{s3}. We first establish a criterion on the boundedness
of sublinear operators from $(HE_\Phi^q)_t(\rn)$
into a quasi-Banach space (see Theorem \ref{suanzi1} and Corollary \ref{suanzi} below),
which are further used to obtain the boundedness of
the $\delta$-type Calder\'on-Zygmund operators
on $(HE_\Phi^q)_t(\rn)$ (see Theorems \ref{cz} and \ref{cz2}).
Moreover, in Proposition \ref{au} below, we clarify the relation between $(\mathfrak{C}_r^q)_t(\rn)$,
with $t\in (0,\fz)$, $r\in (1,\fz)$ and $q\in (\frac n{n+1}, 1]$,
from \cite{ap} and $(HE_\Phi^q)_t(\rn)$ in the case when $\Phi(\tau):=\tau^s$
for any $\tau\in [0,\fz)$ with $s\in (0,q]$.

Observe that a real-variable theory of local Hardy spaces based on the
Orlicz-slice spaces can also be developed. However, to limit the length
of this article, we will consider this local version in another article.

Finally, we make some convention on notation.
For any $x\in\rn$ and $r\in(0,\infty)$, let $B(x,r):=\{y\in\rn:|x-y|<r\}$ and $\bar{B}(x,r)$
be its \emph{closure} in $\rn$.
For any $r\in(0,\infty)$, $f\in L^1_\loc(\rn)$ and $x\in\rn$, let
$$
\dashint_{B(x,\,r)}f(y)\,dy:=\frac1{|B(x,r)|}\int_{B(x,r)}f(y)\,dy,
$$
here and hereafter, $L^1_\loc(\rn)$ denotes the space of all locally integrable functions.
For any set $E$, we use $\chi_{E}$ to denote its \emph{characteristic function} and $\# E$
its \emph{cardinality}. We also use $\vec{0}_n$ to denote the \emph{origin} of $\rn$.
Let $\mathcal{S}(\rn)$ denote the collection of all
\emph{Schwartz functions} on $\rn$, equipped
with the classical well-known topology, and $\mathcal{S}'(\rn)$ its \emph{topological dual}, namely,
the collection of all bounded linear functionals on $\mathcal{S}(\rn)$
equipped with the weak-$\ast$ topology.
Let $\mathbb{N}:=\{1,\,2,...\}$ and $\mathbb{Z}_+:=\mathbb{N}\bigcup\{0\}$.
Denote by the \emph{symbol} $\mathcal{Q}$ the set of all cubes having their edges parallel to the coordinate axes.
Also, for any $x\in\rn$ and $l\in(0,\fz)$, $Q(x,l)$ denotes the cube with the center $x$ and the side-length $l$. Furthermore, for any cube $Q\in\mathcal{Q}$
and $j\in\mathbb{Z}_+$, let $S_j(Q):=(2^{j+1}Q)\setminus(2^jQ)$ with $j\in\mathbb{N}$ and $S_0(Q):=2Q$. For any $\varphi\in\mathcal{S}(\rn)$ and $t\in(0,\infty)$,
let $\varphi_t(\cdot):=t^{-n}\varphi(t^{-1}\cdot)$. For any $s\in\mathbb{R}$, we denote by $\lfloor s\rfloor$ the \emph{largest integer not greater than} $s$.
For any $p\in[0,1]$, let $p'$ be its \emph{conjugate index}, that is, $p'$ satisfies  $1/p+1/p'=1$.
We always use $C$ to denote a \emph{positive constant}, which is independent of the main parameter,
but it may vary from line to line.
Moreover, we use $C_{(\gamma,\ \beta,\ \ldots)}$ to denote a positive constant depending on the indicated
parameters $\gamma,\ \beta,\ \ldots$. If, for any real functions $f$ and $g$, $f\leq Cg$, we then write
$f\lesssim g$ and, if $f\lesssim g\lesssim f$, we then write $f\sim g$.

\section{Orlicz-slice spaces\label{s2}}

In this section, we introduce the slice spaces related to Orlicz functions and
present some of their basic properties such as the boundedness of maximal operators,
which are used in the later sections.
We begin with the notions of both Orlicz functions and Orlicz spaces (see, for example, \cite{mmz}).

\begin{definition}\label{d1.1}
A function $\Phi:\ [0,\infty)\ \rightarrow\ [0,\infty)$ is called an \emph{Orlicz function} if it is
non-decreasing and satisfies $\Phi(0)= 0$, $\Phi(t)>0$ whenever $t\in(0,\infty)$ and $\lim_{t\rightarrow\infty}\Phi(t)=\infty$.
An Orlicz function $\Phi$ is said to be of \emph{lower} (resp., \emph{upper}) \emph{type} $p$ with
$p\in(-\infty,\infty)$ if
there exists a positive constant $C_{(p)}$, depending on $p$, such that, for any $t\in[0,\infty)$
and $s\in(0,1)$ [resp., $s\in [1,\infty)$],
\begin{equation*}
\Phi(st)\le C_{(p)}s^p \Phi(t).
\end{equation*}
A function $\Phi:\ [0,\infty)\ \rightarrow\ [0,\infty)$
is said to be of \emph{positive lower} (resp., \emph{upper}) \emph{type} if it is of lower (resp., upper) type $p$ for some $p\in(0,\infty)$.
\end{definition}

\begin{definition}\label{d1.2}
Let $\Phi$ be an Orlicz function with positive lower type $p_{\Phi}^-$ and positive upper type $p_{\Phi}^+$.
The \emph{Orlicz space $L^\Phi(\rn)$} is defined
to be the set of all measurable functions $f$ such that
 $$\|f\|_{L^\Phi(\rn)}:=\inf\lf\{\lambda\in(0,\infty):\ \int_{\rn}\Phi\lf(\frac{|f(x)|}{\lambda}\r)\,dx\le1\r\}<\infty.$$
\end{definition}

We now give some basic properties of Orlicz functions.

\begin{lemma}\label{lem1}
Let $\Phi$ be an Orlicz function with positive
upper type $p_{\Phi}^+$. Then there exists a positive constant $C$ such that
$$\Phi(t_1+t_2)\le C\lf[\Phi(t_1)+\Phi(t_2)\r],\quad\forall\ t_1,\ t_2\in[0,\infty).$$
\end{lemma}

\begin{proof}
Obviously we only need to consider the case when $t_1+t_2>0$. If $p_{\Phi}^+ \in(0,1]$, then, for
any $i\in\{1,2\}$,
$$\frac{t_i}{t_1+t_2}\Phi(t_1+t_2)\lesssim\Phi(t_i)$$
and hence
$$t_i\Phi(t_1+t_2)\lesssim\Phi(t_i)(t_1+t_2),$$
which, via taking the summation on $i$ on both side, further implies the desired conclusion.
If $p_{\Phi}^+ \in(1,\infty)$, then let $\widetilde{\Phi}(t):=\Phi(t^{1/p_{\Phi}^+})$
for any $t\in[0,\infty)$.
It is easy to check that $\widetilde{\Phi}$ is an Orlicz function of upper type $1$ and hence, by the proved conclusion,
we have
$$
\Phi(t_1+t_2)=\widetilde{\Phi}\lf([t_1+t_2]^{p_{\Phi}^+}\r)\lesssim\widetilde{\Phi}\lf(t_1^{p_{\Phi}^+}\r)
+\widetilde{\Phi}\lf(t_2^{p_{\Phi}^+}\r)
\sim\Phi(t_1)+\Phi(t_2).
$$
This finishes the proof of Lemma \ref{lem1}.
\end{proof}

\begin{remark}
When $\Phi$ is an Orlicz function with positive
upper type $p_{\Phi}^+$, from Lemma \ref{lem1}, it is easy to deduce that $\|\cdot\|_{L^\Phi(\rn)}$ is a
quasi-norm.
\end{remark}

The following lemma is well known.

\begin{lemma}\label{lem2}
Let $\Phi$ be an Orlicz function with positive lower type $p_{\Phi}^-$ and positive upper type $p_{\Phi}^+$ and
$$\widetilde{\Phi}(t):=\int_0^t\frac{\Phi(s)}{s}\,ds,\quad \forall\,t\in(0,\infty).$$
Then $\widetilde{\Phi}$ is also an Orlicz function, which is equivalent to $\Phi$ and $\widetilde{\Phi}$ is continuous and strictly increasing.
\end{lemma}

\begin{remark}\label{re}
Observe that all the results stated in this article are invariant under the change of equivalent
Orlicz functions. Moreover, equivalent Orlicz functions share the same positive upper and the same lower type numbers.
In what follows, by Lemma \ref{lem2}, without loss of generality, we may \emph{always
assume} that an Orlicz function $\Phi$ is continuous and strictly increasing.
\end{remark}

\begin{lemma}\label{lem3}
Let $\Phi$ be an Orlicz function with positive lower type $p_{\Phi}^-$.
If the inequality that
\begin{equation*}
\int_{\rn}\Phi\lf(\frac{|f(x)|}{\lambda}\r)\,dx\leqslant\widetilde{C}\quad for\ some\ \lambda\in(0,\infty)
\ and\ positive\ constant\ \widetilde{C}
\end{equation*}
holds true, then there exists a positive constant $C$, depending on
$\widetilde{C}$ and $p_{\Phi}^-$, such that $\|f\|_{L^\Phi(\rn)}\le C\lambda$.
\end{lemma}

\begin{proof}
The proof is simple and we can take $C:=(1+\widetilde{C}C_{(p_{\Phi}^-)})^{1/p_{\Phi}^-}$
with $\widetilde{C}$ as in the assumption of Lemma \ref{lem3}.
This finishes the proof of Lemma \ref{lem3}.
\end{proof}

Now we introduce the Orlicz-slice space and the Orlicz-amalgam space. The former is a
generalization of the slice spaces introduced in \cite{am,ap}, and the
latter is a generalization of the classical amalgam space $(L^p,\ell^q)$ defined by
N. Wiener in 1926, in the formulation of his generalized harmonic analysis.

\begin{definition}\label{d2}
Let $t,\ q\in(0,\infty)$ and $\Phi$ be an Orlicz function with positive lower type $p_{\Phi}^-$ and positive upper type $p_{\Phi}^+$. The \emph{Orlicz-slice space} $(E_\Phi^q)_t(\rn)$ is defined to be the set of all measurable functions $f$
such that $$\|f\|_{(E_\Phi^q)_t(\rn)}
:=\lf\{\int_{\rn}\lf[\frac{\|f\chi_{B(x,t)}\|_{L^\Phi(\rn)}}
{\|\chi_{B(x,t)}\|_{L^\Phi(\rn)}}\r]^q\,dx\r\}^{\frac{1}{q}}<\infty.$$
\end{definition}

\begin{definition}\label{d1.3}
Let $t,\ q\in(0,\infty)$ and $\Phi$ be an Orlicz function with positive lower type $p_{\Phi}^-$ and positive upper type $p_{\Phi}^+$. The \emph{Orlicz-amalgam space} $\ell^q(L^\Phi_t)(\rn)$ is defined to be the set of all measurable functions $f$ such that
$$
\|f\|_{\ell^q(L^\Phi_t)(\rn)}:=\lf[\sum_{k\in \mathbb{Z}^n}\lf\|f\chi_{Q_{tk}}\r\|^q_{L^\Phi(\rn)}\r]^{\frac{1}{q}}<\infty,
$$
where $Q_{tk}:=t[k+[0,1)^n]$ for any $t\in(0,\infty)$ and $k\in\mathbb{Z}^n$.
\end{definition}

\begin{remark}\label{re2.10}
\begin{itemize}
\item [(i)] Both the Orlicz-slice space and the Orlicz-amalgam space fall into the scale of
Wiener-amalgam  spaces introduced by Feichtinger \cite{f}.
By Lemmas \ref{lem1}, \ref{lem3} and \cite[Theorem 1]{f}, we know that both the
Orlicz-slice space and the Orlicz-amalgam space are quasi-Banach spaces.

\item[(ii)] If $t=1$ and $\Phi(\tau):=\tau^p$
for any $\tau\in [0,\fz)$ with $p\in (0,\fz)$,
then $(E_\Phi^q)_t(\rn)$ coincides with the Weiner amalgam spaces $(L^p,\ell^q)(\rn)$ in \cite{af1}.
By \cite[Proposition 2.1]{af1}, we have $(L^p,\ell^q)(\rn)\subset L^p(\rn)\bigcap L^q(\rn)$ when $p\in(0, q)$
and $L^p(\rn)\bigcup L^q(\rn)\subset(L^p,\ell^q)(\rn)$ when $ q\in(0, p)$,
here and hereafter, for any $r\in(0,\infty]$, the \emph{symbol}
$L^r(\rn)$ denotes the set of all measurable functions $f$
such that
$$
\|f\|_{L^r(\rn)}:=\lf\{\int_{\rn}|f(x)|^r\,dx\r\}^{1/r}<\infty
$$
with the usual modification made when $r=\infty$.

\item[(iii)] If $\Phi(\tau):=\tau^r$ for any $\tau\in[0,\fz)$ with $r\in(0,\fz)$,
then $(E_\Phi^q)_t(\rn)$ and $(E_r^q)_t(\rn)$ from \cite{am,ap} coincide with equivalent
quasi-norms.
\end{itemize}
\end{remark}

The following proposition clarifies the relation between $(E_r^q)_t(\rn)$,
with $t,\ q,\ r\in(0,\infty)$, and $L^q(\rn)$, whose proof is a slight modification of
the proof of \cite[Proposition 2.1]{af1}.

\begin{proposition}\label{ggg}
Let $t,\ q,\ r\in(0,\infty)$.
\begin{itemize}
\item[\rm(i)]
If $r\in(0, q]$, then $L^r(\rn)\cup L^q(\rn)\subset (E_r^q)_t(\rn)$; precisely,
for any $f\in L^r(\rn)\cup L^q(\rn)$, then $f\in (E_r^q)_t(\rn)$ and
$\|f\|_{(E_r^q)_t(\rn)}\leq\min\{\|f\|_{L^r(\rn)}, \|f\|_{L^q(\rn)}\}$;

\item[\rm(ii)] If $q\in(0, r]$, then $(E_r^q)_t(\rn)\subset L^q(\rn)$; precisely,  for any
$f\in(E_r^q)_t(\rn)$, then $f\in L^q(\rn)$ and
$\|f\|_{L^q(\rn)}\leq\|f\|_{(E_r^q)_t(\rn)}$.

\item[\rm(iii)] $(E_q^q)_t(\rn)$ and $L^q(\rn)$ coincide with the same quasi-norms.
\end{itemize}
\end{proposition}

\begin{proof} Observe that (iii) is an immediate corollary of (i) and (ii). Thus,
to complete the proof of this proposition, we only need to show (i) and (ii).

We first show (i). In this case, for any $f\in L^q(\rn)$, using the H\"older inequality and the Fubini theorem, we obtain
\begin{align*}
\|f\|_{(E_r^q)_t(\rn)}&\le \lf[\int_{\rn}\lf\{\frac1{|B(x,t)|}\lf[\int_{B(x,t)}|f(y)|^q\,dy\r]^\frac rq|B(x,t)|^\frac1{(q/r)'}\r\}^\frac qr\,dx\r]^\frac1q\\
&=\lf\{\int_{\rn}\frac1{|B(x,t)|}\int_{B(x,t)}|f(y)|^q\,dy\,dx\r\}^\frac1q=\|f\|_{L^q(\rn)}.
\end{align*}
Also, for any $f\in L^r(\rn)$, applying the Minkowski inequality, we conclude that
$$
\|f\|_{(E_r^q)_t(\rn)}\le\lf[\int_{\rn}\lf\{\int_{\rn}\lf[t^{-n}\chi_{B(x,t)}(y)|f(y)|^r\r]^\frac qrdx\r\}^\frac rqdy\r]^\frac1r=\|f\|_{L^r(\rn)}.
$$
Thus, $L^r(\rn)\bigcup L^q(\rn)\subset(E_r^q)_t(\rn)$, that is, (i) holds true.

Now, we prove (ii). In this case, by the H\"older inequality, we have
\begin{align*}
\|f\|_{L^q(\rn)}&=\lf\{\int_{\rn}\frac1{|B(x,t)|}\int_{B(x,t)}|f(y)|^q\,dy\,dx\r\}^\frac1q\\
&\leq\lf\{\int_{\rn}\frac1{|B(x,t)|}\lf[\int_{B(x,t)}|f(y)|^r\,dy\r]^{\frac{q}{r}}
\lf[\int_{B(x,t)}dy\r]^{\frac{1}{(r/q)'}}\,dx\r\}^\frac1q
=\|f\|_{(E_r^q)_t(\rn)}.
\end{align*}
From this, we deduce $(E_r^q)_t(\rn)\subset L^q(\rn)$, which completes the proof of (ii) and hence of Proposition \ref{ggg}.
\end{proof}

Observing that, for any $x\in\rn$ and $t\in(0,\infty)$,
\begin{equation}\label{qiu}
\|\chi_{B(x,t)}\|_{L^\Phi(\rn)}=\lf[\Phi^{-1}\lf(\frac{1}{|B(x,t)|}\r)\r]^{-1}
=\lf[\Phi^{-1}\lf(\frac{1}{\varepsilon_nt^n}\r)\r]^{-1}=:\widetilde{C}_{(\Phi,t)}
\end{equation}
is independent of $x,$ where $\varepsilon_n$ denotes the volume of the unit ball
in $\rn$ and $\Phi^{-1}$ the inverse function of $\Phi$, we have the next proposition,
which shows that, for any $t\in(0,\infty)$, the Orlicz-slice space
$(E_\Phi^q)_t(\rn)$ is equivalent to the Orlicz-amalgam space $\ell^q(L^\Phi_t)(\rn)$.

\begin{proposition}\label{th2}
Let $t,\ q\in(0,\infty)$ and $\Phi$ be an Orlicz function with positive lower type $p_{\Phi}^-$ and positive upper type $p_{\Phi}^+$. Then $(E_\Phi^q)_t(\rn)$ and $\ell^q(L^\Phi_t)(\rn)$ coincide and, for any $f\in(E_\Phi^q)_t(\rn)$,
\begin{equation*}
\lf[t^{n}\sum_{k\in \mathbb{Z}^n}\lf\|f\chi_{Q_{  kt}}\r\|^q_{L^\Phi(\rn)}\r]^{\frac{1}{q}}
\sim\lf\{\int_{\rn}\lf[\lf\|f\chi_{B(x,t)}\r\|_{L^\Phi(\rn)}\r]^q\,dx\r\}^{\frac{1}{q}},
\end{equation*}
where the equivalent positive constants are independent of $f$ and $t$.
\end{proposition}

\begin{proof}
We first show that
\begin{equation}\label{dj}
\lf[t^{n}\sum_{k\in \mathbb{Z}^n}\lf\|f\chi_{Q_{tk}}\r\|^q_{L^\Phi(\rn)}\r]^{\frac{1}{q}}
\sim\lf\{\int_{\rn}\lf[\lf\|f\chi_{B(x,2\sqrt{n}t)}\r\|_{L^\Phi(\rn)}\r]^q\,dx\r\}^{\frac{1}{q}}.
\end{equation}
Indeed, it is easy to see that
\begin{align*}
\lf[t^{n}\sum_{k\in \mathbb{Z}^n}\lf\|f\chi_{Q_{tk}}\r\|^q_{L^\Phi(\rn)}\r]^{\frac{1}{q}}
=\lf\|\sum_{k\in \mathbb{Z}^n}\lf\|f\chi_{Q_{tk}}\r\|_{L^\Phi(\rn)}\chi_{Q_{tk}}\r\|_{L^q(\rn)}.
\end{align*}
For any $x\in Q_{tk}$, from $Q_{tk}\subset B(x,2\sqrt{n}t)$, it follows that
$$\lf\|f\chi_{Q_{tk}}\r\|_{L^\Phi(\rn)}\chi_{Q_{tk}}(x)\le\lf\|f\chi_{B(x,2\sqrt{n}t)}\r\|_{L^\Phi(\rn)}.$$
Thus, combining the above two formulas, we conclude that

\begin{equation}\label{yyy}
\lf\|\sum_{k\in \mathbb{Z}^n}\lf\|f\chi_{Q_{tk}}\r\|_{L^\Phi(\rn)}\chi_{Q_{tk}}\r\|_{L^q(\rn)}
\le\lf\|\lf\|f\chi_{B(\cdot,2\sqrt{n}t)}\r\|_{L^\Phi(\rn)}\r\|_{L^q(\rn)}.
\end{equation}
To prove the opposite inequality, for any given $x\in \rn$, we let
$$ M_x:=\lf\{k\in\mathbb{Z}^n:\ Q_{tk} \cap B(x,2\sqrt{n}t)\neq\emptyset\r\}.$$
Then the cardinality $\#M_x\lesssim1$ and, if $k\in M_x$, then $x\in B(tk,4\sqrt{n}t)$, which further implies that
\begin{align*}
\lf\|f\chi_{B(x,2\sqrt{n}t)}\r\|_{L^\Phi(\rn)}&=\lf\|\sum_{k\in \mathbb{Z}^n}f
\chi_{B(x,2\sqrt{n}t)}\chi_{Q_{tk}}\r\|_{L^\Phi(\rn)}
\lesssim\sum_{k\in M_x}\lf\|f\chi_{B(x,2\sqrt{n}t)}\chi_{Q_{tk}}\r\|_{L^\Phi(\rn)}\\
&\lesssim\sum_{k\in \mathbb{Z}^n}\lf\|f\chi_{Q_{tk}}\r\|_{L^\Phi(\rn)}\chi_{B(tk,4\sqrt{n}t)}(x).
\end{align*}
Thus, we have
\begin{equation}\label{yy}
\lf\|\lf\|f\chi_{B(\cdot,2\sqrt{n}t)}\r\|_{L^\Phi(\rn)}\r\|_{L^q(\rn)}\lesssim\lf\|\sum_{k\in \mathbb{Z}^n}\lf\|f\chi_{Q_{tk}}\r\|_{L^\Phi(\rn)}\chi_{B(tk,4\sqrt{n}t)}\r\|_{L^q(\rn)}.
\end{equation}
It is easy to see that there exist $N\in\mathbb{N}$ and $\{k_1,\ \ldots,\ k_N\}\subset\mathbb{Z}^n$, independent of $t$, such that $N\lesssim1$ and
$B(\vec{0}_n,4\sqrt{n}t)$ $\subseteq\bigcup_{m=1}^N Q_{tk_m}$ and hence
\begin{equation*}
\sum_{k\in \mathbb{Z}^n}\lf\|f\chi_{Q_{tk}}\r\|_{L^\Phi(\rn)}\chi_{B(tk,4\sqrt{n}t)}
\le\sum_{m=1}^N\sum_{k\in \mathbb{Z}^n}\lf\|f\chi_{Q_{tk}}\r\|_{L^\Phi(\rn)}\chi_{Q_{t(k_m+k)}}.
\end{equation*}
By this, the translation invariance of the Lebesgue measure and $N\lesssim1$, we further obtain
\begin{align*}
\lf\|\sum_{k\in \mathbb{Z}^n}\lf\|f\chi_{Q_{tk}}\r\|_{L^\Phi(\rn)}\chi_{B(tk,4\sqrt{n}t)}\r\|_{L^q(\rn)}
&\leqslant\lf\|\sum_{m=1}^N\sum_{k\in \mathbb{Z}^n}\lf\|f\chi_{Q_{tk}}\r\|_{L^\Phi(\rn)}\chi_{Q_{t(k_m+k)}}\r\|_{L^q(\rn)}\\
&\lesssim\lf[t^{n}\sum_{k\in \mathbb{Z}^n}\lf\|f\chi_{Q_{tk}}\r\|^q_{L^\Phi(\rn)}\r]^{\frac{1}{q}},
\end{align*}
which, together with (\ref{yy}), implies that the opposite inequality of (\ref{yyy}) holds true.
Thus, (\ref{dj}) holds true.

Now, to complete the proof of Proposition \ref{th2}, we only need to show that
$$\lf\|\lf\|f\chi_{B(\cdot,2\sqrt{n}t)}\r\|_{L^\Phi(\rn)}\r\|_{L^q(\rn)}
\backsim\lf\|\lf\|f\chi_{B(\cdot,t)}\r\|_{L^\Phi(\rn)}\r\|_{L^q(\rn)}.$$
Since both $\overline{B}(\vec{0}_n,4\sqrt{n}t)$ and
$\overline{B}(\vec{0}_n,t)$ are compact subsets of $\rn$ with nonempty interiors,
it follows that there exist $M\in\mathbb{N}$ and $\{x_1,\ \ldots,\ x_M\}\subset\rn$, independent of $t$, such that $M\lesssim1$ and $B(\vec{0}_n,4\sqrt{n}t)\subseteq\bigcup_{m=1}^MB(x_m,t).$
Thus, for any $x\in\rn$, we have
\begin{align*}
\lf\|f\chi_{B(x,2\sqrt{n}t)}\r\|_{L^\Phi(\rn)}&=\lf\|f\sum_{m=1}^M\chi_{B(x+x_m,t)}\r\|_{L^\Phi(\rn)}
\lesssim\sum_{m=1}^M\lf\|f\chi_{B(x+x_m,t)}\r\|_{L^\Phi(\rn)}.
\end{align*}
By this, the translation invariance of the Lebesgue measure and $M\lesssim1$, we further obtain $$\lf\|\lf\|f\chi_{B(\cdot,2\sqrt{n}t)}\r\|_{L^\Phi(\rn)}\r\|_{L^q(\rn)}
\lesssim\sum_{m=1}^M\lf\|\lf\|f\chi_{B(x+x_m,t)}\r\|_{L^\Phi(\rn)}\r\|_{L^q(\rn)}
\lesssim\lf\|\lf\|f\chi_{B(\cdot,t)}\r\|_{L^\Phi(\rn)}\r\|_{L^q(\rn)}.$$
The reverse inequality obviously holds true. This finishes the proof of Proposition \ref{th2}.
\end{proof}

Recall that the \emph{centered Hardy-Littlewood maximal operator} $\mathcal{M}$ is defined by setting, for
any locally integrable function $f$ and $x\in\rn$,
$$
\mathcal{M}(f)(x):=\sup_{r\in(0,\infty)}\dashint_{B(x,r)}|f(x)|\,dy,
$$
and the \emph{uncentered Hardy-Littlewood maximal operator $\mathcal{M}_u$} is defined by setting, for
any locally integrable function $f$ and $x\in\rn$,
$$
\mathcal{M}_u(f)(x):=\sup_{x\in B} \dashint_B |f(y)|\,dy,
$$
where the supremum is taken over all balls $B$ of $\rn$ containing $x$.

Borrowing some ideas from the proof of \cite[Lemma 4.1]{ap}, we have the following very useful
 technical lemma, which
plays a vital role in the proof of Theorem \ref{main} below.

\begin{lemma}\label{mm}
Let $t\in(0,\infty),\ r\in(1,\infty)$ and $\Phi$ be an Orlicz function with positive
lower type $p_{\Phi}^-\in(1,\infty)$ and positive upper type $p_{\Phi}^+$.
Then, for any sequence $\{f_j\}_{j\in\mathbb{Z}}$ of
locally integrable functions
and $x\in\rn$, it holds true that
\begin{align*}
\lf\|\lf\{\sum _{j\in\mathbb{Z}}\lf[\mathcal{M}(f_j)\r]^r\r\}^{\frac{1}{r}}\chi_{B(x,t)}\r\|_{L^\Phi(\rn)}
&\le C\lf\|\lf\{\sum_{j\in\mathbb{Z}}\lf|f_j\r|^r\r\}^{\frac{1}{r}}\chi_{B(x,2t)}\r\|_{L^\Phi(\rn)}
\\&\quad+C\|\chi_{B(x,t)}\|_{L^\Phi(\rn)}
\lf\{\sum_{j\in\mathbb{Z}}\lf[\mathcal{M}_u\lf(\dashint_{B(\cdot,t)}\lf|f_j(z)\r|\,dz\r)(x)\r]^r\r\}^{\frac{1}{r}},
\end{align*}
where the positive constant $C$ is independent of $\{f_j\}_{j\in\mathbb{Z}}$,
$x\in\rn$ and $t\in(0,\infty)$.
\end{lemma}

\begin{proof}
Let $\{f_j\}_{j\in\mathbb{Z}}$ be a sequence of locally integrable functions.
Fix $x\in\rn$. Then we have

\begin{align*}
\lf\|\lf\{\sum _{j\in\mathbb{Z}}\lf[\mathcal{M}(f_j)\r]^r\r\}^{\frac{1}{r}}\chi_{B(x,t)}\r\|_{L^\Phi(\rn)}
&\lesssim\lf\|\lf\{\sum _{j\in\mathbb{Z}}\lf[\sup_{s\in(0, t]}\dashint_{B(\cdot,s)}\lf|f_j(z)\r|\,dz\r]^r\r\}^{\frac{1}{r}}\chi_{B(x,t)}\r\|_{L^\Phi(\rn)}\\
&\quad+\lf\|\lf\{\sum _{j\in\mathbb{Z}}\lf[\sup_{s\in(t,\infty)}\dashint_{B(\cdot,s)}\lf|f_j(z)\r|\,dz\r]^r
\r\}^{\frac{1}{r}}\chi_{B(x,t)}\r\|_{L^\Phi(\rn)}\\
&=:\mathrm{I}+\mathrm{II}.
\end{align*}

Since $B(y,s)\subset B(x,2t)$ whenever $s\in(0,t]$ and $y\in B(x,t)$, it follows that
\begin{align*}
\mathrm{I}&\sim\lf\|\chi_{B(x,t)}\lf\{\sum _{j\in\mathbb{Z}}\lf[\sup_{s\in(0, t]}\dashint_{B(\cdot,s)}\lf|f_j(z)\r|\chi_{B(x,2t)}(z)\,dz\r]^r\r\}^{\frac{1}{r}}\r\|_{L^\Phi(\rn)}\\
&\lesssim\lf\|\lf\{\sum _{j\in\mathbb{Z}}\lf[\mathcal{M}\lf(f_j\chi_{B(x,2t)}\r)\r]^r\r\}^{\frac{1}{r}}\r\|_{L^\Phi(\rn)}
\lesssim\lf\|\lf\{\sum _{j\in\mathbb{Z}}\lf|f_j\r|^r\r\}^{\frac{1}{r}}\chi_{B(x,2t)}\r\|_{L^\Phi(\rn)},
\end{align*}
where, in the last inequality, we used the Orlicz Fefferman-Stein vector-valued inequality
(see, for example, \cite{kk} or \cite[Theorem 2.1.4]{yll}).

As for $\mathrm{II}$, observe that, for any $\xi,\ z\in\rn$, $\xi\in B(z,t)$ if and only if $z\in B(\xi,t)$
and, moreover, if $z\in B(y,s)$, $\xi\in B(z,t)$ and $s\in(t,\infty)$, then $\xi\in B(y,2s)$. Besides,
observe that  $y\in B(x,t)$ and $s\in(t,\infty)$ imply that $x\in B(y,2s).$
From these observations and the Fubini theorem, we deduce that
\begin{align*}
\mathrm{II}&\sim\lf\|\chi_{B(x,t)}\lf\{\sum_{j\in\mathbb{Z}}\lf[\sup_{s\in(t,\infty)}
\dashint_{B(\cdot,s)}\dashint_{B(z,t)}\lf|f_j(z)\r|
\,d\xi \,dz\r]^r\r\}^{\frac{1}{r}}\r\|_{L^\Phi(\rn)}\\
&\lesssim\lf\|\chi_{B(x,t)}\lf\{\sum_{j\in\mathbb{Z}}\lf[\sup_{s\in(t,\infty)}
\dashint_{B(\cdot,2s)}\dashint_{B(\xi,t)}\lf|f_j(z)\r|\,dz\,d\xi\r]^r\r\}^{\frac{1}{r}}\r\|_{L^\Phi(\rn)}\\
&\lesssim\lf\|\chi_{B(x,t)}\lf\{\sum_{j\in\mathbb{Z}}
\lf[\mathcal{M}_u\lf(\dashint_{B(\cdot,t)}\lf|f_j(z)\r|\,dz\r)(x)\r]^r\r\}^{\frac{1}{r}}\r\|_{L^\Phi(\rn)}\\
&\lesssim\lf\|\chi_{B(x,t)}\r\|_{L^\Phi(\rn)}\lf\{\sum_{j\in\mathbb{Z}}
\lf[\mathcal{M}_u\lf(\dashint_{B(\cdot,t)}\lf|f_j(z)\r|\,dz\r)(x)\r]^r\r\}^{\frac{1}{r}}.
\end{align*}

Gathering the estimates for $\mathrm{I}$ and $\mathrm{II}$, we then obtain the desired conclusion, which
completes the proof of Lemma \ref{mm}.
\end{proof}

\begin{definition}\label{d1.4}
A convex function $\Phi:\ [0,\infty)\  \rightarrow\  [0,\infty)$ is called a
\emph{Young function} if $\Phi$ is non-decreasing, $\Phi(0)= 0$ and
$\lim_{t\rightarrow\infty}\Phi(t)=\infty$. For any Young function $\Phi$,
its complementary function $\Psi:\ [0,\infty)\  \rightarrow\  [0,\infty)$
is defined by setting, for any $y\in[0,\infty)$
$$\Psi(y):=\sup\lf\{xy-\Phi(x):\ x\in[0,\infty)\r\}.$$
\end{definition}

\begin{definition}\label{ddd}
A Young function $\Phi:\ [0,\infty)\  \rightarrow\  [0,\infty)$ is called an \emph{N-function} if
$\Phi(0)= 0$, $\Phi(t)>0$ for any $t\in(0,\infty)$,
 $\lim_{t\rightarrow\infty}\frac{\Phi(t)}{t}=\infty$ and $\lim_{t\rightarrow0^+}\frac{\Phi(t)}{t}=0$,
 here and hereafter, $t\rightarrow0^+$ means $t\in(0,\infty)$ and $t\rightarrow0$.
\end{definition}

\begin{lemma}\label{odj}
Let $\Phi$ be an Orlicz function with lower type $p_{\Phi}^-\in(1,\infty)$ and positive upper type $p_{\Phi}^+$.
Then there exists an N-function $\widetilde{\Phi}$, which is equivalent to $\Phi$.
\end{lemma}

\begin{proof}
Consider the function
$$
\widetilde{\Phi}(t):=\begin{cases}\displaystyle\int_0^t \sup_{\tau\in(0,s)}\frac{\Phi(\tau)}{\tau}\,ds,\quad\forall\,t \in (0,\infty],\\
\ 0,\quad t=0.
\end{cases}
$$
Then it is easy to prove that $\widetilde{\Phi}$ is convex on $[0,\infty)$. By the
assumption that $p_{\Phi}^-\in[1,\infty)$,
we know that, for any $t\in(0,\infty)$,
$$\widetilde{\Phi}(t)\le t\sup_{\tau\in(0,t)}\frac{\Phi(\tau)}{\tau}\le C_{(p_\Phi^-)}t\sup_{\tau\in(0,t)}\lf(\frac{\tau}{t}\r)^{p_\Phi^-}\frac{\Phi(t)}{\tau}\le C_{(p_\Phi^-)}\Phi(t).$$
On the other hand, for any $t\in(0,\infty)$, we have
$$\Phi(t)\le C_{(p_\Phi^+)}2^{p_\Phi^+}\Phi(t/2)\le C_{(p_\Phi^+)}2^{p_\Phi^+}\int_{t/2}^{t}\sup_{\tau\in(0,s)}\frac{\Phi(\tau)}{\tau}ds
\le C_{(p_\Phi^+)}2^{p_\Phi^+}\widetilde{\Phi}(t).$$
Thus, we obtain $\Phi\sim\widetilde{\Phi}$. Moreover, it is easy to prove that $\widetilde{\Phi}$ is an N-function,
which completes the proof of Lemma \ref{odj}.
\end{proof}

\begin{remark}\label{nfunction}
\begin{enumerate}
\item[\rm{(i)}] Observe that all the results stated in this article are invariant under the change of equivalent
Orlicz functions. In what follows, by Lemma \ref{odj} and its proof, without loss of generality, we may \emph{always assume} that an Orlicz function $\Phi$ of lower type $p_\Phi^-\in(1,\infty)$ is also an N-function
and an Orlicz function $\Phi$ of lower type $p_\Phi^-=1$ is also a Young function.

\item[\rm{(ii)}]   Let $q\in[1,\infty)$ and $\Phi$ be a Young function with lower type $p_{\Phi}^-\in[1,\infty)$ and
positive upper type $p_{\Phi}^+$. We know that
$L^\Phi(\rn)$ is a Banach space (see \cite[p.\,67, Theorem 10]{mmz}). Then it is easy to prove that $(E_\Phi^q)_t(\rn)$
is also a Banach space.
\end{enumerate}
\end{remark}

The following two lemmas come from \cite[p.\,13, Proposition 1(ii); p.\,58, Proposition 1]{mmz}, respectively.

\begin{lemma}\label{un}
Let $\Phi$ be an N-function and $\Psi$ its complementary function. Then $\Phi$ and $\Psi$ are strictly
increasing and hence their inverses $\Phi^{-1}$ and $\Psi^{-1}$ are uniquely defined and, for any $t\in(0,\infty)$,
$$t<\Phi^{-1}(t)\Psi^{-1}(t)<2t.$$
\end{lemma}

\begin{lemma}\label{hold}
Let $\Phi$ be a Young function and $\Psi$ its complementary function. If $f\in L^\Phi(\rn)$ and
$g\in L^\Psi(\rn)$, then
$$\int_{\rn}\lf|f(x)g(x)\r|dx\le2\|f\|_{L^\Phi(\rn)}\|g\|_{L^\Psi(\rn)}.$$
\end{lemma}

The following Fefferman-Stein type inequality for Orlicz-slice spaces extends the well-known Fefferman-Stein
vector-valued maximal inequality \cite[Theorem 1(1)]{fs}
, which plays an important role in the succeeding sections.

\begin{theorem}\label{main}
Let $t\in(0,\infty)$, $q,\ r\in(1,\infty)$ and $\Phi$ be an Orlicz function with lower type $p_{\Phi}^-\in(1,\infty)$ and positive upper type $p_{\Phi}^+$.
Then there exists a positive constant $C$ such that, for any $\{f_j\}_{j\in\mathbb{Z}}\subset (E_\Phi^q)_t(\rn)$,
$$
\left\|\left\{\sum _{j\in\mathbb{Z}}\left[\mathcal{M}(f_j)\right]^r\right\}
^{\frac{1}{r}}\right\|_{(E_\Phi^q)_t(\rn)}
\le C\left\|\left\{\sum _{j\in\mathbb{Z}}|f_j|^r\right\}^{\frac{1}{r}}\right\|_{(E_\Phi^q)_t(\rn)},
$$
where $C$ is independent of $\{f_j\}_{j\in\mathbb{Z}}$ and $t$.
\end{theorem}
\begin{proof}
For any $\{f_j\}_{j\in\mathbb{Z}}\subset (E_\Phi^q)_t(\rn)$, applying Lemma \ref{mm}, we have
\begin{align*}
&\int_{\rn} \lf\{\frac{\|[\sum _{j\in\mathbb{Z}}[\mathcal{M}(f_j)]^r]^{\frac{1}{r}}\chi_{B(x,t)}\|_{L^\Phi(\rn)}}
{\|\chi_{B(x,t)}\|_{L^\Phi(\rn)}}\r\}^q\,dx\\
&\quad\lesssim \int_{\rn}\lf\{\frac{\|\{\sum_{j\in\mathbb{Z}}|f_j|^r\}^{\frac{1}{r}}
\chi_{B(x,2t)}\|_{L^\Phi(\rn)}}{\|\chi_{B(x,t)}\|_{L^\Phi(\rn)}}\r\}^q\, dx
+\int_{\rn}\lf\{\sum_{j\in\mathbb{Z}}\lf[\mathcal{M}_u\lf(\dashint_{B(\cdot,t)}|f_j(z)|dz\r)(x)\r]^r\r\}^{\frac{q}{r}}dx
\\&\quad=:\mathrm{I}+\mathrm{II}.
\end{align*}
Since both $\overline{B}(\vec{0}_n,2t)$ and $\overline{B}(\vec{0}_n,t)$ are compact subsets of $\rn$ with nonempty interiors, it
follows that there exist $N\in\mathbb{N}$ and $\{x_1,\ \ldots,\ x_N\}\subset\rn$, independent of $t$, such that
$N\lesssim1$ and $B(\vec{0}_n,2t)\subseteq\bigcup_{m=1}^NB(x_m,t)$.
Thus, by this, (\ref{qiu}) and the translation invariance of the Lebesgue measure, we conclude that
\begin{align*}
\mathrm{I}&\sim\lf[\widetilde{C}_{(\Phi,t)}\r]^q\int_{\rn}\lf\|\lf\{\sum_{j\in\mathbb{Z}}\lf|f_j\r|^r
\r\}^{\frac{1}{r}}\chi_{B(x,2t)}\r\|_{L^\Phi(\rn)}^q\, dx\\
&\lesssim\lf[\widetilde{C}_{(\Phi,t)}\r]^q\int_{\rn}\lf\|
\lf\{\sum_{j\in\mathbb{Z}}\lf|f_j\r|^r\r\}^{\frac{1}{r}}\sum_{m=1}^N\chi_{B(x+x_m,t)}\r\|_{L^\Phi(\rn)}^q\, dx\\
&\lesssim\lf[\widetilde{C}_{(\Phi,t)}\r]^q\sum_{m=1}^N\int_{\rn}
\lf\|\lf\{\sum_{j\in\mathbb{Z}}\lf|f_j\r|^r\r\}^{\frac{1}{r}}\chi_{B(x+x_m,t)}\r\|_{L^\Phi(\rn)}^q\, dx\\
&\lesssim\lf[\widetilde{C}_{(\Phi,t)}\r]^q\int_{\rn}\lf\|\lf\{\sum_{j\in\mathbb{Z}}
\lf|f_j\r|^r\r\}^{\frac{1}{r}}\chi_{B(x,t)}\r\|_{L^\Phi(\rn)}^q\, dx\sim
\lf\|\lf\{\sum _{j\in\mathbb{Z}}\lf|f_j\r|^r\r\}^{\frac{1}{r}}\r\|_{(E_\Phi^q)_t(\rn)}^q,
\end{align*}
where $\widetilde{C}_{(\Phi,t)}$ is as in \eqref{qiu}.

As for $\mathrm{II}$, by the Fefferman-Stein vector-valued inequality in $L^q(\rn)$ (see \cite{fs}), we have
$$\mathrm{II}
\lesssim\int_{\rn}\lf\{\sum_{j\in\mathbb{Z}}\lf[\dashint_{B(x,t)}|f_j(z)|\,dz\r]^r\r\}^{\frac{q}{r}}dx.$$
Let $r':=\frac{r}{r-1}.$ Then there exists $\{b_j\}_{j\in\mathbb{Z}}\in \ell^{r'}$,
with $\|\{b_j\}_{j\in\mathbb{Z}}\|_{\ell^{r'}}=1,$ such that
$$\int_{\rn}\lf\{\sum_{j\in\mathbb{Z}}\lf[\dashint_{B(x,t)}|f_j(z)|\,dz\r]^r\r\}^{\frac{q}{r}}dx
=\int_{\rn}\lf[\sum_{j\in\mathbb{Z}}b_j\dashint_{B(x,t)}|f_j(z)|\,dz\r]^{q}\,dx.$$
Using Lemma \ref{hold} and the H\"older inequality, we further conclude that
\begin{align*}
&\int_{\rn}\lf[\sum_{j\in\mathbb{Z}}b_j\dashint_{B(x,t)}|f_j(z)|\,dz\r]^{q}\,dx\\
&\quad\lesssim\int_{\rn}\lf\{\dashint_{B(x,t)}\lf[\sum_{j\in\mathbb{Z}}|f_j(z)|^r\r]^{\frac{1}{r}}
\lf(\sum_{j\in\mathbb{Z}}b_j^{r'}\r)^{\frac{1}{r'}}\,dz\r\}^{q}\,dx\\
&\quad\lesssim\int_{\rn}\lf\{\lf\|\lf[\sum_{j\in\mathbb{Z}}\lf|f_j\r|^r\r]^{\frac{1}{r}}
\chi_{B(x,t)}\r\|_{L^\Phi(\rn)}\frac{\|\chi_{B(x,t)}\|_{L^\Psi(\rn)}}{|B(x,t)|}\r\}^q\, dx.
\end{align*}
Applying Lemma \ref{un}, we obtain
\begin{align*}
\frac{\|\chi_{B(x,t)}\|_{L^\Psi(\rn)}}{|B(x,t)|}
&=\frac{1}{|B(x,t)|\Psi^{-1}(\frac{1}{|B(x,t)|})}
=\frac{\Phi^{-1}(\frac{1}{|B(x,t)|})}{|B(x,t)|\Phi^{-1}(\frac{1}{|B(x,t)|})\Psi^{-1}(\frac{1}{|B(x,t)|})}\\
&<\Phi^{-1}\lf(\frac{1}{|B(x,t)|}\r)=\frac{1}{\|\chi_{B(x,t)}\|_{L^\Phi(\rn)}}.
\end{align*}
Thus,
$$\mathrm{II}\lesssim\lf\|\lf\{\sum _{j\in\mathbb{Z}}|f_j|^r\r\}^{\frac{1}{r}}\r\|_{(E_\Phi^q)_t(\rn)}^q,$$
which completes the proof of Theorem \ref{main}.
\end{proof}

\begin{remark}
Let $t\in(0,\infty),$ $q\in(1,\infty)$ and $\Phi(\tau):=\tau^q$ for any $\tau\in[0,\infty).$ Then,
by Remark \ref{re2.10}(ii) and Proposition \ref{ggg}(iii),
we know that $(E_\Phi^q)_t(\rn)=L^q(\rn)$ and, in this case,
Theorem \ref{main} is just the well-known Fefferman-Stein vector-valued maximal inequality
\cite[Theorem 1(1)]{fs}.
\end{remark}

As an immediate consequence of Theorem \ref{main},
we have the following boundedness of Hardy-Littlewood maximal operators on Orlicz-slice spaces.

\begin{proposition}\label{main2}
Let $t\in(0,\infty)$, $q\in(1,\infty)$ and $\Phi$ be an Orlicz function with positive lower type $p_{\Phi}^-\in(1,\infty)$ and positive upper type $p_{\Phi}^+$.
Then the central Hardy-Littlewood maximal function $\mathcal{M}$ is bounded on
the Orlicz-slice space $(E_\Phi^q)_t(\rn)$ with the operator norm independent of $t$.
\end{proposition}

\begin{remark}
Let $t\in(0,\infty)$ and $q,\ r\in(1,\infty).$ Recall that Auscher and Prisuelos-Arribas
\cite[Proposition 8.3(a)]{ap}
obtained the boundedness of the Hardy-Littlewood maximal operator $\mathcal{M}$ on the space
$(E_r^q)_t(\rn).$ It is easy to see that, if $\Phi(\tau):=\tau^r$ for any $\tau\in[0,\infty)$,
then $(E_\Phi^q)_t(\rn)=(E_r^q)_t(\rn)$ and, in this case, Proposition \ref{main2} is
just \cite[Proposition 8.3(a)]{ap}. Thus, Proposition \ref{main2} essentially generalizes
\cite[Proposition 8.3(a)]{ap}.
\end{remark}

\begin{definition}
Let $q\in(1,\infty)$ and $\{E_k\}_{k\in\mathbb{N}}$ be a sequence of Banach spaces.
The \emph{amalgam space $\ell^q(\{E_k\}_{k\in\mathbb{N}})$} is defined to be
set of all sequences $x: = \{x_k\}_{k\in\mathbb{N}}$
such that
$$
\|x\|_{\ell^q(\{E_k\}_{k\in\mathbb{N}})}:=\lf[\sum_{k=1}^\infty\|x_k\|_{E_k}^q\r]^{\frac{1}{q}}<\infty.
$$
\end{definition}

The following  lemma comes from \cite[p.\,359]{kg}.

\begin{lemma}\label{dou}
Let $q\in(1,\infty)$. Then the space $\ell^q(\{E_k\}_{k\in\mathbb{N}})$ is a Banach space
and its dual space is $\ell^{q'}(\{(E_k)^*\}_{k\in\mathbb{N}})$, where
$\frac{1}{q}+\frac{1}{q'}=1$ and $(E_k)^*$ denotes the dual space of $E_k$.
\end{lemma}

\begin{theorem}\label{dual}
Let $t\in(0,\infty)$, $q\in(1,\infty)$ and $\Phi$ be an Orlicz function with lower type $p_{\Phi}^-\in(1,\infty)$
and positive upper type $p_{\Phi}^+$.
 Let $\Psi$ be the complementary function of $\Phi$.
Then the dual space of $(E_\Phi^q)_t(\rn)$ is isomorphic and homeomorphic to $(E_\Psi^{q'})_t(\rn)$.
\end{theorem}

\begin{proof}
Let $t\in(0,\infty)$. Using Proposition \ref{th2}, we obtain
$$(E_\Phi^q)_t(\rn)=\ell^q(\{L^\Phi(Q_{tk})\}_{k\in\mathbb{Z}^n})$$
and, for any $f\in(E_\Phi^q)_t(\rn)$,
$$
\|f\|_{(E_\Phi^q)_t(\rn)}\sim\frac{t^{\frac{n}{q}}}{[\Phi^{-1}(\frac{1}{\varepsilon_nt^n})]^{-1}}
\|f\|_{\ell^q(\{L^\Phi(Q_{tk})\}_{k\in\mathbb{Z}^n})}.
$$
Then it is easy to prove that $((E_\Phi^q)_t(\rn))^*=(\ell^q(\{L^\Phi(Q_{tk})\}_{k\in\mathbb{Z}^n}))^*$ and, for any $f\in((E_\Phi^q)_t(\rn))^*$,
\begin{equation}\label{11}
\frac{t^{\frac{n}{q}}}{[\Phi^{-1}(\frac{1}{\varepsilon_nt^n})]^{-1}}
\|f\|_{((E_\Phi^q)_t(\rn))^*}\sim\|f\|_{(\ell^q(\{L^\Phi(Q_{tk})\}_{k\in\mathbb{Z}^n}))^*},
\end{equation}
where the equivalent positive constants are independent of $t$ and $f$.
Applying \cite[p.\,110, Theorem 7]{mmz} or \cite[Theorem 8.19]{af}, we know that $(L^\Phi(Q_{tk}))^*$ is
isomorphic and homeomorphic to $L^\Psi(Q_{tk})$, which, together with Lemma \ref{dou}, implies that
$(\ell^q(\{L^\Phi(Q_{tk})\}_{k\in\mathbb{Z}^n}))^*
=\ell^{q'}(\{L^\Psi(Q_{tk})\}_{k\in\mathbb{Z}^n})$ and, for any
$f\in(\ell^q(\{L^\Phi(Q_{tk})\}_{k\in\mathbb{Z}^n}))^*$,
\begin{equation}\label{22}
\|f\|_{(\ell^q(\{L^\Phi(Q_{tk})\}_{k\in\mathbb{Z}^n}))^*}
\sim\|f\|_{\ell^{q'}(\{L^\Psi(Q_{tk})\}_{k\in\mathbb{Z}^n})},
\end{equation}
where the equivalent positive constants are independent of $f$ and $t$.
Using Proposition \ref{th2} again, we find that, for any $f\in \ell^{q'}(\{L^\Psi(Q_{tk})\}_{k\in\mathbb{Z}^n})$,
$$
\|f\|_{\ell^{q'}(\{L^\Psi(Q_{tk})\}_{k\in\mathbb{Z}^n})}\sim\frac{[\Psi^{-1}
(\frac{1}{\varepsilon_nt^n})]^{-1}}{t^{\frac{n}{q'}}}
\|f\|_{(E_\Psi^{q'})_t(\rn)},
$$
which, combined with (\ref{11}) and (\ref{22}), implies that $((E_\Phi^q)_t(\rn))^*=(E_\Psi^{q'})_t(\rn)$ and
$$
\frac{t^n}{[\Phi^{-1}(\frac{1}{\varepsilon_nt^n})]^{-1}[\Psi^{-1}
(\frac{1}{\varepsilon_nt^n})]^{-1}}\|f\|_{((E_\Phi^q)_t)(\rn)^*}
\sim\|f\|_{(E_\Psi^{q'})_t(\rn)},
$$
where all equivalent positive constants are independent of $f$ and $t$.
By Lemma \ref{un}, we have
$$
\frac{t^n}{[\Phi^{-1}(\frac{1}{\varepsilon_nt^n})]^{-1}[\Psi^{-1}
(\frac{1}{\varepsilon_nt^n})]^{-1}}
\sim 1
$$
with the equivalent positive constants independent of $f$, which further implies the desired conclusion and hence
completes the proof of Theorem \ref{dual}.
\end{proof}

Next, we recall the notion of ball quasi-Banach
function spaces defined in \cite[Definition 2.1]{ykds}. In what follows, the
\emph{symbol} $\mathbb{M}(\rn)$ denotes
the set of all measurable functions on $\rn$.

\begin{definition}\label{ball}
A quasi-Banach space $X\subset\mathbb{M}(\rn)$ is called a \emph{ball quasi-Banach function space} on $\rn$
if it satisfies
\begin{enumerate}
\item[{\rm(i)}] $\|f\|_{X}=0$ implies that $f=0$ almost everywhere;
\item[{\rm(ii)}] $|g|\le|f|$ almost everywhere implies that $\|g\|_{X}\le\|f\|_{X}$;
\item[{\rm(iii)}] $0\le f_m\uparrow f$ almost everywhere on $\rn$
implies that $\|f_m\|_{X}\uparrow\|f\|_{X}$;

\item[{\rm(iv)}] $B\in\mathbb{B}$ implies that $\chi_B\in X$, where
$$\mathbb{B}:=\{B(x,r):\ x\in\rn\ \ \mbox{and}\ \ r\in(0,\infty)\}.$$
\end{enumerate}
\end{definition}

Recall that Sawano et al. \cite{ykds} developed a real-variable
theory of Hardy spaces associated with ball quasi-Banach function spaces. Next we show that the
Orlicz-slice spaces are ball quasi-Banach function spaces, which further
implies that the Orlicz-slice Hardy space is a special case of the
Hardy type space considered in \cite{ykds}.

\begin{lemma}\label{ballproof}
Let $t,\ q\in(0,\infty)$ and $\Phi$ be an Orlicz function with positive lower type $p_{\Phi}^-$ and positive upper type $p_{\Phi}^+$.
Then $(E_\Phi^q)_t(\rn)$ is a ball quasi-Banach function space.
\end{lemma}

\begin{proof}
By Remark \ref{re}, without loss of generality,
we may assume that $\Phi$ is continuous and strictly increasing. Then, from the definition
of $\|\cdot\|_{(E_\Phi^q)_t(\rn)}$, it is easy to deduce that $(E_\Phi^q)_t(\rn)$ satisfies
(i), (ii) and (iv) of Definition \ref{ball}.

We now prove that $(E_\Phi^q)_t(\rn)$ satisfies Definition \ref{ball}(iii). To this end, let $\{f_m\}_{m\in\mathbb{N}}\subset(E_\Phi^q)_t(\rn)$ and $f\in(E_\Phi^q)_t(\rn)$
satisfy $0\le f_m\uparrow f$ almost everywhere on $\rn$. For any fixed $x\in\rn$ and $t\in(0,\infty)$, let $A_{(x,t)}\in(0,\|f\chi_{B(x,t)}\|_{L^\Phi(\rn)})$. Then, by the definition of $\|f\chi_{B(x,t)}\|_{L^\Phi(\rn)}$,
we have
$$
\int_{B(x,t)}\Phi\lf(\frac{|f(y)|}{A_{(x,t)}}\r)\,dy>1,
$$
which, together with the monotone convergence theorem, implies that there exists $N\in\mathbb{N}$ such that
$$
\int_{B(x,t)}\Phi\lf(\frac{|f_K(y)|}{A_{(x,t)}}\r)\,dy>1,\quad \forall\,K\ge N.
$$
Thus, when $K\ge N$, $\|f_K\chi_{B(x,t)}\|_{L^\Phi(\rn)}>A_{(x,t)}$, which, together with the arbitrariness of $A_{(x,t)}\in(0,\|f\chi_{B(x,t)}\|_{L^\Phi(\rn)})$, implies that, for any $x\in\rn$,
$\lim_{m\rightarrow\infty}\|f_m\chi_{B(x,t)}\|_{L^\Phi(\rn)}=\|f\chi_{B(x,t)}\|_{L^\Phi(\rn)}$.
Then, by the monotone convergence theorem in $L^q(\rn)$, we obtain
$$
\lim_{m\rightarrow\infty}\|f_m\|_{(E_\Phi^q)_t(\rn)}=\|f\|_{(E_\Phi^q)_t(\rn)}.
$$
This finishes the proof of Lemma \ref{ballproof}.
\end{proof}

\begin{definition}\label{lian}
A ball quasi-Banach function space $X$ is said to have an \emph{absolutely continuous quasi-norm} if
$\|\chi_{E_j}\|_{X}\downarrow0$ as $j\rightarrow\infty$ whenever $\{E_j\}_{j=1}^\infty$ is a sequence of measurable sets
in $\rn$ satisfying that $E_j\supset E_{j+1}$ for any $j\in\mathbb{N}$ and $\cap_{j=1}^\infty E_j=\emptyset$.
\end{definition}

\begin{definition}\label{convex}
Let $X$ be a ball quasi-Banach function space and $p\in(0,\infty)$.
\begin{enumerate}
\item[{\rm(i)}] The \emph{$p$-convexification} $X^p$ of $X$ is defined by setting
$X^p:=\{f\in\mathbb{M}(\rn):\ |f|^p\in X\}$
equipped with the quasi-norm $\|f\|_{X^p}:=\||f|^p\|_{X^p}^{\frac{1}{p}}$ for any $f\in X^p$.

\item[{\rm(ii)}] The space $X$ is said to be \emph{p-convex} if there exists a positive constant $C$ such that,
for any $\{f_j\}_{j\in\mathbb{N}}\subset X^{\frac{1}{p}}$,
$$
\lf\|\sum_{j=1}^{\infty}\lf|f_j\r|\r\|_{X^{\frac{1}{p}}}\le C\sum_{j=1}^{\infty}\lf\|f_j\r\|_{X^{\frac{1}{p}}}.
$$
In particular, when $C=1$, $X$ is said to be \emph{strictly p-convex}.

\end{enumerate}
\end{definition}

\begin{lemma}\label{littlewood}
Let $t$, $q\in(0,\infty)$ and $\Phi$ be an Orlicz function with positive lower type $p_{\Phi}^-$ and positive upper type $p_{\Phi}^+$. Let $r\in(0,\min\{p_\Phi^-,q\})$. Then $\mathcal{M}$ is bounded on $[(E_\Phi^q)_t(\rn)]^{\frac{1}{r}}$ with the operator norm independent of $t$, where $[(E_\Phi^q)_t(\rn)]^{\frac{1}{r}}$ is the $\frac{1}{r}$-convexification of $(E_\Phi^q)_t(\rn)$.
\end{lemma}

\begin{proof}
For any $\tau\in(0,\infty)$, let $\Phi_r(\tau):=\Phi(\sqrt[r]{\tau})$. Then $\Phi_r$ is of upper type
$\frac{p_{\Phi}^+}{r}$ and of lower type $\frac{p_{\Phi}^-}{r}$, and $\frac{p_{\Phi}^-}{r}\in(1,\infty)$.
This implies that, for any $t\in(0,\infty)$, $f\in[(E_\Phi^q)_t(\rn)]^{\frac{1}{r}}$ and $x\in\rn$,
$$
\lf\||f|^{\frac{1}{r}}\chi_{B(x,t)}\r\|_{L^\Phi(\rn)}=\lf\||f|\chi_{B(x,t)}\r\|_{L^{\Phi_r}(\rn)}^{\frac{1}{r}}
$$
and
$$
\lf\|\chi_{B(x,t)}\r\|_{L^\Phi(\rn)}=\lf\|\lf[\chi_{B(x,t)}\r]^{\frac{1}{r}}\r\|_{L^\Phi(\rn)}
=\lf\|\chi_{B(x,t)}\r\|_{L^{\Phi_r}(\rn)}^{\frac{1}{r}}.
$$
Combining this and Definition \ref{convex}, we obtain
\begin{align*}
\lf\|\mathcal{M}(f)\r\|_{[(E_\Phi^q)_t(\rn)]^{\frac{1}{r}}}
&=\lf\{\int_{\rn}\lf[\frac{\||\mathcal{M}(f)|^{\frac{1}{r}}\chi_{B(x,t)}\|
_{L^\Phi(\rn)}}{\|\chi_{B(x,t)}\|_{L^\Phi(\rn)}}\r]^q\,dx\r\}^{\frac{r}{q}}\\
&=\lf\{\int_{\rn}\lf[\frac{\||\mathcal{M}(f)|\chi_{B(x,t)}\|
_{L^{\Phi_r}(\rn)}}{\|\chi_{B(x,t)}\|_{L^{\Phi_r}(\rn)}}\r]^{\frac{q}{r}}\,dx\r\}^{\frac{r}{q}}
=\lf\|\mathcal{M}(f)\r\|_{(E_{\Phi_r}^{q/r})_t(\rn)}.
\end{align*}
Since $\frac{p_{\Phi}^-}{r}\in(1,\infty)$ and $\frac{q}{r}\in(1,\infty)$, from Proposition \ref{main2},
it follows that
$$
\lf\|\mathcal{M}(f)\r\|_{(E_{\Phi_r}^{q/r})_t(\rn)}
\lesssim\lf\|f\r\|_{(E_{\Phi_r}^{q/r})_t(\rn)}
\sim\lf\|f\r\|_{[(E_\Phi^q)_t(\rn)]^{\frac{1}{r}}}.
$$
This finishes the proof of Lemma \ref{littlewood}.
\end{proof}

\section{Orlicz-slice Hardy spaces \label{s3}}

In this section, we introduce the Orlicz-slice Hardy spaces, which are
defined via the radial maximal functions. We then present a series of real-variable
characterizations of these Orlicz-slice Hardy spaces,
including characterizations via grand and non-tangential maximal functions, poisson integrals, atoms and finite atoms, and
Littlewood-Paley functions. A Lebesgue-Hardy type coincidence relation is also established between Orlicz-slice spaces and Orlicz-slice Hardy  spaces.

Let us begin with the following notion of the radial maximal function.

\begin{definition}\label{31}
Let $\varphi\in\mathcal{S}(\rn)$ and $f\in\mathcal{S}'(\rn)$.
The \emph{radial maximal function} $M(f,\varphi)$ is defined by setting
$$M(f,\varphi)(x):=\sup_{s\in(0,\infty)}\lf|(\varphi_s \ast f)(x)\r|,\quad \forall\,x\in\rn.$$
\end{definition}

\begin{definition}\label{dh}
Let $t$, $q\in(0,\infty) $ and $\Phi$ be an Orlicz function with positive lower type $p_{\Phi}^-$ and positive upper type $p_{\Phi}^+$.
Then the \emph{Orlicz-slice Hardy space $(HE_\Phi^q)_t(\rn)$} is defined by setting
$$
(HE_\Phi^q)_t(\rn):=\lf\{f\in\mathcal{S}'(\rn):\ \|f\|_{(HE_\Phi^q)_t(\rn)}:=
\|M(f,\varphi)\|_{(E_\Phi^q)_t(\rn)}<\infty \r\},
$$
where $\varphi\in\mathcal{S}(\rn)$ satisfies
$
\int_{\rn}\varphi(x)\,dx\neq0.
$
In particular, when $\Phi(s):=s^r$ for any $s\in[0,\fz)$ with $r\in(0,\fz)$,
the Hardy-type space $(HE_r^q)_t(\rn):=(HE_{\Phi}^q)_t(\rn)$ is called
the \emph{slice Hardy space}.
\end{definition}

\begin{remark}\label{ReHpq}
\begin{enumerate}
\item[(i)] If $t=1$ and $\Phi(\tau):=\tau^p$
for any $\tau\in [0,\fz)$ with $p\in (0,\fz)$,
then $(HE_\Phi^q)_t(\rn)$ coincides with the Hardy-amalgam space $\mathcal{H}^{p,q}(\rn)$ in \cite{af1}.
\item[(ii)] If $t=q=1$, then $(HE_\Phi^q)_t(\rn)$ coincides with
 the variant of the Orlicz-Hardy space $H_*^\Phi(\rn)$ of Bonami and Feuto \cite{bf}.
\end{enumerate}
\end{remark}

\subsection{Characterizations in terms of various maximal functions\label{s3.1}}

We now present some maximal function characterizations of $(HE_\Phi^q)_t(\rn)$, whose
proofs are given in Section \ref{s4}. Define,
for any $N\in\mathbb{N}$
and $\varphi\in\mathcal{S}(\rn)$,
$$
p_N(\varphi):=\sum_{\alpha\in\mathbb{Z}^n_+,|\alpha|\le N} \sup_{x\in\rn}(1+|x|)^{N+n}|\partial^\alpha\varphi(x)|,
$$
and let $\mathcal{F}_N(\rn):=\{\varphi\in\mathcal{S}(\rn):\ \ p_N(\varphi)\le1\}$. Also recall that
$\mathbb{R}_+^{n+1}:=\rn\times(0,\infty)$.

\begin{definition}\label{jdhs}
Let $\varphi\in\mathcal{S}(\rn)$, $N\in\mathbb{N}$, $a,\ b\in(0,\infty)$ and $f\in\mathcal{S}'(\rn)$.
\begin{enumerate}
\item[{\rm(i)}] The \emph{grand maximal function $M_N(f)$} is defined by setting, for any $x\in\rn$,
$$M_N(f)(x):=\sup\lf\{|\varphi_s \ast f(y)|:\ s\in(0,\infty),\  |x-y|<s,\  \varphi\in\mathcal{F}_N(\rn)\r\};$$

\item[{\rm(ii)}] The \emph{grand radial maximal function $M_N^0(f)$} is defined by setting, for any $x\in\rn$,
$$M_N^0(f)(x):=\sup\lf\{|\varphi_s  \ast f(x)|:\ s\in(0,\infty),\
 \varphi\in\mathcal{F}_N(\rn)\r\};$$

\item[{\rm(iii)}]  The \emph{non-tangential maximal function $M_a^*(f,\varphi)$}, with aperture $a\in(0,\infty)$,
 is defined by setting, for any $x\in\rn$,
$$M_a^*(f,\varphi)(x):=\sup_{s\in(0,\infty)}\lf\{\sup_{y\in\rn,|y-x|<as}\lf|(\varphi_s \ast f)(y)\r|\r\};$$

\item[{\rm(iv)}] The \emph{maximal function $M_b^{**}(f,\varphi)$ of Peetre type} is defined by setting, for any $x\in\rn$,
$$M_b^{**}(f,\varphi)(x):=\sup_{(y,s)\in\mathbb{R}_+^{n+1}}\frac{|(\varphi_s \ast f)(x-y)|}{(1+s^{-1}|y|)^b};$$

\item[{\rm(v)}] The \emph{grand maximal function $M(f,\varphi)$ of Peetre type} is defined by setting, for any $x\in\rn$,
$$M_{b,\ N}^{**}(f)(x):=\sup_{\psi\in\mathcal{F}_N(\rn)}\lf\{\sup_{(y,s)\in\mathbb{R}_+^{n+1}}\frac{|(\psi_s \ast f)(x-y)|}{(1+s^{-1}|y|)^b}\r\}.$$
\end{enumerate}
\end{definition}

It is easy to see that, for any $N\in\mathbb{Z}_+$, there exists a positive constant $C_{(N)}$,
depending on $N$, such that, for any $f\in\mathcal{S}'(\rn)$
and $x\in\rn$,
\begin{equation}\label{zhujida}
M_N^0(f)(x)\le M_N(f)(x)\le C_{(N)}M_N^0(f)(x);
\end{equation}
see \cite[Proposition 3.10]{b}.

Via the above maximal functions, we can characterize $(HE_\Phi^q)_t(\rn)$ as follows.

\begin{theorem}\label{mdj}
Let $t,\ a,\ b,\ q\in(0,\infty)$. Let $\Phi$ be an Orlicz function with positive lower type $p_{\Phi}^-$ and positive upper type $p_{\Phi}^+$. Let $\varphi\in\mathcal{S}(\rn)$ satisfy
$\int_{\rn}\varphi(x)\,dx\neq0.$
\begin{enumerate}
\item[{\rm(i)}] Let $N\ge\lfloor b+1\rfloor$ be an integer. Then, for any $f\in\mathcal{S}'(\rn)$,
it holds true that
$$
\lf\|M(f,\varphi)\r\|_{(E_\Phi^q)_t(\rn)}\lesssim\lf\|M_a^*(f,\varphi)\r\|_{(E_\Phi^q)_t(\rn)}
\lesssim\lf\|M_b^{**}(f,\varphi)\r\|_{(E_\Phi^q)_t(\rn)},
$$
$$
\lf\|M(f,\varphi)\r\|_{(E_\Phi^q)_t(\rn)}\lesssim\lf\|M_N(f)\r\|_{(E_\Phi^q)_t(\rn)}
\lesssim\lf\|M_{\lfloor b+2 \rfloor}(f)\r\|_{(E_\Phi^q)_t(\rn)}\lesssim\lf\|M_b^{**}(f,\varphi)\r\|_{(E_\Phi^q)_t(\rn)}
$$
and
$$
\lf\|M_b^{**}(f,\varphi)\r\|_{(E_\Phi^q)_t(\rn)}\sim\lf\|M_{b,\ N}^{**}(f)\r\|_{(E_\Phi^q)_t(\rn)},
$$
where the implicit positive constants are independent of $f$ and $t$.

\item[{\rm(ii)}] Assume $b\in(\frac{n}{\min\{p_\Phi^-,q\}},\infty)$. Then, for any $f\in\mathcal{S}'(\rn)$,
$$
\lf\|M_{b,\ N}^{**}(f)\r\|_{(E_\Phi^q)_t(\rn)}\lesssim\lf\|M(f,\varphi)\r\|_{(E_\Phi^q)_t(\rn)},
$$
where the implicit positive constant is independent of $f$ and $t$. In particular, when $N\ge\lfloor b+1\rfloor$,
if one of the following quantities
$$
\lf\|M(f,\varphi)\r\|_{(E_\Phi^q)_t(\rn)},\  \lf\|M_a^*(f,\varphi)\r\|_{(E_\Phi^q)_t(\rn)},\ \lf\|M_N(f)\r\|_{(E_\Phi^q)_t(\rn)},
$$
$$
\lf\|M_b^{**}(f,\varphi)\r\|_{(E_\Phi^q)_t(\rn)}\quad and \quad  \lf\|M_{b,\ N}^{**}(f)\r\|_{(E_\Phi^q)_t(\rn)}
$$
is finite, then the others are also finite and mutually equivalent with the implicit
positive constants independent of $f$ and $t$.
\end{enumerate}
\end{theorem}

\subsection{Characterization in terms of Poisson integrals\label{s3.2}}

In this section, we characterize $(HE_\Phi^q)_t(\rn)$ by means of the Poisson integral.

Recall that $f\in\mathcal{S}'(\rn)$ is said to be a \emph{bounded tempered distribution} if, for any
$\varphi\in\mathcal{S}(\rn)$, $\varphi \ast f \in L^{\infty}(\rn)$. Moreover, for any bounded tempered
distribution $f$, the \emph{Poisson semigroup} of $f$ is defined by setting, for any $s\in(0,\infty)$,
$$
P_sf:=e^{-s\sqrt{-\triangle}}f:=\mathcal{F}^{-1}(e^{-s|\cdot|}\mathcal{F}f)
$$
(see, for example, \cite[p.\,89]{em} for the details), where $\mathcal{F}$ denotes the \emph{Fourier transform}.
Recall that $\mathcal{F}f$ is defined by setting, for any $\varphi\in\mathcal{S}(\rn)$,
$\la\mathcal{F}f,\varphi\ra:=\la f,\mathcal{F}\varphi\ra$, where, for any $\xi:=(\xi_1,...,\xi_n)\in\rn$,
$$
\mathcal{F}\varphi(\xi):=(2\pi)^{-n/2}\int_{\rn}\varphi(x)e^{-ix\xi}\,dx
$$
with $x\xi=\sum_{i=1}^{n}x_i\xi_i$ for any $x:=(x_1,...,x_n)\in\rn$; also, $\mathcal{F}^{-1}$
denotes the \emph{inverse Fourier transform} which is defined by setting, for any $f\in\mathcal{S}(\rn)$
[or $\mathcal{S}'(\rn)$] and $\xi\in\rn$, $\mathcal{F}^{-1}f(\xi):=\mathcal{F}f(-\xi)$.
Then we have the following characterization of
$(HE_\Phi^q)_t(\rn)$.

\begin{theorem}\label{poisson}
Let $t$, $q\in(0,\infty) $ and $\Phi$ be an Orlicz function with positive lower type $p_{\Phi}^-$ and positive upper type $p_{\Phi}^+$. Assume that $f\in\mathcal{S}'(\rn)$. Then $f\in (HE_\Phi^q)_t(\rn)$ if and
only if $f$ is a bounded tempered distribution and $\sup_{s\in(0,\infty)}|P_s\ast f|\in (E_\Phi^q)_t(\rn)$.
\end{theorem}

\begin{remark}\label{remax}
Let $t,\ q\in(0,\infty)$ and $\Phi(\tau):=\tau^q$ for any $\tau\in[0,\infty).$ Then,
by Remark \ref{re2.10}(ii) and Proposition \ref{ggg}(iii), we know that
$(E_\Phi^q)_t(\rn)=L^q(\rn)$ and, in this case, $(HE_\Phi^q)_t(\rn)=H^q(\rn),$
where $H^q(\rn)$ denotes the classical Hardy space, and
Theorems \ref{mdj} and \ref{poisson} coincide with the well-known
results on $H^q(\rn)$ (see, for example, \cite{l} or \cite[p.\,60, Theorem 1]{g}).
\end{remark}

\subsection{Relations between $(E_\Phi^q)_t(\rn)$ and $(HE_\Phi^q)_t(\rn)$\label{s3.3}}

In this section, we discuss the relation between the spaces $(E_\Phi^q)_t(\rn)$ and $(HE_\Phi^q)_t(\rn)$. More precisely,
we generalize the classical result that $H^p(\rn)=L^p(\rn)$ with $p\in(1,\infty)$ as follows.

\begin{theorem}\label{dayu1}
Let $t\in(0,\infty)$, $q\in(1,\infty) $ and $\Phi$ be an Orlicz function with lower type $p_{\Phi}^-\in(1,\infty)$ and positive upper type $p_{\Phi}^+$.
\begin{enumerate}
\item[{\rm(i)}] It holds true that $(E_\Phi^q)_t(\rn)\ \hookrightarrow\ \mathcal{S}'(\rn)$.

\item[{\rm(ii)}] If $f\in (E_\Phi^q)_t(\rn)$, then $f\in (HE_\Phi^q)_t(\rn)$.

\item[{\rm(iii)}] If $f\in {(HE_\Phi^q)_t}(\rn)$, then there exists a locally
integrable function $g\in (E_\Phi^q)_t(\rn)$ such that $g$ represents $f$, which means that $f=g$ in $\mathcal{S}'(\rn)$ and
$\|f\|_{(HE_\Phi^q)_t(\rn)}=\|g\|_{(HE_\Phi^q)_t(\rn)}$.
\end{enumerate}
\end{theorem}

\subsection{Atomic and molecular characterizations\label{s3.4}}

In this section, we present the atomic and the molecular characterizations of $(HE_\Phi^q)_t(\rn)$. In what follows, for any $L\in\mathbb{Z}_+$, the \emph{symbol} $\mathcal{P}_L(\rn)$ denotes the set of all polynomials on $\rn$ of degree
not greater than $L$. For any $a\in L^1(\rn)$ satisfying
$$
\int_{\rn}(1+|x|)^L|a(x)|\,dx<\infty,
$$
we write $a\bot\mathcal{P}_L(\rn)$ if
$$
\int_{\rn}a(x)x^\alpha\,dx=0
$$
for any $\alpha\in\mathbb{Z}_+^n$ with $|\alpha|\le L$.

\begin{definition}\label{deatom}
Let $t,\ q\in(0,\infty)$, $r\in[1,\infty]$ and $d\in\mathbb{Z}_+$.
Let $\Phi$ be an Orlicz function with positive lower type $p_{\Phi}^-$ and positive upper type $p_{\Phi}^+$.
The function $a$ is called an \emph{$((E_\Phi^q)_t(\rn),\ r,\ d)$-atom}
if there exists a cube $Q\in\mathcal{Q}$ such that $\supp(a)\subset Q$,
$$
\|a\|_{L^r(\rn)}\le\frac{|Q|^{\frac{1}{r }}}{\|\chi_Q\|_{(E_\Phi^q)_t(\rn)}}
$$
and $a\bot\mathcal{P}_d(\rn)$.
\end{definition}

\begin{definition}\label{atomic hardy}
Let $t,\ q\in(0,\infty)$ and $\Phi$ be an Orlicz function with positive lower type
$p_{\Phi}^-$ and positive upper type $p_{\Phi}^+$. Let $r\in(\max\{1,q,\ p_{\Phi}^+\},\infty]$,
$s \in(0,\min\{p^-_{\Phi},q,1\})$ and $d\in\mathbb{Z}_+$ satisfying $d\ge\lfloor n(\frac{1}{s}-1)\rfloor$.
The \emph{atomic Orlicz-slice Hardy space $(HE_\Phi^q)_t^{r,d}(\rn)$}
 is defined to be the set of all $f \in\mathcal{S}'(\rn)$ satisfying
that there exist a sequence $\{a_j\}_{j=1}^\infty$ of $((E_\Phi^q)_t(\rn),\ r,\ d)$-atoms supported,
respectively, on the cubes
$\{Q_j\}_{j=1}^\infty\subset\mathcal{Q}$ and a sequence $\{\lambda_j\}_{j=1}^\infty\subset[0,\infty)$
such that
\begin{equation}\label{33}
f=\sum_{j=1}^{\infty}\lambda_j a_j\ \ \mathrm{in}\ \mathcal{S}'(\rn)
\end{equation}
and
$$
\|f\|_{(HE_\Phi^q)_t^{r,d}(\rn)}:=\inf
\lf\|\lf\{\sum_{j=1}^{\infty} \lf[\frac{\lambda_j}
{\|\chi_{Q_j}\|_{(E_\Phi^q)_t(\rn)}} \r]^s\chi_{Q_j} \r\}
^{\frac{1}{s}}\r\|_{(E_\Phi^q)_t(\rn)}<\infty,
$$
where the infimum is taken over all decompositions of $f$ as above.
\end{definition}

We have the following atomic characterization of $(HE_\Phi^q)_t(\rn)$.

\begin{theorem}\label{atom ch}
Let all assumptions be as in Definition \ref{atomic hardy}.
Then $(HE_\Phi^q)_t(\rn)=(HE_\Phi^q)_t^{r,d}(\rn)$
with equivalent quasi-norms.
\end{theorem}

\begin{definition}\label{demole}
Let $t,\ q\in(0,\infty)$ and $\Phi$ be an Orlicz function with positive lower type $p_{\Phi}^-$ and positive upper type $p_{\Phi}^+$. Let
$r\in[1,\infty]$, $d\in\mathbb{Z}_+$ and $\tau\in(0,\infty)$.
A measurable function $m$ on $\rn$ is called an $((E_\Phi^q)_t(\rn),\ r,\ d,\ \tau)$-\emph{molecule} centered at a cube $Q\in\mathcal{Q}$ if, for any $j\in\mathbb{Z}_+$,
$$
\lf\|\chi_{S_j(Q)}m\r\|_{L^r(\rn)}\le2^{-\tau j}\frac{|Q|^{\frac{1}{r}}}{\|\chi_Q\|_{(E_\Phi^q)_t(\rn)}}
$$
and $a\bot\mathcal{P}_d(\rn)$. In analogy, one defines an $((E_\Phi^q)_t(\rn),\ r,\ d,\ \tau)$-molecule centered at a ball $B$.
\end{definition}

\begin{theorem}\label{molecular}
Let $t,\ q\in(0,\infty)$ and $\Phi$ be an Orlicz function with positive lower type $p_{\Phi}^-$ and positive upper type $p_{\Phi}^+$. Let $r\in(\max\{1,q,\ p_{\Phi}^+\},\infty]$ and $s \in(0,\min\{p^-_{\Phi},q,1\})$. Assume that $d\in\mathbb{Z}_+$ satisfies $d\ge\lfloor n(\frac{1}{s}-1)\rfloor$ and $\tau\in(0,\infty)$ satisfies
$\tau>n(\frac{1}{s}-\frac{1}{r}).$
Then $f\in (HE_\Phi^q)_t(\rn)$ if and only if there exist a sequence $\{m_j\}_{j=1}^\infty$
of $((E_\Phi^q)_t(\rn),\ r,\ d,\ \tau)$-molecules centered, respectively,
at the cubes $\{Q_j\}_{j=1}^\infty\subset\mathcal{Q}$
and $\{\lambda_j\}_{j=1}^\infty\subset[0,\infty)$ satisfying
$$
\lf\|\lf\{\sum_{j=1}^{\infty} \lf[\frac{\lambda_j}{\|\chi_{Q_j}\|_{(E_\Phi^q)_t(\rn)}}\r]^s\chi_{Q_j} \r\}
^{\frac{1}{s}}\r\|_{(E_\Phi^q)_t(\rn)}<\infty
$$
such that
$$
f=\sum_{j=1}^{\infty}\lambda_j m_j\ \ in\ \mathcal{S}'(\rn).
$$
Moreover,
$$
\|f\|_{(HE_\Phi^q)_t(\rn)}\sim\inf\lf\|\lf\{\sum_{j=1}^{\infty} \lf[\frac{\lambda_j}
{\|\chi_{Q_j}\|_{(E_\Phi^q)_t(\rn)}} \r]^s\chi_{Q_j} \r\}
^{\frac{1}{s}}\r\|_{(E_\Phi^q)_t(\rn)},
$$
where the infimum is taken over all decompositions of $f$ as above
and the equivalent positive constants are independent of $f$ and $t$.
\end{theorem}

\begin{remark}\label{reatom}
\begin{itemize}
\item[(i)] Let $t,\ q$ and $\Phi$ be as in Remark \ref{remax}.
In this case, we have $(E_\Phi^q)_t(\rn)=L^q(\rn)$ and $(HE_\Phi^q)_t(\rn)=H^q(\rn)$
and, for any $\tau\in(0,\infty),\ r\in[1,\infty]$ and $d\in\zz_+,$
any $((E_\Phi^q)_t(\rn),r,d)$-atom from Definition \ref{deatom}
and any $((E_\Phi^q)_t(\rn),\ r,\ d,\ \tau)$-molecule from
Definition \ref{demole}
just become, respectively, a well-known classical atom (see, for example, \cite[Definition 1.1]{l}
or \cite[p.\,112]{em}) and a well-known classical molecule (see, for example, \cite[Definition 1.2]{hyz}
with $X=\rn$).

\item[(ii)] Let $t,\ q$ and $\Phi$ be as in Remark \ref{remax}.
In this case, when $r\in[1,\infty]\cap(q,\fz]$ and $s=q$,
then Theorem \ref{atom ch} coincides with the classical
atomic characterization of $H^q(\rn)$ (see, for example, \cite[p.\,34, Theorem 3.1]{l}
and \cite[p.\,107, Theorem 2]{em}) and Theorem \ref{molecular} with the classical
molecular characterization of $H^q(\rn)$ (see, for example, \cite[Theorem 2.2]{hyz}
with $X=\rn$). However, it is still unclear whether or not
both Theorems \ref{atom ch} and \ref{molecular} still hold true when
$r=\max\{1,q,\ p_{\Phi}^+\}$ and $s=\min\{p^-_{\Phi},q,1\}$.

Observe that the atomic and the molecular characterizations obtained,
respectively, in Theorems \ref{atom ch} and \ref{molecular}
are more close, in spirit, to the atomic characterization (\cite[Theorem 4.6]{ns})
and the molecular characterization (\cite[Theorem 5.2]{ns}) of variable Hardy spaces,
respectively.
\end{itemize}
\end{remark}

As a corollary of the above theorems, we have the following conclusion.

\begin{proposition}\label{shoulian}
Let all the assumptions be as in Definition \ref{atomic hardy}. Then
\begin{enumerate}
\item[{\rm(i)}] $(HE_\Phi^q)_t(\rn)\cap L^\infty(\rn)$ is dense in $(HE_\Phi^q)_t(\rn)$.

\item[{\rm(ii)}] The summations in \eqref{33} converge in $(HE_\Phi^q)_t(\rn)$.
\end{enumerate}
\end{proposition}

\subsection{Characterizations in terms of Littlewood-Paley functions\label{s3.5}}

In this section, we establish various Littlewood-Paley function characterizations of $(HE_\Phi^q)_t(\rn)$.

\begin{definition}\label{cone}
For any $x\in\rn$, let $\Gamma(x):=\{(y,s)\in\mathbb{R}_+^{n+1}:\ |x-y|<s\}$, which is
called the \emph{cone} of aperture $1$
with vertex $x\in\rn$.
\end{definition}

For any $\tau\in(0,\infty)$, $f\in\mathcal{S}'(\rn)$
and $\varphi\in\mathcal{S}(\rn)$, let
$$
\varphi(\tau D)(f):=\mathcal{F}^{-1}[\varphi(\tau\cdot)\mathcal{F}f].
$$
Recall that a distribution $f\in\mathcal{S}'(\rn)$ is said to be \emph{vanish weakly at infinity} if
$\lim_{t\downarrow0}\varphi(tD)(f)=0$ in $\mathcal{S}'(\rn)$ for any $\varphi\in\mathcal{S}(\rn)$.

Let $\varphi\in\mathcal{S}(\rn)$ be such that
$$
\chi_{B(\vec{0}_n,4) \setminus B(\vec{0}_n,2)}\le\varphi\le\chi_{B(\vec{0}_n,8) \setminus B(\vec{0}_n,1)}.
$$
For any $f\in\mathcal{S}'(\rn)$, the Littlewood-Paley $g$-function $g(f)$, the \emph{Lusin area function}
$S(f)$ and the \emph{Littlewood-Paley $g_\lambda^*$-function} $g_\lambda^*(f)$,
with $\lambda\in(1,\infty)$,
of $f$ are defined, respectively, by setting, for any $x\in\rn$,
$$
g(f)(x):=\lf\{\int_0^{\infty}|\varphi(\tau D)(f)(x)|^{2}\,\frac{d\tau}{\tau}\r\}^{
\frac{1}{2}},
$$
$$
S(f)(x):=\lf\{\int_{\Gamma(x)}|\varphi(\tau D)(f)(y)|^{2}\,\frac{dy\,d\tau}{\tau^{n+1}}\r\}^{
\frac{1}{2}}
$$
and
$$
g_\lambda^*(f)(x):=\lf\{\int_0^{\infty}\int_{\rn}\lf(\frac{\tau}{\tau+|x-y|}\r)^{\lambda n}|\varphi(\tau D)(f)(y)|^{2}\,\frac{dy\,d\tau}{\tau^{n+1}}\r\}^{
\frac{1}{2}}.
$$

Using these functions, we have the following characterizations.

\begin{theorem}\label{lusin}(Lusin area function characterization)
Let $t,\ q\in(0,\infty)$ and $\Phi$ be an Orlicz function with positive lower type $p_{\Phi}^-$ and positive upper type $p_{\Phi}^+$.
Then $f\in (HE_\Phi^q)_t(\rn)$ if and only if $f$ vanishes weakly at infinity and
$$
\lf\|S(f)\r\|_{(E_\Phi^q)_t(\rn)}<\infty.
$$
Moreover,
$$
\|f\|_{(HE_\Phi^q)_t(\rn)}\sim\lf\|S(f)\r\|_{(E_\Phi^q)_t(\rn)},
$$
where the equivalent positive constants are independent of $f$ and $t$.
\end{theorem}

\begin{theorem}\label{gfunction}(Littlewood-Paley g-function characterization)
Let $t,\ q\in(0,\infty)$ and $\Phi$ be an Orlicz function with positive lower type $p_{\Phi}^-$ and positive upper type $p_{\Phi}^+$.
Then $f\in (HE_\Phi^q)_t(\rn)$ if and only if $f$ vanishes weakly at infinity and
$$
\lf\|g(f)\r\|_{(E_\Phi^q)_t(\rn)}<\infty.
$$
Moreover,
$$
\|f\|_{(HE_\Phi^q)_t(\rn)}\sim
\lf\|g(f)\r\|_{(E_\Phi^q)_t(\rn)},
$$
where the equivalent positive constants are independent of $f$ and $t$.
\end{theorem}

\begin{theorem}\label{glamda}(Littlewood-Paley $g_\lambda^*$-function characterization)
Let $t,\ q\in(0,\infty)$ and $\Phi$ be an Orlicz function with positive lower type $p_{\Phi}^-$ and positive upper type $p_{\Phi}^+$. Let $\lambda\in (1+\frac{2}{\min\{p_{\Phi}^-,\ q\}},\infty)$.
Then $f\in (HE_\Phi^q)_t(\rn)$ if and only if $f$ vanishes weakly at infinity and
$$
\lf\|g_\lambda^*(f)\r\|_{(E_\Phi^q)_t(\rn)}<\infty.
$$
Moreover,
$$
\|f\|_{(HE_\Phi^q)_t(\rn)}\sim\lf\|g_\lambda^*(f)\r\|_{(E_\Phi^q)_t(\rn)},
$$
where the equivalent positive constants are independent of $f$ and $t$.
\end{theorem}

\begin{remark}
Let $t,\ q$ and $\Phi$ be as in Remark \ref{remax}. In this case,
we have $p^{-}_\Phi=q$, $(E_\Phi^q)_t(\rn)=L^q(\rn)$
and $(HE_\Phi^q)_t(\rn)=H^q(\rn)$ and the best known range of
$\lambda$ in Theorem \ref{glamda} is $\lambda\in(2/q,\infty)$
(see, for example, \cite[Corollary 7.4]{fs82}). However,
it is still unclear whether or not Theorem \ref{glamda}
still holds true when $\lambda\in(\frac{2}{\min\{p^{-}_\Phi,q\}},1+\frac{2}{\min\{p^{-}_\Phi,q\}}]$.
\end{remark}

\subsection{Finite atomic characterizations\label{s3.6}}

In this section, we establish a finite atomic decomposition theorem
on $(HE_\Phi^q)_t(\rn)$.

\begin{definition}\label{definite}
Let $t,\ q\in(0,\infty)$ and $\Phi$ be an Orlicz function with positive lower type $p_{\Phi}^-$ and positive upper type $p_{\Phi}^+$. Let $r\in(\max\{1,q,\ p_{\Phi}^+\},\infty]$, $s \in(0,\min\{p^-_{\Phi},q,1\})$ and $d\in\mathbb{Z}_+$ satisfying
$d\ge\lfloor n(\frac{1}{s}-1)\rfloor$. The \emph{finite atomic Orlicz-slice Hardy space} $(HE_\Phi^{q,r,d})_t^{\fin}(\rn)$ is define to be the set of all finite linear combinations of $((E_\Phi^q)_t(\rn),\ r,\ d)$-atoms.
The quasi-norm $\|\cdot\|_{(HE_\Phi^{q,r,d})_t^{\fin}(\rn)}$ in $(HE_\Phi^{q,r,d})_t^{\fin}(\rn)$ is defined by
setting, for any $f\in(HE_\Phi^{q,r,d})_t^{\fin}(\rn)$
\begin{align*}
&\|f\|_{(HE_\Phi^{q,r,d})_t^{\fin}(\rn)}\\
&\quad :=\inf\lf\{\lf\|\lf\{\sum_{j=1}^{m} \lf[\frac{\lambda_j}
{\|\chi_{Q_j}\|_{(E_\Phi^q)_t(\rn)}} \r]^s\chi_{Q_j} \r\}
^{\frac{1}{s}}\r\|_{(E_\Phi^q)_t(\rn)}:\ m\in\mathbb{N},\ f=\sum_{j=1}^{m}\lambda_j a_j,\ \{\lambda_j\}_{j=1}^m\subset[0,\infty)\r\},
\end{align*}
where the infimum is taken over all finite linear combinations of $f$
via $((E_\Phi^q)_t(\rn),\ r,\ d)$-atoms $\{a_j\}_{j=1}^m$ supported, respectively,
on cubes $\{Q_j\}_{j=1}^m$.
\end{definition}

Then we have the following conclusion. In what follows, the \emph{symbol} $\mathcal{C}{(\rn)}$ is defined to be the
set of all continuous complex-valued functions on $\rn.$

\begin{theorem}\label{finite}
Let $t,\ q\in(0,\infty)$ and $\Phi$ be an Orlicz function with positive lower type $p_{\Phi}^-$ and positive upper type $p_{\Phi}^+$. Let $r\in(\max\{1,q,\ p_{\Phi}^+\},\infty]$, $s\in(0,\min\{p^-_{\Phi},q,1\})$ and $d\in\mathbb{Z}_+$ satisfying
$d\ge\lfloor n(\frac{1}{s}-1)\rfloor$.
\begin{enumerate}
\item[{\rm(i)}] If $r\in(\max\{1,q,\ p_{\Phi}^+\},\infty)$, then $\|\cdot\|_{(HE_\Phi^q)_t(\rn)}$ and $\|\cdot\|_{(HE_\Phi^{q,r,d})_t^{\fin}(\rn)}$ are equivalent on the space $(HE_\Phi^{q,r,d})_t^{\fin}(\rn)$ with the equivalent positive constant independent of $t$.
\item[{\rm(ii)}] If $r=\fz$, then $\|\cdot\|_{(HE_\Phi^q)_t(\rn)}$ and $\|\cdot\|_{(HE_\Phi^{q,\infty,d})_t^{\fin}(\rn)}$ are equivalent on $(HE_\Phi^{q,\infty,d})_t^{\fin}(\rn)\cap\mathcal{C}{(\rn)}$ with the equivalent positive constant independent of $t.$
\end{enumerate}
\end{theorem}

\begin{remark}
Let $t,\ q,\ r,\ d$ and $\Phi$ be as in Remark \ref{reatom}(ii).
In this case, when $r\in[1,\infty]\cap(q,\fz]$ and $s=q$,
Theorem \ref{finite} coincides with the classical
finite atomic decomposition theorem of $H^q(\rn)$ (see, for example,
\cite[Theorem 3.1, Remark 3.3]{msv} and \cite[Theorem 5.6]{gly} with $X=\rn$). However, it is still unclear whether or not
Theorem \ref{finite} still holds true when $r=\max\{1,q,\ p_{\Phi}^+\}$ and $s=\min\{p^-_{\Phi},q,1\}$.
\end{remark}

\subsection{Further remarks\label{s3.7}}

Amalgam spaces were first introduced by N. Wiener in 1926. In general, for any $t,\ p,\ q\in(0,\infty)$,
the \emph{amalgam space $\ell^q(L^p_t)(\rn)=(L^p_t,\ell^q)(\rn)$} is defined by setting
$$
\ell^q(L^p_t)(\rn):=\lf\{f\ \mbox{measurable}:\ \|f\|_{\ell^q(L^p_t)(\rn)}:=\lf[\sum_{k\in \mathbb{Z}^n}\lf\|f\chi_{Q_{tk}}\r\|^q_{L^p(\rn)}\r]^{\frac{1}{q}}<\infty\r\}.
$$
It is easy to see that the amalgam space $\ell^q(L^p_t)(\rn)$
is a special case of the Orlicz-amalgam space $\ell^q(L^\Phi_t))(\rn)$ in Definition \ref{d1.3}.

In \cite{af1}, Abl\'e and Feuto introduced the Hardy type space  $\mathcal{H}^{(p,q)}(\rn)$ with $p,\ q\in(0,\infty)$ based on the amalgam space $\ell^q(L^p_1)(\rn)$  and
obtained their atomic characterization when $q\in(0,\fz)$ and $p\in(0,\min\{1,q\})$.
The atomic characterization obtained in Section \ref{s3.4} of this article
essentially generalizes \cite[Theorem 4.4]{af1}.

In \cite{bf}, Bonami and Feuto introduced  the Hardy type spaces $H_*^\Phi(\rn)$
with respect to the amalgam space $(L^\Phi,\ell^1)(\rn)=\ell^1(L^\Phi_1)(\rn)$ with
$\Phi(t):=\frac{t}{\log(e+t)}$ for any $t\in[0,\infty)$, and applied these spaces to study the linear decomposition of the product of
the Hardy space $H^1(\rn)$ and its dual space $\BMO(\rn)$;
see also \cite{cky1}. Since $\ell^1(L^\Phi_1)(\rn)$
is a special case of the Orlicz-amalgam spaces introduced in Definition \ref{d1.3},
from Proposition \ref{th2},
we deduce that the space $H_*^\Phi(\rn)$
is also a special case of the Orlicz-slice Hardy spaces
$(HE_\Phi^{q})_t(\rn)$ considered in this article.

\section{Proofs of main results from Section \ref{s3}\label{s4}}

In this section, we give the proofs of the results presented in Section \ref{s3}.
Since Orlicz-slice spaces are ball quasi-Banach function spaces (see Lemma \ref{ballproof}),
some of these results can be deduced directly from \cite{ykds}, in which a real-variable
theory of Hardy spaces related to ball quasi-Banach function spaces was developed.
However, some properties and characterizations of $(HE_\Phi^q)_t(\rn)$, such as Littlewood-Paley function
and finite atomic characterizations,
need independent and  detailed proofs.

We begin with the proof of Theorem \ref{mdj}.

\begin{proof}[Proof of Theorem \ref{mdj}]
By Lemmas \ref{ballproof} and \ref{littlewood}, we know that $(E_\Phi^q)_t(\rn)$ is a ball
quasi-Banach function space and, for any $r\in(0,\min\{p_\Phi^-,q\})$, $\mathcal{M}$ is bounded on $[(E_\Phi^q)_t(\rn)]^{\frac{1}{r}}$,
where $[(E_\Phi^q)_t(\rn)]^{\frac{1}{r}}$ is the $\frac{1}{r}$-convexification of $(E_\Phi^q)_t(\rn)$ as in
Definition \ref{convex}(i).
Thus, all the assumptions of \cite[Theorem 3.1]{ykds} are satisfied, which
further implies that all the conclusions of Theorem \ref{mdj} hold true. This finishes the proof of Theorem \ref{mdj}.
\end{proof}

\begin{remark} We point out that, by a carefully checking on the proof of \cite[Theorem 3.1]{ykds},
we find that $\lfz b+2\rfz$ in \cite[Theorem 3.1]{ykds} should be $\lfz b+1\rfz$.
\end{remark}

\begin{proof}[Proof of Theorem \ref{poisson}]
By Lemma \ref{littlewood}, we know that, for any $r\in(\frac{1}{\min\{p_\Phi^-,q\}},\infty)$, $\mathcal{M}$ is bounded on $[(E_\Phi^q)_t(\rn)]^{r}$.
Moreover, by \eqref{qiu}, for any $t\in(0,\infty)$ and $z\in\rn$, we have
\begin{align*}
\lf\|\chi_{B(z,1)}\r\|_{(E_\Phi^q)_t(\rn)}&=\lf\{\int_{\rn}\lf[\frac{\|\chi_{B(z,1)}\chi_{B(x,t)}\|
_{L^\Phi(\rn)}}{\|\chi_{B(x,t)}\|_{L^\Phi(\rn)}}\r]^q\,dx\r\}^{\frac{1}{q}}
\ge\lf\{\int_{B(z,\frac{1}{2})}\lf[\frac{\|\chi_{B(z,1)}\chi_{B(x,t)}\|
_{L^\Phi(\rn)}}{\|\chi_{B(x,t)}\|_{L^\Phi(\rn)}}\r]^q\,dx\r\}^{\frac{1}{q}}\\
&\ge\lf\{\int_{B(z,\frac{1}{2})}\lf[\frac{\|\chi_{B(x,\min\{t,1/2\})}\|
_{L^\Phi(\rn)}}{\|\chi_{B(x,t)}\|_{L^\Phi(\rn)}}\r]^q\,dx\r\}^{\frac{1}{q}}
\sim 1,
\end{align*}
which further implies that $\inf_{z\in\rn}\lf\|\chi_{B(z,1)}\r\|_{(E_\Phi^q)_t(\rn)}\gs 1$.
Thus, all assumptions of \cite[Theorem 3.3]{ykds} are satisfied, from which we deduce all the desired conclusions
of Theorem \ref{poisson}.
This finishes the proof of Theorem \ref{poisson}.
\end{proof}

For any $\theta\in(0,\infty)$, the \emph{powered Hardy-Littlewood maximal operator} $\mathcal{M}^{(\theta)}$ is
defined by setting, for any $f\in L^1_{\loc}(\rn)$ and $x\in\rn$,
$$
\mathcal{M}^{(\theta)}(f)(x):=\lf\{\mathcal{M}\lf(|f|^{\theta}\r)(x) \r\}^{\frac{1}{\theta}}.
$$

\begin{lemma}\label{sconvex}
Let $t$, $q\in(0,\infty)$ and $\Phi$ be an Orlicz function with positive lower type $p_{\Phi}^-$ and positive upper type $p_{\Phi}^+$. Let $s\in (0,\min\{p^-_{\Phi},q\}]$ . Then $(E_\Phi^q)_t(\rn)$ is a strictly $s$-convex
ball quasi-Banach function space as in Definition \ref{convex}(ii).
\end{lemma}

\begin{proof}
By Lemma \ref{ballproof}, we already know that $(E_\Phi^q)_t(\rn)$ is a ball quasi-Banach function space.
Now, we show that $(E_\Phi^q)_t(\rn)$ is strictly $s$-convex. To this end, let $s\in(0,\min\{p_{\Phi}^-,q\}]$
and, for any $\tau\in[0,\infty)$, let $\Phi_s(\tau):=\Phi(\sqrt[s]{\tau})$.
Then $\Phi_s$ is of upper type
$\frac{p_{\Phi}^+}{s}$ and of lower type $\frac{p_{\Phi}^-}{s}$, and $\frac{p_{\Phi}^-}{s}\in[1,\infty)$.
Thus, for any $t\in(0,\infty)$, $f\in[(E_\Phi^q)_t(\rn)]^{\frac{1}{s}}$ and $x\in\rn$, we have
$$
\lf\||f|^{\frac{1}{s}}\chi_{B(x,t)}\r\|_{L^\Phi(\rn)}=\lf\|f\chi_{B(x,t)}\r\|_{L^{\Phi_s}(\rn)}^{\frac{1}{s}}
\quad\mathrm{and}\quad
\lf\|\chi_{B(x,t)}\r\|_{L^\Phi(\rn)}
=\lf\|\chi_{B(x,t)}\r\|_{L^{\Phi_s}(\rn)}^{\frac{1}{s}}.
$$
By this and Definition \ref{convex}(i), we know that, for any $t\in(0,\infty)$
and $\{f_j\}_{j=1}^{\infty}\subset[(E_\Phi^q)_t(\rn)]^{\frac{1}{s}}$,
\begin{align}\label{eq}
\lf\|\sum_{j=1}^{\infty}|f_j|\r\|_{[(E_\Phi^q)_t(\rn)]^{\frac{1}{s}}}
&=\lf\{\int_{\rn}\lf[\frac{\|(\sum_{j=1}^{\infty}|f_j|)^{\frac{1}{s}}\chi_{B(x,t)}\|
_{L^\Phi(\rn)}}{\|\chi_{B(x,t)}\|_{L^\Phi(\rn)}}\r]^q\,dx\r\}^{\frac{s}{q}}\\ \noz
&=\lf\{\int_{\rn}\lf[\frac{\|\sum_{j=1}^{\infty}|f_j|\chi_{B(x,t)}\|
_{L^{\Phi_s}(\rn)}}{\|\chi_{B(x,t)}\|_{L^{\Phi_s}(\rn)}}\r]^{\frac{q}{s}}\,dx\r\}^{\frac{s}{q}}
=\lf\|\sum_{j=1}^{\infty}|f_j|\r\|_{(E_{\Phi_s}^{q/s})_t(\rn)}.
\end{align}
Since $\frac{p_{\Phi}^-}{s}\in[1,\infty)$ and $\frac{q}{s}\in[1,\infty)$, from Remark \ref{nfunction}(i),
we deduce that $(E_{\Phi_s}^{q/s})_t(\rn)$ is a Banach space, which, together with \eqref{eq},
further implies that
$$
\lf\|\sum_{j=1}^{\infty}|f_j|\r\|_{[(E_\Phi^q)_t(\rn)]^{\frac{1}{s}}}
=\lf\|\sum_{j=1}^{\infty}|f_j|\r\|_{(E_{\Phi_s}^{q/s})_t(\rn)}
\le\sum_{j=1}^{\infty}\lf\|f_j\r\|_{(E_{\Phi_s}^{q/s})_t(\rn)}
= \sum_{j=1}^{\infty}\lf\|f_j\r\|_{[(E_\Phi^q)_t(\rn)]^{\frac{1}{s}}}.
$$
Thus, $(E_\Phi^q)_t(\rn)$ is strictly $s$-convex, which completes the proof of Lemma \ref{sconvex}.
\end{proof}

\begin{lemma}\label{man28}
Let $t,\ q\in(0,\infty)$ and $\Phi$ be an Orlicz function with positive lower type $p_{\Phi}^-$ and positive upper type $p_{\Phi}^+$. Let
$s\in (0,\min\{p^-_{\Phi},q\}]$ . Then, for any $\theta\in(0,s)$, there exists a positive constant
$C_{(s,\theta)}$, depending on $\theta$ and $s$, but independent of $t$, such that, for any
$\{f_j\}_{j=1}^{\infty}\subset\mathbb{M}(\rn)$,
\begin{equation}\label{28}
\lf\|\lf\{\sum_{j=1}^{\infty}\lf[\mathcal{M}^{(\theta)}\lf(f_j\r)\r]^s\r\}^{\frac{1}{s}}\r\|_{(E_\Phi^q)_t(\rn)}
\le C_{(s,\theta)}
\lf\|\lf\{\sum_{j=1}^{\infty}\lf|f_j\r|^s\r\}^{\frac{1}{s}}\r\|_{(E_\Phi^q)_t(\rn)}.
\end{equation}
\end{lemma}

\begin{proof}
Let $\theta$ and $s$ be as in the lemma.
For any $\tau\in[0,\infty)$, let $\Phi_\theta(\tau):=\Phi(\sqrt[\theta]{\tau})$.
Then $\Phi_\theta$ is of upper type
$\frac{p_{\Phi}^+}{\theta}$ and of lower type $\frac{p_{\Phi}^-}{\theta}$, and $\frac{p_{\Phi}^-}{\theta}$, $\frac{q}{\theta}\in(1,\infty)$.
Then, by Definition \ref{d2}, for any $f\in(E_\Phi^q)_t(\rn)$, we have
$$
\lf\|\lf\{\sum _{j=1}^\infty\lf[\mathcal{M}^{(\theta)}\lf(f_j\r)\r]^s\r\}^{\frac{1}{s}}\r\|_{(E_\Phi^q)_t(\rn)}
=\lf\|\lf\{\sum _{j=1}^\infty\lf[\mathcal{M}\lf(|f_j|^{\theta}\r)\r]^{\frac{s}{\theta}}\r\}^{\frac{\theta}
{s}\frac{1}{\theta}}\r\|_{(E_\Phi^q)_t(\rn)}
=\lf\|\lf\{\sum _{j=1}^\infty\lf[\mathcal{M}\lf(|f_j|^{\theta}\r)\r]^{\frac{s}{\theta}}\r\}^{\frac{\theta}
{s}}\r\|_{(E_{\Phi_\theta}^{q/\theta})_t(\rn)}^{\frac{1}{\theta}}.
$$
From this and Theorem \ref{main}, it follows that
$$
\lf\|\lf\{\sum _{j=1}^\infty\lf[\mathcal{M}^{(\theta)}(f_j)\r]^s\r\}^{\frac{1}{s}}\r\|_{(E_\Phi^q)_t(\rn)}
\lesssim\lf\|\lf\{\sum _{j=1}^\infty|f_j|^s\r\}^{\frac{1}{s}}\r\|_{(E_\Phi^q)_t(\rn)},
$$
namely, \eqref{28} holds true, which completes the proof of Lemma \ref{man28}.
\end{proof}

\begin{lemma}\label{man37}
Let $t,\ q\in(0,\infty)$ and $\Phi$ be an Orlicz function with positive lower type $p_{\Phi}^-$ and positive upper type $p_{\Phi}^+$. Let $r\in(\max\{q,\ p_{\Phi}^+\},\infty]$ and $s\in (0,\min\{p^-_{\Phi},q\})$. Then
there exists a positive constant $C_{(s,r)}$, depending on $s$ and $r$, but independent of $t$,
such that, for any $f\in\mathbb{M}(\rn)$,
\begin{equation}\label{jdjd}
\lf\|\mathcal{M}^{((r/s)')}(f) \r\|_{([(E_\Phi^q)_t(\rn)]^{1/s})^*}\leq C_{(s,r)}\lf\|f \r\|_{([(E_\Phi^q)_t(\rn)]^{1/s})^*},
\end{equation}
here and hereafter, $[(E_\Phi^q)_t(\rn)]^{1/s}$ is the $\frac1s$-convexification of $(E_\Phi^q)_t(\rn)$ as in Definition \ref{convex}(i)
and $([(E_\Phi^q)_t(\rn)]^{1/s})^*$ denotes its dual space.
\end{lemma}

\begin{proof}
For any $\tau\in[0,\infty)$, let $\Phi_s(\tau):=\Phi(\sqrt[s]{\tau})$. Then $\Phi_s$ is of upper type
$p_{\Phi}^+/s$ and of lower type $p_{\Phi}^-/s$, and $p_{\Phi}^-/s\in(1,\infty)$.
As in the proof of Lemma \ref{littlewood}, we know that, for any $f\in\mathbb{M}(\rn)$,
$$
\lf\|f\r\|_{[(E_\Phi^q)_t(\rn)]^{1/s}}=\lf\|f\r\|_{(E_{\Phi_s}^{q/s})_t(\rn)}.
$$
From this, Theorem \ref{dual} and \cite[Proposition 7.8]{ykds}, we deduce that
$$
\lf\|\mathcal{M}^{((r/s)')}(f) \r\|_{([(E_\Phi^q)_t(\rn)]^{1/s}(\rn))^*}=\lf\|\mathcal{M}^{((r/s)')}(f) \r\|_{((E_{\Phi_s}^{q/s})_t(\rn))^*}\sim\lf\|\mathcal{M}^{((r/s)')}(f) \r\|_{(E_{\Psi_s}^{(q/s)'})_t(\rn)},
$$
where $\Psi_s$ is the complementary function to $\Phi_s$ and $\Psi_s$ is of upper type $(p_{\Phi}^-/s)'$
and lower type $(p_{\Phi}^+/s)'.$
Thus, we have
$$
\lf\|\mathcal{M}^{((r/s)')}(f) \r\|_{(E_{\Psi_s}^{(q/s)'})_t(\rn)}=
\lf\|\lf[\mathcal{M}\lf(|f|^{(r/s)'}\r)\r]^{\frac{1}{(r/s)'}} \r\|_{(E_{\Psi_s}^{(q/s)'})_t(\rn)}.
$$
Since $(r/s)'\in(0, \min\{(p_{\Phi}^+/s)',(q/s)'\})$, from Lemma \ref{littlewood},
it follows that
$$
\lf\|\lf[\mathcal{M}\lf(|f|^{(r/s)'}\r)\r]^{\frac{1}{(r/s)'}} \r\|_{(E_{\Psi_s}^{(q/s)'})_t(\rn)}
\lesssim\lf\|f \r\|_{(E_{\Psi_s}^{(q/s)'})_t(\rn)},
$$
which further implies (\ref{jdjd}) and hence completes the proof of Lemma \ref{man37}.
\end{proof}

\begin{lemma}\label{aboso}
Let $t,\ q\in(0,\infty)$ and $\Phi$ be an Orlicz function with positive lower type $p_{\Phi}^-$ and positive upper type $p_{\Phi}^+$. Then $(E_\Phi^q)_t(\rn)$ has an absolutely continuous quasi-norm as in Definition \ref{lian}.
\end{lemma}

\begin{proof}
Let $\{E_j\}_{j=1}^\infty$ be a sequence of measurable sets
that satisfy $E_j\supset E_{j+1}$ for any $j\in\mathbb{N}$ and $\cap_{j=1}^\infty E_j=\emptyset$.
By the fact that
$$
\lf\|\chi_{E_j}\chi_{B(x,t)}\r\|
_{L^\Phi(\rn)}\lesssim\max\lf\{\lf[\int_{B(x,t)}\Phi\lf(\chi_{E_j}(y)\r)\,dy\r]^{1/p_{\Phi}^+},\
\lf[\int_{B(x,t)}\Phi\lf(\chi_{E_j}(y)\r)\,dy\r]^{1/p_{\Phi}^-}  \r\},
$$
we have $\lim_{j\rightarrow\infty}\|\chi_{E_j}\chi_{B(x,t)}\|_{L^\Phi(\rn)}=0$,
which further implies that
$$
\lim_{j\rightarrow\infty}\|\chi_{E_j}\|_{(E_\Phi^q)_t(\rn)}
=\lim_{j\rightarrow\infty}\lf\{\int_{\rn}\lf[\frac{\|\chi_{E_j}\chi_{B(x,t)}\|
_{L^\Phi(\rn)}}{\|\chi_{B(x,t)}\|_{L^\Phi(\rn)}}\r]^q\,dx\r\}^{\frac{1}{q}}=0.
$$
This shows that $(E_\Phi^q)_t(\rn)$ has an absolutely continuous quasi-norm and hence
finishes the proof of Lemma \ref{aboso}.
\end{proof}

\begin{proof}[Proofs of Theorems \ref{atom ch}, \ref{molecular}, \ref{lusin}
and Proposition
\ref{shoulian}]
By Lemmas \ref{ballproof}, \ref{sconvex}, \ref{man28} and \ref{man37}, we know that
$(E_\Phi^q)_t(\rn)$ satisfies all the assumptions of \cite[Theorems 3.6 and 3.7]{ykds}.
Thus, Theorem \ref{atom ch} is a direct
consequence of \cite[Theorems 3.6 and 3.7]{ykds}.

From \cite[Theorems 3.9 and 3.21]{ykds}, we further deduce
Theorems \ref{molecular} and \ref{lusin}.
Using Lemma \ref{aboso} and \cite[Corollary 3.11]{ykds}, we also obtain Proposition
\ref{shoulian}. This finishes the proofs of Theorems \ref{atom ch}, \ref{molecular}, \ref{lusin}
and Proposition \ref{shoulian}.
\end{proof}

\begin{proof}[Proof of Theorem \ref{gfunction}]
We first prove the necessity of Theorem \ref{gfunction}. Let $f\in (HE_\Phi^q)_t(\rn)$. By Theorem \ref{lusin},
we know that $f$ vanishes weakly at infinity. Now we prove that $g(f)\in (E_\Phi^q)_t(\rn)$ and $\|g(f)\|_{(E_\Phi^q)_t(\rn)}\lesssim\|f\|_{(HE_\Phi^q)_t(\rn)}$.
Let $s\in (0,\min\{p^-_{\Phi},q,1\})$. Then, by Theorem \ref{atom ch} and the fact
that $f\in (HE_\Phi^q)_t(\rn)$, we find that
\begin{equation}\label{450}
f=\sum_{Q\in\mathcal{Q}}\lambda_Q a_Q\ \ \text{in}\ \ \mathcal{S}'(\rn),
\end{equation}
where, for any $Q\in \mathcal{Q}$, $a_Q$ is an $((E_\Phi^q)_t(\rn),\ \infty,\ d)$-atom supported on $Q$,
$d\in[\lfloor n(\frac{1}{s}-1)\rfloor,\infty)\bigcap\mathbb{Z}_+$ and $\{\lambda_Q\}_{Q\in\mathcal{Q}}\subset[0,\infty)$ satisfying
\begin{equation}\label{451}
\lf\|\lf\{\sum_{Q\in \mathcal{Q}} \lf[\frac{\lambda_Q}
{\|\chi_{Q}\|_{(E_\Phi^q)_t(\rn)}} \r]^s\chi_{Q} \r\}
^{\frac{1}{s}}\r\|_{(E_\Phi^q)_t(\rn)}\lesssim_s\|f\|_{(HE_\Phi^q)_t(\rn)}.
\end{equation}
From (\ref{450}), we deduce that, for any $x\in \rn$,
\begin{equation}\label{454}
g(f)(x)\le\sum_{Q\in\mathcal{Q}}\lambda_Q g(a_Q)(x).
\end{equation}
Let $r\in(\max\{1,\ q,\ p_{\Phi}^+\},\infty]$. Since $g$ is bounded on $L^r(\rn)$, it follows that,
for any $Q\in\mathcal{Q}$,
$$
\lf\|\chi_{2Q}g(a_Q)\r\|_{L^r(\rn)}\lesssim\|a_{Q}\|_{L^r(\rn)}
\lesssim|Q|^{1/r}\lf\|\chi_{Q}\r\|_{(E_\Phi^q)_t(\rn)}^{-1},
$$
which, combined with \cite[Theoreom 2.10]{ykds}, implies that
\begin{equation}\label{455}
\lf\|\sum_{Q\in\mathcal{Q}}\lambda_Q\chi_{2Q}g(a_Q)\r\|_{(E_\Phi^q)_t(\rn)}
\lesssim
\lf\|\lf\{\sum_{Q\in \mathcal{Q}} \lf[\frac{\lambda_Q}
{\|\chi_{Q}\|_{(E_\Phi^q)_t(\rn)}} \r]^s\chi_{Q} \r\}
^{\frac{1}{s}}\r\|_{(E_\Phi^q)_t(\rn)}.
\end{equation}
Let $\theta\in(0,s)$. Repeating the proof of \cite[(4.4)]{ns} with $\|\chi_Q\|_{L^{p(\cdot)}}$ replaced by $\|\chi_Q\|_{(E_\Phi^q)_t(\rn)}$, we find that, for any $x\in\rn\setminus(2Q)$,
$$
g(a_Q)(x)\lesssim\lf(\frac{l_Q}{|x-x_Q|}\r)^{n+d+1}\lf\|\chi_Q\r\|_{(E_\Phi^q)_t(\rn)}^{-1}
\lesssim\frac{1}{\|\chi_Q\|_{(E_\Phi^q)_t(\rn)}}\mathcal{M}^{(\theta)}(\chi_Q)(x),
$$
where $l_Q$ and $x_Q$ denote the side-length and the center of $Q$, respectively.
This, together with Lemma \ref{man28}, further implies that
$$
\lf\|\sum_{Q\in\mathcal{Q}}\lambda_Q\chi_{\rn\setminus 2Q}g(a_Q)\r\|_{(E_\Phi^q)_t(\rn)}
\lesssim
\lf\|\lf\{\sum_{Q\in \mathcal{Q}} \lf[\frac{\lambda_Q}
{\|\chi_{Q}\|_{(E_\Phi^q)_t(\rn)}} \r]^s\chi_{Q} \r\}
^{\frac{1}{s}}\r\|_{(E_\Phi^q)_t(\rn)}.
$$
From this, (\ref{454}), (\ref{455}) and (\ref{451}), we deduce that
$$
\lf\|g(f)\r\|_{(E_\Phi^q)_t(\rn)}\le
\lf\|\sum_{Q\in\mathcal{Q}}\lambda_Q g(a_Q)\r\|_{(E_\Phi^q)_t(\rn)}
\lesssim\lf\|\lf\{\sum_{Q\in \mathcal{Q}} \lf[\frac{\lambda_Q}
{\|\chi_{Q}\|_{(E_\Phi^q)_t(\rn)}} \r]^s\chi_{Q} \r\}
^{\frac{1}{s}}\r\|_{(E_\Phi^q)_t(\rn)}\lesssim\|f\|_{(HE_\Phi^q)_t(\rn)}.
$$
Therefore, $g(f)\in (E_\Phi^q)_t(\rn)$ and $\|g(f)\|_{(E_\Phi^q)_t(\rn)}\lesssim\|f\|_{(HE_\Phi^q)_t(\rn)}$,
which completes the proof of the necessity of Theorem \ref{gfunction}.

To complete the proof of Theorem \ref{gfunction}, it remains to show the sufficiency of Theorem \ref{gfunction}.
To this end, by Theorem \ref{lusin}, we only need to prove that, if $f\in\mathcal{S}'(\rn)$
vanishes weakly at infinity and
$g(f)\in(E_\Phi^q)_t(\rn)$, then
\begin{equation}\label{349}
\lf\|S(f)\r\|_{(E_\Phi^q)_t(\rn)}\lesssim\lf\|g(f)\r\|_{(E_\Phi^q)_t(\rn)}.
\end{equation}
Let $\varphi\in\mathcal{S}(\rn)$ be such that
$$
\chi_{B(\vec{0}_n,4) \setminus B(\vec{0}_n,2)}\le\varphi\le\chi_{B(\vec{0}_n,8) \setminus B(\vec{0}_n,1)}.
$$
Let $\psi\in\mathcal{S}(\rn)$ satisfy that $\mathcal{F}^{-1}(\psi)=\varphi$. Then it is easy to prove that
$$
\int_{\rn}\psi(x)x^\alpha\,dx=0,\ \ \ \ \ \forall \alpha\in\mathbb{Z}_+^n.
$$
For any $a,\ \tau\in(0,\infty)$, $f\in\mathcal{S}'(\rn)$ and $x\in\rn$, let
$$
(\psi_\tau^*f)_a(x):=\sup_{y\in\rn}\frac{|\psi_\tau*f(y)|}{(1+|x-y|/\tau)^a}.
$$
For any $l\in \mathbb{Z}$, denote $\psi_{2^{-l}}$ and $(\psi_{2^{-l}}^*)_a$ simply by $\psi_{l}$ and $(\psi_{l}^*)_a$,
respectively. It is easy to see that, for any $x\in\rn$,
\begin{align*}
S(f)(x)
&=\lf\{\int_{\Gamma(x)}|\varphi(\tau D)(f)(y)|^{2}\,\frac{dy\,d\tau}{\tau^{n+1}}\r\}^{
\frac{1}{2}}
\lesssim\lf\{\int_{0}^{\infty}\sup_{\{y\in\rn:|y-x|<\tau\}}|\varphi(\tau D)(f)(y)|^{2}\,\frac{dy\,d\tau}{\tau}\r\}^{
\frac{1}{2}}\\
&\lesssim\lf\{\int_{0}^{\infty}\lf[(\psi_\tau^*f)_a(x)\r]^{2}\,\frac{d\tau}{\tau}\r\}^{
\frac{1}{2}}.
\end{align*}
Let, for any $x\in\rn$,
$$
P_a(f)(x):=\lf\{\int_{0}^{\infty}\lf[(\psi_\tau^*f)_a(x)\r]^{2}\,\frac{d\tau}{\tau}\r\}^{
\frac{1}{2}}.
$$
Thus, to show (\ref{349}), it suffices to prove that, if $f\in\mathcal{S}'(\rn)$ and
$g(f)\in(E_\Phi^q)_t(\rn)$, then
\begin{equation}\label{77}
\lf\|P_a(f)\r\|_{(E_\Phi^q)_t(\rn)}\lesssim\lf\|g(f)\r\|_{(E_\Phi^q)_t(\rn)}.
\end{equation}
Let $a\in(\frac{n}{\min\{p_{\Phi}^-,\ q\}},\infty)$. We choose $r\in(\frac{n}{a},\min\{p_{\Phi}^-,\ q\})$.
Then, by \cite[Lemma 3.5]{lsuy}, we find that, for any $l\in \mathbb{Z}$, $\tau\in[1,2]$, $N\in \mathbb{N}$, $a\in(0,N]$
and $x\in\rn$,
$$
\lf[(\psi_{2^{-l}\tau}^*f)_a(x)\r]^{r}
\lesssim\sum_{k=0}^{\infty}2^{-kNr}2^{(k+l)n}\int_{\rn}\frac{|(\psi_{k+l})_\tau*f(y)|^r}{(1+2^l|x-y|)^{ar}}dy.
$$
From the Minkowski inequality, it follows that
\begin{align*}
\lf\{\int_{1}^{2}\lf[(\psi_{2^{-l}\tau}^\ast f)_a(x)\r]^{2}\,\frac{d\tau}{\tau}\r\}^{
\frac{r}{2}}
&\lesssim\lf\{\int_{1}^{2}\lf[\sum_{k=0}^{\infty}2^{-kNr}2^{(k+l)n}
\int_{\rn}\frac{|(\psi_{k+l})_\tau*f(y)|^r}{(1+2^l|x-y|)^{ar}}dy\r]^{\frac{2}{r}}\,\frac{d\tau}{\tau}\r\}^{
\frac{r}{2}}\\
&\lesssim\sum_{k=0}^{\infty} 2^{-kNr}2^{(k+l)n}\int_{\rn}
\frac{[\int_{1}^{2}|(\psi_{k+l})_\tau*f(y)|^2\frac{d\tau}{\tau}]^{\frac{r}{2}}}{(1+2^l|x-y|)^{ar}}\,dy\\
&\lesssim\sum_{k=0}^{\infty} 2^{-kNr}2^{kn}
\lf(g_l\ast\lf[\int_{1}^{2}|(\psi_{k+l})_\tau\ast f(\cdot)|^2\frac{d\tau}{\tau}\r]^{\frac{r}{2}}\r)(x)\\
&\lesssim\sum_{k=0}^{\infty} 2^{-k(Nr-n)}
\mathcal{M}\lf(\lf[\int_{1}^{2}|(\psi_{k+l})_\tau\ast f(\cdot)|^2\frac{d\tau}{\tau}\r]^{\frac{r}{2}}\r)(x),
\end{align*}
where, for any $l\in\mathbb{Z}$ and $x\in\rn$,
$$
g_l(x):=\frac{2^{nl}}{(1+2^l|x|)^{ar}}\ \in L^1(\rn)\ \ \ \mbox{and}\ \ \ \|g_l\|_{L^1(\rn)}\lesssim1.
$$
Then, by the Minkowski inequality, we find that
\begin{align*}
\lf\|P_a(f)\r\|_{(E_\Phi^q)_t(\rn)}
&=\lf\|\lf\{\sum_{l=-\infty}^{\infty}
\int_{2^{-l}}^{2^{-l+1}}\lf[(\psi_\tau^\ast f)_a(\cdot)\r]^{2}\,\frac{d\tau}{\tau}\r\}^{
\frac{1}{2}}\r\|_{(E_\Phi^q)_t(\rn)}\\
&=\lf\|\lf\{\sum_{l=-\infty}^{\infty}
\int_{1}^{2}\lf[(\psi_{2^{-l}}^\ast f)_a(\cdot)\r]^{2}\,\frac{d\tau}{\tau}\r\}^{
\frac{1}{2}}\r\|_{(E_\Phi^q)_t(\rn)}\\
&\lesssim\lf\|\lf\{\sum_{l=-\infty}^{\infty}\lf[
\sum_{k=0}^{\infty} 2^{-k(Nr-n)}
\mathcal{M}\lf(\lf[\int_{1}^{2}|(\psi_{k+l})_\tau\ast f(\cdot)|^2\frac{d\tau}{\tau}\r]^{\frac{r}{2}}\r)\r]^{\frac{2}{r}}\r\}^{
\frac{1}{2}}\r\|_{(E_\Phi^q)_t(\rn)}\\
&\lesssim\lf\|\lf\{\sum_{l=-\infty}^{\infty}\lf[
\mathcal{M}\lf(\lf[\int_{1}^{2}|(\psi_{l})_\tau\ast f(\cdot)|^2\frac{d\tau}{\tau}\r]^{\frac{r}{2}}\r)\r]^{\frac{2}{r}}\r\}^{
\frac{1}{2}}\r\|_{(E_\Phi^q)_t(\rn)}.
\end{align*}
By the fact that $r\in(\frac{n}{a},\min\{p_{\Phi}^-,\ q\})$ and Theorem \ref{main}, we conclude that
\begin{align*}
&\lf\|\lf\{\sum_{l=-\infty}^{\infty}\lf[
\mathcal{M}\lf(\lf[\int_{1}^{2}|(\psi_{l})_\tau\ast f(\cdot)|^2\frac{d\tau}{\tau}\r]^{\frac{r}{2}}\r)\r]^{\frac{2}{r}}\r\}^{
\frac{1}{2}}\r\|_{(E_\Phi^q)_t(\rn)}\\
&\quad\lesssim
\lf\|\lf\{\sum_{l=-\infty}^{\infty}\int_{1}^{2}|(\psi_{l})_\tau\ast f(\cdot)|^2\frac{d\tau}{\tau}\r\}^{
\frac{1}{2}}\r\|_{(E_\Phi^q)_t(\rn)}\lesssim\lf\|g(f)\r\|_{(E_\Phi^q)_t(\rn)},
\end{align*}
which implies that (\ref{77}) holds true. This finishes the proof of Theorem \ref{gfunction}.
\end{proof}

\begin{proof}[Proof of Theorem \ref{glamda}]
To prove this theorem, we only need to show the necessity, since
the sufficiency is easy because of Theorem \ref{lusin} and the obvious fact that, for any $f\in\mathcal{S}'(\rn)$
and $x\in\rn$, $S(f)(x)\le g_\lambda^*(f)(x)$.

To show the necessity, for any $f\in (HE_\Phi^q)_t(\rn)$, by Theorem \ref{lusin}, we know that $f$
vanishes weakly at infinity. From the fact that $\lambda\in(1+\frac{2}{\min\{p_{\Phi}^-,\ q\}},\infty)$, we deduce that
there exists $a\in(\frac{n}{\min\{p_{\Phi}^-,\ q\}},\infty)$ such that $\lambda\in(1+\frac{2a}{n},\infty)$ and, for any $x\in\rn$,
\begin{align*}
g_\lambda^*(f)(x)&=\lf\{\int_0^{\infty}\int_{\rn}\lf(\frac{\tau}{\tau+|x-y|}\r)^{\lambda n}|\varphi(\tau D)(f)(y)|^{2}\,\frac{dy\,d\tau}{\tau^{n+1}}\r\}^{\frac{1}{2}}\\
&\lesssim\lf\{\int_0^{\infty}\lf[(\psi_\tau^*f)_a(x)\r]^2\int_{\rn}\lf(1+\frac{|x-y|}{\tau}\r)^{2a-\lambda n}\,\frac{dy\,d\tau}{\tau^{n+1}}\r\}^{\frac{1}{2}}\\
&\sim\lf\{\int_0^{\infty}\lf[(\psi_\tau^*f)_a(x)\r]^2\,\frac{d\tau}{\tau}\r\}^{\frac{1}{2}}\sim P_a(f)(x),
\end{align*}
which, combined with (\ref{77}) and Theorem \ref{gfunction}, implies that
$$
\lf\|g_\lambda^*(f)\r\|_{(E_\Phi^q)_t(\rn)}\lesssim\|f\|_{(HE_\Phi^q)_t(\rn)}.
$$
This finishes the proof of Theorem \ref{glamda}.
\end{proof}

To show Theorem \ref{finite}, we need the following lemma.
\begin{lemma}\label{youjie}
Let $t,\ q\in(0,\infty)$ and $\Phi$ be an Orlicz function with positive lower type $p_{\Phi}^-$ and positive upper type $p_{\Phi}^+$.
Let $N\in\mathbb{N}\cap(\lfloor\frac{n}{\min\{p_\Phi^-,q\}}+1\rfloor,\infty)$.
Suppose $f\in(HE_\Phi^q)_t(\rn)$,
$\|f\|_{(HE_\Phi^q)_t(\rn)}=1$ and $\supp(f)\subset B(\vec{0}_n,R)$ with $R\in(1,\infty)$. Then
there exists a positive constant $C_{(N)}$, depending on $N$, but independent of $f$ and $t$, such that,
for any $x\notin B(\vec{0}_{n},4R)$,
\begin{equation}\label{423}
M_N(f)(x)\le  C_{(N)}\lf\|\chi_{B(\vec{0}_n,R)}\r\|_{(E_\Phi^q)_t(\rn)}^{-1}.
\end{equation}
\end{lemma}
\begin{proof}
For any $x\notin B(\vec{0}_n,4R)$, by (\ref{zhujida}), we have
$$
M_N(f)(x)\le M_N^0(f)(x).
$$
To prove (\ref{423}), it suffices to show that, for any $\varphi\in\mathcal{F}_N(\rn)$,
$\tau\in(0,\infty)$ and $x\notin B(\vec{0}_{n},4R)$,
$$
|\varphi_\tau\ast f(x)|\lesssim \lf\|\chi_{B(\vec{0}_n,R)}\r\|_{(E_\Phi^q)_t(\rn)}^{-1}.
$$
Let $\theta\in\mathcal{S}(\rn)$ be such that
$\supp(\theta)\subset B(\vec{0}_n,2)$, $0\le\theta\le1$ and $\theta\equiv1$ on $B(\vec{0}_n,1)$. We distinguish two
cases with respect to the size of $\tau$.

For any $\tau\in [R,\infty)$ and $x\notin B(\vec{0}_{n},4R)$,
arguing as in the proof of \cite[Lemma 7.10]{cw}, we have
\begin{equation}\label{424}
\varphi_\tau\ast f(x)=\psi_R\ast f(\vec{0}_n)
\end{equation}
and $c\psi\in\mathcal{F}_N(\rn)$ with $c:=C_{(N)}$, where, for any $\tau\in [R,\infty)$ and $z\in\rn$,
$$
\psi(z):=\lf(\frac{R}{\tau}\r)^n\Phi\lf(\frac{x}{\tau}+\frac{Rz}{\tau}\r)\theta(z).
$$
Therefore, (\ref{424}) ensures that, for any $x\notin B(\vec{0}_n,4R)$,
$$
|\varphi_\tau\ast f(x)|\lesssim M_N(f)(z),\quad \forall\ z\in B(\vec{0}_n,R),
$$
which, together with $\|f\|_{(HE_\Phi^q)_t(\rn)}=1$, further implies that, for any $x\notin B(\vec{0}_n,4R)$,
\begin{align}\label{111}
|\varphi_\tau\ast f(x)|&\lesssim  \inf_{z\in B(\vec{0}_n,R)}M_N(f)(z)\\\noz
&\lesssim
\frac{\|\chi_{B(\vec{0}_n,R)}\inf_{z\in B(\vec{0}_n,R)}M_N(f)(z)\|_{(E_\Phi^q)_t(\rn)}}{\|\chi_{B(\vec{0}_n,R)}\|_{(E_\Phi^q)_t(\rn)}}\\ \noz
&\lesssim\lf\|\chi_{B(\vec{0}_n,R)}\r\|_{(E_\Phi^q)_t(\rn)}^{-1}\lf\|M_N(f)\r\|_{(E_\Phi^q)_t(\rn)}
\lesssim\lf\|\chi_{B(\vec{0}_n,R)}\r\|_{(E_\Phi^q)_t(\rn)}^{-1}.
\end{align}
For any $\tau\in(0, R)$ and $u\in B(\vec{0}_n,\frac{R}{2})$, following \cite[Lemma 7.10]{cw}, we obtain
\begin{equation*}
\varphi_\tau\ast f(x)=\psi_\tau\ast f(u)
\end{equation*}
and $c\psi\in\mathcal{F}_N(\rn)$ with $c:=C_{(N)}$, where, for any $\tau\in(0, R)$ and $z\in\rn$,
$$
\psi(z):=\Phi\lf(\frac{x-u}{\tau}+z\r)\theta\lf(\frac{u}{R}-\frac{tz}{R}\r).
$$
Thus, for any $x\notin B(\vec{0}_n,4R)$, we have
$$
|\varphi_\tau\ast f(x)|\lesssim M_N(f)(u),\quad \forall\ u\in B\lf(\vec{0}_n,\frac{R}{2}\r).
$$
By proceeding as in (\ref{111}), we further conclude that, for any $\tau\in(0, R)$
and $x\notin B(\vec{0}_n,4R)$,
$$
|\varphi_\tau\ast f(x)|\lesssim \lf\|\chi_{B(\vec{0}_n,R)}\r\|_{(E_\Phi^q)_t(\rn)}^{-1}.
$$
This finishes the proof of Lemma \ref{youjie}.
\end{proof}

\begin{proof}[Proof of Theorem \ref{finite}]
Obviously, from Theorem \ref{atom ch}, we deduce that
$$(HE_\Phi^{q,r,d})_t^{\fin}(\rn)\subset(HE_\Phi^q)_t(\rn)$$
and, for any $f\in(HE_\Phi^{q,r,d})_t^{\fin}(\rn)$,
$$
\|f\|_{(HE_\Phi^q)_t(\rn)}\lesssim\|f\|_{(HE_\Phi^{q,r,d})_t^{\fin}(\rn)}.
$$
Thus, to complete the proof of Theorem \ref{finite},
we still need to show that, for any given $t,\ q,\ d$ as in Theorem \ref{finite} and
$r\in(\max\{1,\ q,\ p_{\Phi}^+\},\infty)$ and any $f\in(HE_\Phi^{q,r,d})_t^{\fin}(\rn)$,
$$
\|f\|_{(HE_\Phi^{q,r,d})_t^{\fin}(\rn)}\lesssim\|f\|_{(HE_\Phi^q)_t(\rn)},
$$
and that a similar estimate also holds true for $r=\infty$ and any given $t,\ q,\ d$ as in Theorem \ref{finite}
and any $f\in(HE_\Phi^{q,\infty,d})_t^{\fin}(\rn)\cap\mathcal{C}{(\rn)}$.

Assume that $r\in(\max\{1,\ q,\ p_{\Phi}^+\},\infty]$ and, by the homogeneity of both $\|\cdot\|_{(HE_\Phi^{q,r,d})_t^{\fin}(\rn)}$ and $\|\cdot\|_{(HE_\Phi)_t(\rn)}$,
without loss of generality,
we may also assume that $f\in(HE_\Phi^{q,r,d})_t^{\fin}(\rn)$ and $\|f\|_{(HE_\Phi^q)_t(\rn)}=1$.
Since $f$ is a finite linear combination of $((E_\Phi^q)_t(\rn),\ r,\ d)$-atoms,
it follows that there exists $R\in(1,\infty)$ such that
$f$ is supported on $B(\vec{0}_n,R)$. Thus, if let $N$ be as in Lemma \ref{youjie}, then, by Lemma \ref{youjie},
there exists a positive constant $C_{(N)}$ such that, for any $x\notin B(\vec{0}_{n},4R)$,
\begin{equation}\label{78}
M_N(f)(x)\le  C_{(N)}\lf\|\chi_{B(\vec{0}_n,R)}\r\|_{(E_\Phi^q)_t(\rn)}^{-1}.
\end{equation}
For each $j\in\mathbb{Z}$, let $\mathcal{O}_j:=\lf\{x\in\rn:\ M_N(f)(x)>2^j\r\}$. Denote by
$j'$ the largest integer $j$ such that
\begin{equation}\label{130}
2^{j'}<C_{(N)}\lf\|\chi_{B(\vec{0}_n,R)}\r\|_{(E_\Phi^q)_t(\rn)}^{-1}.
\end{equation}
Then, by (\ref{78}), for any $j\in\{j'+1, j'+2, \ldots\}$,
\begin{equation}\label{oo}
\mathcal{O}_j\subset B(\vec{0}_n,4R).
\end{equation}
Since $f\in L^r(\rn)$, from the proof
of \cite[Proposition 4.3]{ykds},
it follows that there exist a sequence $\{(a_{j,k},Q_{j,k})\}_{j\in\mathbb{Z},k\in K_j}$ of pairs of
$((E_\Phi^q)_t(\rn),\ \infty,\ d)$-atoms and their supports, and a
sequence of scalars, $\{\lambda_{j,k}\}_{j\in\mathbb{Z},k\in K_j}\subset[0,\infty)$, such that
\begin{equation}\label{eqq}
f=\sum_{j=-\infty}^{\infty}\sum_{k\in K_j}\lambda_{j,k}a_{j,k}
\end{equation}
in both $\mathcal{S}'(\rn)$ and almost everywhere,
where $\{K_j\}_{j\in\mathbb{Z}}$ is a set of indices and $\{Q_{j,k}\}_{j\in\mathbb{Z},k\in K_j}$ a
family of closed cubes with disjoint interiors such that $\mathcal{O}_j=\cup_{k\in K_j}Q_{j,k}$ as in \cite[Lemma 2.23]{ykds}.
Moreover, for some given $s\in (0,\min\{p_\Phi^-,q\})$, we have
\begin{equation}\label{331}
\lf\|\lf\{\sum_{j=-\infty}^{\infty}\sum_{k\in K_j} \lf[\frac{\lambda_{j,k}}
{\|\chi_{Q_{j,k}}\|_{(E_\Phi^q)_t(\rn)}} \r]^s\chi_{Q_{j,k}} \r\}
^{\frac{1}{s}}\r\|_{(E_\Phi^q)_t(\rn)}\lesssim\|f\|_{(HE_\Phi^q)_t(\rn)}.
\end{equation}
Define
\begin{equation}\label{eq2}
h:=\sum_{j=-\infty}^{j'}\sum_{k\in K_j}\lambda_{j,k}a_{j,k}\quad\mathrm{and}
\quad l:=\sum_{j=j'+1}^{\infty}\sum_{k\in K_j}\lambda_{j,k}a_{j,k},
\end{equation}
where the series converge in both $\mathcal{S}'(\rn)$ and almost everywhere. Clearly $f=h+l$ and,
by (\ref{oo}), $\supp(l)\subset\cup_{j>j'}\mathcal{O}_j\subset B(\vec{0}_n,4R)$. Therefore, $h=l=0$ on
$\rn\setminus B(\vec{0}_n,4R)$ and hence $\supp(h)\subset B(\vec{0}_n,4R)$. Moreover, by the proof
of \cite[Proposition 4.3]{ykds}, we know that there exists a positive constant $C_0$ such that
$\|\lambda_{j,k}a_{j,k}\|_{L^{\infty}(\rn)}\le C_02^j$.
Since $f\in L^r(\rn)$ and $r\in(1,\infty]$, from the boundedness on $L^r(\rn)$ of the
Hardy-Littlewood maximal operator, it follows that $M_N(f)\in L^r(\rn)$.
Then we have
\begin{align}\label{ll}
\|l\|_{L^r(\rn)}&\leq\lf\|\sum_{j=j'+1}^{\infty}\sum_{k\in K_j}|\lambda_{j,k}a_{j,k}|\r\|_{L^r(\rn)}\lesssim\lf\|\sum_{j=j'+1}^{\infty}
\sum_{k\in K_j}2^j\chi_{Q_{j,k}}\r\|_{L^r(\rn)}\\ \noz
&\lesssim\lf\|\sum_{j=j'+1}^{\infty}2^j\chi_{\mathcal{O}_j}\r\|_{L^r(\rn)}\lesssim\|M_N(f)\|_{L^r(\rn)}.
\end{align}
Thus, $l\in L^r(\rn)$ and so $h=f-l\in L^r(\rn)$.
It follows from (\ref{ll}) and the H\"older inequality that, for any $|\beta|\leq d$,
\begin{align*}
\int_{\rn}\sum_{j=j'+1}^{\infty}\sum_{k\in K_j}\lf|x^\beta\r||\lambda_{j,k}a_{j,k}(x)|\,dx
&\le\lf\|\sum_{j=j'+1}^{\infty}\sum_{k\in K_j}|\lambda_{j,k}a_{j,k}|\r\|_{L^r(\rn)}
\lf\{\int_{B(\vec{0}_n,4R)}\lf|x^\beta\r|^{r'}\,dx\r\}^{\frac1r'}\\
&\lesssim_R\|M_N(f)\|_{L^r(\rn)}<\infty.
\end{align*}
This, combined with the vanishing moments of $a_{j,k}$, implies that $l$
has vanishing moments up to $d$ and hence so does $h$ by $h=f-l$.

In order to estimate the size of $g$ in $B(\vec{0}_n,4R)$, recall that
\begin{equation}\label{129}
\lf\|\lambda_{j,k}a_{j,k}\r\|_{L^{\infty}(\rn)}\lesssim2^j,\ \supp(a_{j,k})\subset Q_{j,k}
\quad\mathrm{and}\quad\sum_{k\in K_j}\chi_{Q_{j,k}}\lesssim1.
\end{equation}
It is easy to show that
\begin{equation}\label{131}
\lf\|\chi_{B(\vec{0}_n,R)}\r\|_{(E_\Phi^q)_t(\rn)}\sim\lf\|\chi_{Q(\vec{0}_n,8R)}\r\|_{(E_\Phi^q)_t(\rn)}.
\end{equation}
Indeed, it is easy to see that there exist $M\in\mathbb{N}$ and $\{x_1, \ldots, x_M\}\subset\rn$, independent
of $t$ and $f$,
such that $M\lesssim1$ and $Q(\vec{0}_n,8R)\subseteq\bigcup_{m=1}^MB(x_m,R)$,
which further implies that
\begin{equation}\label{555}
\lf\|\chi_{Q(\vec{0}_n,8R)}\r\|_{(E_\Phi^q)_t(\rn)}\lesssim
\lf\|\sum_{m=1}^M\chi_{B(x_m,R)}\r\|_{(E_\Phi^q)_t(\rn)}\lesssim
\sum_{m=1}^M\lf\|\chi_{B(x_m,R)}\r\|_{(E_\Phi^q)_t(\rn)}.
\end{equation}
Observing that, for any $t\in(0,\infty)$, $m\in\mathbb{N}$ and $x\in\rn $,
$$
\lf\|\chi_{B(x_m,R)}\chi_{B(x,t)}\r\|_{L^\Phi(\rn)}
=\lf\|\chi_{B(\vec{0}_n,R)}\chi_{B(x-x_m,t)}\r\|_{L^\Phi(\rn)},
$$
by this and \eqref{qiu} with $\widetilde{C}_{(\Phi,t)}$ as therein, we have
\begin{align*}
\lf\|\chi_{B(x_m,R)}\r\|_{(E_\Phi^q)_t(\rn)}
&=\lf\{\int_{\rn}\lf[\frac{\|\chi_{B(\vec{0}_n,R)}\chi_{B(x-x_m,t)}\|_{L^\Phi(\rn)}}
{\|\chi_{B(x,t)}\|_{L^\Phi(\rn)}}\r]^q\,dx\r\}^{\frac{1}{q}}\\
&=\frac{1}{\widetilde{C}_{(\Phi,t)}}\lf\{\int_{\rn}\lf\|\chi_{B(\vec{0}_n,R)}\chi_{B(x-x_m,t)}\r\|_{L^\Phi(\rn)}
^q\,dx\r\}^{\frac{1}{q}}\\
&=\frac{1}{\widetilde{C}_{(\Phi,t)}}\lf\{\int_{\rn}\lf\|\chi_{B(\vec{0}_n,R)}\chi_{B(x,t)}\r\|_{L^\Phi(\rn)}
^q\,dx\r\}^{\frac{1}{q}}
=\lf\|\chi_{B(\vec{0}_n,R)}\r\|_{(E_\Phi^q)_t(\rn)},
\end{align*}
which, together with $\eqref{555}$, implies that $\|\chi_{Q(\vec{0}_n,8R)}\|_{(E_\Phi^q)_t(\rn)}\lesssim\|\chi_{B(\vec{0}_n,R)}\|_{(E_\Phi^q)_t(\rn)}.$
The converse inequality holds true obviously.
Thus, we obtain \eqref{131}.

Combining (\ref{130}), (\ref{129}) and (\ref{131}), we conclude that
\begin{equation*}
\|h\|_{L^{\infty}(\rn)}\le\sum_{j\le j'}\lf\|\sum_{k\in K_j}|\lambda_{j,k}a_{j,k}|\r\|_{L^{\infty}(\rn)}
\lesssim\sum_{j\le j'}2^j\lesssim2^{j'}\lesssim\lf\|\chi_{B(\vec{0}_n,R)}\r\|_{(E_\Phi^q)_t(\rn)}^{-1}
\le\widetilde{C}\lf\|\chi_{Q(\vec{0}_n,8R))}\r\|_{(E_\Phi^q)_t(\rn)}^{-1},
\end{equation*}
where $\widetilde{C}$ is a positive constant independent of $f$ and $t$. From this and
the fact that $h$ has vanishing moments up to $d$, it follows that $\widetilde{C}^{-1}h$
is an $((E_\Phi^q)_t(\rn),\ \infty,\ d)$-atom.

Now, to complete the proof of Theorem \ref{finite}(i),
we assume that $r\in(\max\{1,\ q,\ p_{\Phi}^+\},\infty)$.
We rewrite $l$ as a finite linear
combination of $((E_\Phi^q)_t(\rn),\ r,\ d)$-atoms. For any $i\in\mathbb{N}$, let
$$
F_i:=\lf\{(j,k)\in\mathbb{Z}\times\mathbb{Z}_+:\ j\in\lf\{j'+1,j'+2, \ldots\r\},\ k\in K_j,\ |j|+k\le i\r\},
$$
and $l_i:=\sum_{(j,k)\in F_i}\lambda_{j,k}a_{j,k}$. Since the series
$l=\sum_{j=j'+1}^{\infty}\sum_{k\in K_j}\lambda_{j,k}a_{j,k}$ converges in $L^r(\rn)$,
it follows that there exists a positive integer $i_0$, which may depend on $t$ and $f$,
such that
$$\|l-l_{i_0}\|_{L^r(\rn)}\le\frac{|Q(\vec{0}_n,8R)|^{\frac{1}{r }}}{\|\chi_{Q(\vec{0}_n,8R)}\|_{(E_\Phi^q)_t(\rn)}}.$$
Thus, $l-l_{i_0}$ is an $((E_\Phi^q)_t(\rn),\ r,\ d)$-atom, because $\supp(l-l_{i_0})\subset B(\vec{0}_n,4R)\subset Q(\vec{0}_n,8R)$
and, for any $|\beta|\le d$, $\int_{\rn}(l-l_{i_0})(x)x^\beta dx=0$. Therefore,
$$
f=h+l=\widetilde{C}\widetilde{C}^{-1}h+(l-l_{i_0})+l_{i_0}
$$
is a finite decomposition of $f$ in terms of $((E_\Phi^q)_t(\rn),\ r,\ d)$-atoms. Moreover, by (\ref{331}),
we have
\begin{align*}
\|f\|_{(HE_\Phi^{q,r,s})_t^{\fin}(\rn)}
&\le\lf\|\lf\{\lf[\frac{\widetilde{C}}
{\|\chi_{Q(\vec{0}_n,8R)}\|_{(E_\Phi^q)_t(\rn)}} \r]^s\chi_{Q(\vec{0}_n,8R)}+\lf[\frac{1}
{\|\chi_{Q(\vec{0}_n,8R)}\|_{(E_\Phi^q)_t(\rn)}} \r]^s\chi_{Q(\vec{0}_n,8R)}\r.\r.\\
&\quad\lf.\lf.+\sum_{(j,k)\in F_{i_0}}\lf[\frac{\lambda_{j,k}}
{\|\chi_{Q_{j,k}}\|_{(E_\Phi^q)_t(\rn)}} \r]^s\chi_{Q_{j,k}} \r\}
^{\frac{1}{s}}\r\|_{(E_\Phi^q)_t(\rn)}\\
&\lesssim1+\lf\|\lf\{\sum_{(j,k)\in F_{i_0}}\lf[\frac{\lambda_{j,k}}
{\|\chi_{Q_{j,k}}\|_{(E_\Phi^q)_t(\rn)}} \r]^s\chi_{Q_{j,k}} \r\}
^{\frac{1}{s}}\r\|_{(E_\Phi^q)_t(\rn)}\\
&\lesssim1+\lf\|\lf\{\sum_{j=-\infty}^{\infty}\sum_{k\in K_j} \lf[\frac{\lambda_{j,k}}
{\|\chi_{Q_{j,k}}\|_{(E_\Phi^q)_t(\rn)}} \r]^s\chi_{Q_{j,k}} \r\}
^{\frac{1}{s}}\r\|_{(E_\Phi^q)_t(\rn)}\\
&\lesssim1+\|f\|_{(HE_\Phi^q)_t(\rn)}\lesssim1.
\end{align*}
Thus, $\|f\|_{(HE_\Phi^{q,r,d})_t^{\fin}(\rn)}\lesssim1$.
This finishes the proof of Theorem \ref{finite}(i).

To prove Theorem \ref{finite}(ii), we assume that
$f\in(HE_\Phi^{q,\infty,d})_t^{\fin}(\rn)\cap\mathcal{C}(\rn)$ and
$\|f\|_{(HE_\Phi^q)_t(\rn)}=1$. Since $f$ has a compact support, it follows that $f$ is uniformly continuous.
Then, by this, the proof of \cite[Proposition 4.3]{ykds} and the argument
presented in \cite[pp.\,108-109]{em}, we know that each $((E_\Phi^q)_t(\rn),\ \infty,\ d)$-atom $a_{j,k}$
in \eqref{eqq} is continuous. Since $f$ is bounded, from the boundedness of $M_N(f)$ on
$L^{\infty}(\rn)$, it follows that there exists a
positive integer $j''>j'$ such that $\mathcal{O}_j=\emptyset$ for
any $j\in\{j''+1, j''+2, \ldots\}$. Consequently,
in this case, $l$ in \eqref{eq2} becomes
$$
l=\sum_{j=j'+1}^{j''}\sum_{k\in K_j}\lambda_{j,k}a_{j,k}.
$$

Let $\epsilon\in(0,\infty)$. Since $f$ is uniformly continuous, it follows that there exists
$\delta\in(0,\infty)$ such that, if $|x-y|<\delta$, then $|f(x)-f(y)|<\epsilon$.
Write $l=l_1^\epsilon+l_2^\epsilon$ with
$l_1^\epsilon:=\sum_{(j,k)\in F_1}\lambda_{j,k}a_{j,k}$ and
$l_2^\epsilon:=\sum_{(j,k)\in F_2}\lambda_{j,k}a_{j,k}$, where
$$F_1:=\lf\{(j,k)\in\mathbb{Z}\times\mathbb{Z}_+:\ j\in\lf\{j'+1, \ldots,j''\r\},\ k\in K_j,
\ {\rm diam}(Q_{j,k})\ge\delta\r\}$$
and
$$F_2:=\lf\{(j,k)\in\mathbb{Z}\times\mathbb{Z}_+:\ j\in\lf\{j'+1, \ldots, j''\r\},\ k\in K_j,\ {\rm
diam}(Q_{j,k})<\delta\r\}.$$
Observe that $l_1^\epsilon$ is a finite summation.
Since the atoms are continuous, we know that
$l_1^\epsilon$ is also a continuous function. Furthermore, using this fact and repeating
the proof of \cite[Theorem 6.2]{bly}, we conclude that
$$\|l_2^\epsilon\|_{L^\infty(\rn)}\lesssim(j''-j')\epsilon.$$
This means that one can write $l$ as the sum of one continuous term and one which is uniformly arbitrarily small.
Thus, $l$ is continuous and so is $h=f-l$.

To find a finite atomic decomposition of $f$, we use again the splitting $l=l_1^\epsilon+l_2^\epsilon$.
It is clear that, for any $\epsilon\in(0,\infty)$, $l_1^\epsilon$ is a finite combination of continuous
$((E_\Phi^q)_t(\rn),\ \infty,\ d)$-atoms. Also, since both $l$ and
$l_1^\epsilon$ are continuous and have vanishing moments up to order $d$, it follows that
$l_2^\epsilon=l-l_1^\epsilon$ is also continuous and has vanishing moments up to order $d$.
Moreover, $\supp(l_2^\epsilon)\subset B(\vec{0}_n,4R)\subset Q(\vec{0}_n,8R)$ and $ \|l_2^\epsilon\|_{L^\infty(\rn)}\lesssim(j''-j')\epsilon$.
So we can choose $\epsilon$ small enough such that $l_2^\epsilon$ becomes an arbitrarily small multiple
of a continuous $((E_\Phi^q)_t(\rn),\ \infty,\ d)$-atom. Therefore, $f=h+l_1^\epsilon+l_2^\epsilon$ is
a finite linear continuous atomic combination. Then, by an argument similar to the proof of (i), we obtain
$\|f\|_{(HE_\Phi^{q,\infty,d})_t^{\fin}(\rn)}\lesssim1$. This finishes the proof of (ii) and hence
of Theorem \ref{finite}.
\end{proof}

\section{Dual spaces of Orlicz-slice Hardy spaces}\label{s5}

In this section, we provide a description of the dual space of the
Orlicz-slice Hardy space $(HE_\Phi^q)_t(\rn)$, with $\max\{p_{\Phi}^+,\ q\}\in(0,1]$,
in terms of Campanato spaces.
This description is a consequence of
both their atomic characterization from Theorem \ref{atom ch}
and their finite atomic characterization from Theorem \ref{finite} as well as
some basic tools from functional analysis.

\begin{definition}
A function $\Phi:\ [0,\infty)\rightarrow\mathbb{R}$ is said to be \emph{concave} if,
for any $t,\ s\in[0,\infty)$ and $\lambda\in[0,1]$,
$$
\lambda\Phi(t)+(1-\lambda)\Phi(s)\leq \Phi(\lambda t+(1-\lambda)s).
$$
\end{definition}

\begin{lemma}\label{djdj}
Let $\Phi$ be an Orlicz function with positive lower type $p_{\Phi}^-$ and positive upper type $p_{\Phi}^+$
satisfying $p_{\Phi}^+\in(0,1].$
Then there exists a concave function $\widetilde{\Phi}$ with the same types as
$\Phi$, which is equivalent to $\Phi$.
\end{lemma}
\begin{proof}
Consider the function
$$\widetilde{\Phi}(t):=\begin{cases}
\displaystyle\int_0^t \inf_{\tau\in(0,s)}\frac{\Phi(\tau)}{\tau}\,ds,\quad &t \in (0,\infty],\\
\ 0,\quad &t=0.
\end{cases}
$$
Then it is easy to prove that $\widetilde{\Phi}$ is concave on $[0,\infty)$. By the
assumption that $p_{\Phi}^+\in(0,1]$,
we know that, for any $t\in[0,\infty)$,
$$\widetilde{\Phi}(t)\geq t\inf_{\tau\in(0,t)}\frac{\Phi(\tau)}{\tau}\gtrsim t\inf_{\tau\in(0,t)}\lf(\frac{\tau}{t}\r)^{p_\Phi^+}\frac{\Phi(t)}{\tau}\sim \Phi(t).$$
On the other hand, for any $t\in[0,\infty)$, we have
$$
\widetilde{\Phi}(t)=\int_{0}^{t}\inf_{\tau\in(0,s)}\frac{\Phi(\tau)}{\tau}\,ds
\lesssim\frac{\Phi(t)}{t^{p_\Phi^-}}\int_{0}^{t}\inf_{\tau\in(0,s)}\frac{1}{\tau^{1-p_\Phi^-}}\,ds
\sim\frac{\Phi(t)}{t^{p_\Phi^-}}\int_{0}^{t}\frac{1}{s^{1-p_\Phi^-}}\,ds\sim\Phi(t).
$$
Thus, we obtain $\Phi\sim\widetilde{\Phi}$. Moreover, it is easy to prove that $\widetilde{\Phi}$ is an Orlicz function with positive lower type $p_{\Phi}^-$ and positive upper type $p_{\Phi}^+$
satisfying $p_{\Phi}^+\in(0,1]$, which completes the proof of Lemma \ref{djdj}.
\end{proof}

\begin{remark}\label{concave}
Observe that all the results of this article are invariant under the change of equivalent
Orlicz functions. By this and Lemma \ref{djdj}, without loss of generality,
in this section, we may always assume that an Orlicz function with positive lower type $p_{\Phi}^-$ and positive upper type $p_{\Phi}^+$
satisfying $p_{\Phi}^+\in(0,1]$ is also concave.
\end{remark}

\begin{lemma}\label{quasi}
Let $t\in(0,\infty)$, $q\in(0,1]$ and $\Phi$ be an Orlicz function with positive lower type $p_{\Phi}^-$ and positive upper type $p_{\Phi}^+\in(0,1]$.
Then there exists a nonnegative constant $C$ such that, for any sequence $\{f_j\}_{j\in\nn}\subset
(E_\Phi^{q})_t(\rn)$ of nonnegative functions
such that $\sum_{j\in\mathbb{N}}f_j$ converges
in $(E_\Phi^{q})_t(\rn),$
$$
\lf\|\sum_{j\in\mathbb{N}}f_j\r\|_{(E_\Phi^{q})_t(\rn)}
\ge C\sum_{j\in\mathbb{N}}\lf\|f_j\r\|_{(E_\Phi^{q})_t(\rn)}.
$$
\end{lemma}

\begin{proof}
By Lemma \ref{djdj}, we know that there exists a concave function $\widetilde{\Phi}$ with same types of $\Phi$, which is equivalent to $\Phi$. Thus, for any $f\in(E_\Phi^{q})_t(\rn)$,
$$
\lf\|f\r\|_{(E_\Phi^{q})_t(\rn)}\sim\lf\|f\r\|_{(E_{\widetilde{\Phi}}^{q})_t(\rn)}
$$
and, to prove this lemma, by the Levi theorem,
we only need to show that, for any nonnegative $f_1,\ f_2\in(E_\Phi^{q})_t(\rn)$,
$$
\lf\|f_1+f_2\r\|_{(E_{\widetilde{\Phi}}^{q})_t(\rn)}\geq\lf\|f_1\r\|_{(E_{\widetilde{\Phi}}^{q})_t(\rn)}+
\lf\|f_2\r\|_{(E_{\widetilde{\Phi}}^{q})_t(\rn)}.
$$
Fix $x\in\rn$ and let $a_1,\ a_2\in\rr$ satisfy
$a_1 \in(0,\|f_1\|_{L^{\wz\Phi}(B(x,t))})$ and $a_2\in (0,\|f_2\|_{L^{\wz\Phi}(B(x,t))})$.
Since $\widetilde{\Phi}$ is concave, it follows that
\begin{align*}
\int_{B(x,t)}\widetilde{\Phi}\lf(\frac{f_1+f_2}{a_1+a_2}\r)\,dx
&=\int_{B(x,t)}\widetilde{\Phi}\lf(\frac{f_1}{a_1}\frac{a_1}{a_1+a_2}+\frac{f_2}{a_2}\frac{a_2}{a_1+a_2}\r)\,dx\\
&\geq\frac{a_1}{a_1+a_2}\int_{B(x,t)}\widetilde{\Phi}\lf(\frac{f_1}{a_1}\r)\,dx
+\frac{a_2}{a_1+a_2}\int_{B(x,t)}\widetilde{\Phi}\lf(\frac{f_2}{a_2}\r)\,dx\\
&>\frac{a_1}{a_1+a_2}+\frac{a_2}{a_1+a_2}=1.
\end{align*}
Thus,
$$
\lf\|f_1+f_2\r\|_{L^{\wz\Phi}(B(x,t))}\geq a_1+a_2,
$$
which further implies that
\begin{equation*}
\lf\|f_1+f_2\r\|_{L^{\wz\Phi}(B(x,t))}\geq
\lf\|f_1\r\|_{L^{\wz\Phi}(B(x,t))}+\lf\|f_2\r\|_{L^{\wz\Phi}(B(x,t))}.
\end{equation*}
From this and the definition of $(E_{\widetilde\Phi}^{q})_t(\rn)$, it easily follows that
\begin{align*}
\lf\|f_1+f_2\r\|_{(E_{\widetilde{\Phi}}^{q})_t(\rn)}\geq\lf\|f_1\r\|_{(E_{\widetilde{\Phi}}^{q})_t(\rn)}+
\lf\|f_2\r\|_{(E_{\widetilde{\Phi}}^{q})_t(\rn)},
\end{align*}
which completes the proof of Lemma \ref{quasi}.
\end{proof}

\begin{lemma}\label{atom}
Let $t\in(0,\infty)$, $q\in(0,1]$ and $\Phi$ be an Orlicz function with positive lower type $p_{\Phi}^-$ and positive upper type $p_{\Phi}^+\in(0,1]$. Let $r\in(1,\infty]$, $s\in (0,\min\{p^-_{\Phi},q\})$ and $d\in\mathbb{Z}_+$ satisfying $d\ge\lfloor n(\frac{1}{s}-1)\rfloor$.
Suppose $L$ is a continuous linear functional on $(HE_\Phi^q)_t(\rn)=(HE_\Phi^q)_t^{r,d}(\rn)$ . Then
\begin{align*}
\|L\|_{((HE_\Phi^q)_t^{r,d}(\rn))^*}
:=&\sup\lf\{|Lf|:\ \|f\|_{(HE_\Phi^q)_t^{r,d}(\rn)}\leq1\r\}\\
\sim&\sup\lf\{|La|:\ a\ is\ an\ ((E_\Phi^q)_t(\rn),\ r,\ d)\text{-atom}\r\}
\end{align*}
with the equivalent positive constant independent of $L$ and $t$.
\end{lemma}

\begin{proof}
Observing that any $((E_\Phi^q)_t(\rn),\ r,\ d)$-atom $a$ satisfies $\|a\|_{(HE_\Phi^q)_t^{r,d}(\rn)}\leq1$,
to prove this lemma, we only need to show
\begin{equation}\label{85}
\sup\lf\{|Lf|:\ \|f\|_{(HE_\Phi^q)_t^{r,d}(\rn)}\leq1\r\}
\lesssim\sup\lf\{|La|:\ a\ \text{is a}\ ((E_\Phi^q)_t(\rn),\ r,\ d)\text{-atom}\r\}.
\end{equation}
Take any $f\in(HE_\Phi^q)_t(\rn)$ and $\|f\|_{(HE_\Phi^q)_t^{r,d}(\rn)}\leq1$, which is reasonable by
Theorem \ref{atom ch}.
By Definition \ref{atomic hardy}, we know that, for any $\epsilon\in(0,\infty)$, there exist a sequence of $((E_\Phi^q)_t(\rn),\ r,\ d)$-atoms
and a sequence $\{\lambda_j\}_{j=1}^\infty\subset[0,\infty)$
such that
$f=\sum_{j=1}^{\infty}\lambda_j a_j$  in $\mathcal{S}'(\rn)$
and
$$
\lf\|\lf\{\sum_{j=1}^{\infty} \lf[\frac{\lambda_j}
{\|\chi_{Q_j}\|_{(E_\Phi^q)_t(\rn)}} \r]^s\chi_{Q_j} \r\}
^{\frac{1}{s}}\r\|_{(E_\Phi^q)_t(\rn)}\leq1+\epsilon.
$$
Combining this and Lemma \ref{quasi}, we have
\begin{align}\label{52}
\sum_{j=1}^{\infty}|\lambda_j|
&\lesssim\lf\|\sum_{j=1}^{\infty} \frac{\lambda_j}
{\|\chi_{Q_j}\|_{(E_\Phi^q)_t(\rn)}} \chi_{Q_j}
\r\|_{(E_\Phi^q)_t(\rn)}\\ \noz
&\lesssim\lf\|\lf\{\sum_{j=1}^{\infty} \lf[\frac{\lambda_j}
{\|\chi_{Q_j}\|_{(E_\Phi^q)_t(\rn)}} \r]^s\chi_{Q_j} \r\}
^{\frac{1}{s}}\r\|_{(E_\Phi^q)_t(\rn)}
\lesssim1+\epsilon.
\end{align}
Observe that, by Proposition \ref{shoulian}, we know that $f=\sum_{j=1}^{\infty}\lambda_j a_j$
holds true in $(HE_\Phi^q)_t(\rn)$. From this and \eqref{52}, it follows that
$$
|Lf|\leq\sum_{j=1}^{\infty}|\lambda_j||La_j|\lesssim(1+\epsilon)
\sup\lf\{|La|:\ a\ \mathrm{is\ an}\ ((E_\Phi^q)_t(\rn),\ r,\ d)\text{-atom}\r\}.
$$
Letting $\epsilon\rightarrow 0^+$, we then obtain \eqref{85},
which completes the proof of Lemma \ref{atom}.
\end{proof}

\begin{definition}\label{decamp}
Let $t\in(0,\infty)$, $q\in(0,1]$ and $\Phi$ be an Orlicz function with positive lower type $p_{\Phi}^-$ and positive upper type $p_{\Phi}^+\in(0,1]$. Let $r\in[1,\infty)$, $s\in(0,\min\{p^-_{\Phi},q\})$ and $d\in\mathbb{Z}_+$ satisfying
$d\ge\lfloor n(\frac{1}{s}-1)\rfloor$.
The \emph{Campanato space} $\mathcal{L}_{\Phi,t}^{q,r,d}(\rn)$
is defined to be the space of all locally $L^r(\rn)$ functions $g$
such that
$$
\|g\|_{\mathcal{L}_{\Phi,t}^{q,r,d}(\rn)}:=\sup_{B\subset\rn}\inf_{P\in\mathcal{P}_d(\rn)}
\frac{|B|}{\|\chi_{B}\|_{(E_\Phi^q)_t(\rn)}}\lf[\frac{1}{|B|}\int_{B}|g(x)-P(x)|^r\,dx\r]^{\frac{1}{r}}<\infty,
$$
where the first supremum is taken over all the balls $B\subset\rn$ and
$\mathcal{P}_d(\rn)$ denotes the space of all polynomials on $\rn$ with order not greater than $d$.
\end{definition}

As usual, by a little abuse of notation, we identify $f\in\mathcal{L}_{\Phi,t}^{q,r,d}(\rn)$ with an equivalent class $f+\mathcal{P}_d(\rn)$.

\begin{theorem}\label{dual2}
Let $t\in(0,\infty)$, $q\in(0,1]$ and $\Phi$ be an Orlicz function with positive lower type $p_{\Phi}^-$ and positive upper type $p_{\Phi}^+\in(0,1]$. Let $r\in(1,\infty]$, $s \in(0,\min\{p^-_{\Phi},q\})$ and $d\in\mathbb{Z}_+$ satisfying $d\ge\lfloor n(\frac{1}{s}-1)\rfloor$.
Then the dual space of $(HE_\Phi^q)_t(\rn)$, denoted by $((HE_\Phi^q)_t(\rn))^*$, is $\mathcal{L}_{\Phi,t}^{q,r',d}(\rn)$
in the following sense:
\begin{enumerate}
\item[{\rm(i)}] Any
$g\in\mathcal{L}_{\Phi,t}^{q,r',d}(\rn)$ induces a linear
functional given by
\begin{equation}\label{linf}
L_g:\ f\mapsto\ L_g(f):=\int_{\rn}f(x)g(x)dx,
\end{equation}
which is initially defined on
$(HE_\Phi^{q,r,d})_t^{\fin}(\rn)$ and  has a bounded
extension to $(HE_\Phi^q)_t(\rn)$.

\item[{\rm(ii)}] Conversely, any continuous linear
functional on $(HE_\Phi^q)_t(\rn)$ is of the form \eqref{linf}
for a unique $g\in\mathcal{L}_{\Phi,t}^{q,r',d}(\rn)$.
\end{enumerate}
Moreover, in any case, $\|g\|_{\mathcal{L}_{\Phi,t}^{q,r',d}(\rn)}$ is
equivalent to $\|L_g\|_{((HE_\Phi^q)_t(\rn))^*}$ with the equivalent positive constants independent of $t$,
here and hereafter, $\|\cdot\|_{((HE_\Phi^q)_t(\rn))^*}$ denotes the norm of $((HE_\Phi^q)_t(\rn))^*$.
\end{theorem}

\begin{remark} Let $t\in(0,\infty)$.
\begin{itemize}
\item[(i)] Let $q\in(0,1],\ r\in[1,\infty)$ and $\Phi(\tau):=\tau^q$
for any $\tau\in[0,\infty)$. In this case, via some simple computations, we know that,
for any ball $B\subset\rn$, $\|\chi_B\|_{(E_\Phi^q)_t(\rn)}=|B|^\frac1q$.  Thus,
in this case, $\mathcal{L}_{\Phi,t}^{q,r,d}(\rn)$ coincides with the classical
Campanato space $L_{\frac1q-1,r,d}(\rn)$ which was introduced by Campanato \cite{c}.

\item[(ii)]Let $q\in(0,1)$ and $\Phi(\tau):=\tau^q$
for any $\tau\in[0,\infty)$. In this case,
we have $p^{-}_\Phi=q$, $(E_\Phi^q)_t(\rn)=L^q(\rn)$
and $(HE_\Phi^q)_t(\rn)=H^q(\rn)$, and the best known range of
$r$ in Theorem \ref{dual2} is $[1,\infty]$
(see, for example, \cite[Theorem 4.1]{l}). However,
it is still unclear whether or not
Theorem \ref{glamda}
still holds true when $r=1$ and $\max\{p_{\Phi}^+,\ q\}\in(0,1)$.
\end{itemize}
\end{remark}

\begin{lemma}\label{bbbb}
Let $t,\ q\in(0,\infty)$ and $\Phi$ be an Orlicz function with positive lower type $p_{\Phi}^-$ and positive upper type $p_{\Phi}^+$.
Let $x_0\in\rn$ and $r\in(0,\infty)$.
Then $Q(x_0,\frac{2r}{\sqrt{n}})\subset B(x_0,r)\subset Q(x_0,2r)$ and
there exists a positive constant $C$, independent of $t$, $x_0$ and $r$, such that
\begin{equation*}
\lf\|\chi_{Q(x_0,\frac{2r}{\sqrt{n}})}\r\|_{(E_\Phi^q)_t(\rn)}\leq
\lf\|\chi_{B(x_0,r)}\r\|_{(E_\Phi^q)_t(\rn)}\leq\lf\|\chi_{Q(x_0,2r)}\r\|_{(E_\Phi^q)_t(\rn)}
\leq C\lf\|\chi_{Q(x_0,\frac{2r}{\sqrt{n}})}\r\|_{(E_\Phi^q)_t(\rn)}.
\end{equation*}
\end{lemma}

\begin{proof}
Obviously, for any $x_0\in\rn$ and $r\in(0,\infty)$, we have
$$Q\lf(x_0,\frac{2r}{\sqrt{n}}\r)\subset B\lf(x_0,r\r)\subset Q\lf(x_0,2r\r)$$
and
$$
\lf\|\chi_{Q(x_0,\frac{2r}{\sqrt{n}})}\r\|_{(E_\Phi^q)_t(\rn)}\leq
\lf\|\chi_{B(x_0,r)}\r\|_{(E_\Phi^q)_t(\rn)}\leq\lf\|\chi_{Q(x_0,2r)}\r\|_{(E_\Phi^q)_t(\rn)}.
$$
Thus, to complete the proof of Lemma \ref{bbbb},
we only need to show that, for any $x_0\in\rn$ and $r\in(0,\infty)$,
$$
\lf\|\chi_{Q(x_0,2r)}\r\|_{(E_\Phi^q)_t(\rn)}\lesssim\lf\|\chi_{Q(x_0,\frac{2r}{\sqrt{n}})}\r\|_{(E_\Phi^q)_t(\rn)}.
$$
Assume that $r\in(0,\infty)$ and, without loss of generality,
we may assume that $x_0=\vec{0}_n$.
Then it is easy to see that there exist $M\in\mathbb{N}$ and $\{x_1, \ldots, x_M\}\subset\rn$, independent
of $t$ and $r$,
such that $M\lesssim1$ and $Q(\vec{0}_n,2r)\subseteq\bigcup_{m=1}^MQ(x_m,\frac{2r}{\sqrt{n}})$,
which implies that
\begin{equation}\label{1111}
\lf\|\chi_{Q(\vec{0}_n,2r)}\r\|_{(E_\Phi^q)_t(\rn)}\leq
\lf\|\sum_{m=1}^M\chi_{Q(x_m,\frac{2r}{\sqrt{n}})}\r\|_{(E_\Phi^q)_t(\rn)}\lesssim
\sum_{m=1}^M\lf\|\chi_{Q(x_m,\frac{2r}{\sqrt{n}})}\r\|_{(E_\Phi^q)_t(\rn)}.
\end{equation}
Observing that, for any $m\in\mathbb{N}$, $t\in(0,\infty)$ and $x\in\rn$, $$\lf\|\chi_{Q(x_m,\frac{2r}{\sqrt{n}})}\chi_{B(x,t)}\r\|_{L^\Phi(\rn)}
=\lf\|\chi_{B(\vec{0}_n,\frac{2r}{\sqrt{n}})}\chi_{B(x-x_m,t)}\r\|_{L^\Phi(\rn)},$$
by this and \eqref{qiu} with $\widetilde{C}_{(\Phi,t)}$ as therein, we have
\begin{align*}
\lf\|\chi_{Q(x_m,\frac{2r}{\sqrt{n}})}\r\|_{(E_\Phi^q)_t(\rn)}
&=\lf\{\int_{\rn}\lf[\frac{\|\chi_{Q(\vec{0}_n,\frac{2r}{\sqrt{n}})}\chi_{B(x-x_m,t)}\|_{L^\Phi(\rn)}}
{\|\chi_{B(x,t)}\|_{L^\Phi(\rn)}}\r]^q\,dx\r\}^{\frac{1}{q}}\\
&=\frac{1}{\widetilde{C}_{(\Phi,t)}}\lf\{\int_{\rn}\lf\|
\chi_{Q(\vec{0}_n,\frac{2r}{\sqrt{n}})}\chi_{B(x-x_m,t)}\r\|_{L^\Phi(\rn)}
^q\,dx\r\}^{\frac{1}{q}}\\
&=\frac{1}{\widetilde{C}_{(\Phi,t)}}\lf\{\int_{\rn}\lf\|
\chi_{Q(\vec{0}_n,\frac{2r}{\sqrt{n}})}\chi_{B(x,t)}\r\|_{L^\Phi(\rn)}
^q\,dx\r\}^{\frac{1}{q}}
=\lf\|\chi_{Q(\vec{0}_n,\frac{2r}{\sqrt{n}})}\r\|_{(E_\Phi^q)_t(\rn)},
\end{align*}
which, combined with $\eqref{1111}$, implies that $\|\chi_{Q(\vec{0}_n,2r)}\|_{(E_\Phi^q)_t(\rn)}\lesssim\|\chi_{Q(\vec{0}_n,\frac{2r}{\sqrt{n}})}
\|_{(E_\Phi^q)_t(\rn)}.$
Thus, for any $x_0\in\rn$ and $r\in(0,\infty)$,
$$
\lf\|\chi_{Q(x_0,2r)}\r\|_{(E_\Phi^q)_t(\rn)}\lesssim\lf\|\chi_{Q(x_0,\frac{2r}{\sqrt{n}})}\r\|_{(E_\Phi^q)_t(\rn)}.
$$
This finishes the proof Lemma \ref{bbbb}.
\end{proof}

\begin{proof}[Proof of Theorem \ref{dual2}]
We first show (i).
By Theorem \ref{atom ch}, to prove
$\mathcal{L}_{\Phi,t}^{q,r',d}(\rn)\subset((HE_\Phi^q)_t(\rn))^*$, it suffices
to show
$$
\mathcal{L}_{\Phi,t}^{q,r',d}(\rn)\subset((HE_\Phi^q)_t^{r,d}(\rn))^*.
$$
Let $g\in\mathcal{L}_{\Phi,t}^{q,r',d}(\rn)$
and $a$ be an $((E_\Phi^q)_t(\rn),\ r,\ d)$-atom
supported on a cube $Q\subset\rn$. Let the ball
$B\subset\rn$ such that $Q\subset B$ and $|Q|\sim|B|$.
Then, by the
moment and the size conditions of $a$, together with the H\"older inequality and Lemma \ref{bbbb}, we know
that
\begin{align*}
|L_g(a)|:&=\lf|\int_{\rn}a(x)g(x)\,dx\r|
=\inf_{P\in \cp_d(\rn)}\lf|\int_{\rn}a(x)\lf[g(x)-P(x)\r]\,dx\r|\\
&\leq\|a\|_{L^r(\rn)}\inf_{P\in \cp_d(\rn)}\lf[\int_{\rn}|g(x)-P(x)|^{r'}\,dx\r]^{\frac{1}{r'}}\\
&\leq\frac{|Q|^{\frac{1}{r}}}{\|\chi_{Q}\|_{(E_\Phi^q)_t(\rn)}}
\inf_{P\in \cp_d(\rn)}\lf[\int_{\rn}|g(x)-P(x)|^{r'}\,dx\r]^{\frac{1}{r'}}\\
&\sim\frac{|B|^{\frac{1}{r}}}{\|\chi_{B}\|_{(E_\Phi^q)_t(\rn)}}
\inf_{P\in \cp_d(\rn)}\lf[\int_{\rn}|g(x)-P(x)|^{r'}\,dx\r]^{\frac{1}{r'}}
\lesssim\|g\|_{\mathcal{L}_{\Phi,t}^{q,r',d}(\rn)}.
\end{align*}
Moreover, for any $f\in(HE_\Phi^{q,r,d})_t^{\fin}(\rn)$, by Definition \ref{definite}, we know that there exist
a sequence $\{a_j\}_{j=1}^m$ of $((E_\Phi^q)_t(\rn),\ r,\ d)$-atoms supported, respectively, on the cubes
$\{Q_j\}_{j=1}^m$ and a sequence $\{\lambda_j\}_{j=1}^m\subset[0,\infty)$
such that
$$
\lf\|\lf\{\sum_{j=1}^{m} \lf[\frac{\lambda_j}
{\|\chi_{Q_j}\|_{(E_\Phi^q)_t(\rn)}} \r]^s\chi_{Q_j} \r\}
^{\frac{1}{s}}\r\|_{(E_\Phi^q)_t(\rn)}
\lesssim\|f\|_{(HE_\Phi^{q,r,d})_t^{\fin}(\rn)}.
$$
From this and Lemma \ref{quasi}, it follows that
\begin{align*}
|L_g(f)|:&=\lf|\int_{\rn}f(x)g(x)\,dx\r|
\leq\sum_{j=1}^{m}\lambda_j\lf|\int_{\rn}a_j(x)g(x)\,dx\r|\\
&\lesssim\sum_{j=1}^{m}\lambda_j\|g\|_{\mathcal{L}_{\Phi,t}^{q,r',d}(\rn)}
\lesssim\lf\|\sum_{j=1}^{m} \frac{\lambda_j}
{\|\chi_{Q_j}\|_{(E_\Phi^q)_t(\rn)}} \chi_{Q_j}
\r\|_{(E_\Phi^q)_t(\rn)}\|g\|_{\mathcal{L}_{\Phi,t}^{q,r',d}(\rn)}\\
&\lesssim\lf\|\lf\{\sum_{j=1}^{m} \lf[\frac{\lambda_j}
{\|\chi_{Q_j}\|_{(E_\Phi^q)_t(\rn)}} \r]^s\chi_{Q_j} \r\}
^{\frac{1}{s}}\r\|_{(E_\Phi^q)_t(\rn)}
\|g\|_{\mathcal{L}_{\Phi,t}^{q,r',d}(\rn)}\\
&\lesssim\|f\|_{(HE_\Phi^{q,r,d})_t^{\fin}(\rn)}
\|g\|_{\mathcal{L}_{\Phi,t}^{q,r',d}(\rn)}.
\end{align*}
By this and the fact that $(HE_\Phi^{q,r,d})_t^{\fin}(\rn)$
is dense in $(HE_\Phi^q)_t(\rn)$ as well as Theorem
\ref{finite},  we obtain  Theorem \ref{dual2}(i).

As for (ii), for any ball $B\subset\rn$,  let $\pi_B :\ L^1(B)\rightarrow \mathcal{P}_d(\rn)$ be
the natural projection
such that, for any $f\in L^1(B)$ and $Q\in\mathcal{P}_d(\rn)$,
$$
\int_{B}\pi_B(f)(x)Q(x)\,dx=\int_{B}f(x)Q(x)\,dx.
$$
It is well known (see, for example, \cite[p.\,51,\,(8.9)]{b} or \cite[p.\,54,\,Lemma 4.1]{l}) that
\begin{equation}\label{5.3}
\sup_{x\in B}|\pi_Bf(x)|\lesssim \frac{1}{|B|}\int_{B}|f(x)|\,dx.
\end{equation}

For any $r\in(1,\infty]$ and ball $B\subset\rn$, we define the \emph{closed subspace $L^r_0(B)$} of $L^r(B)$
by setting
$$
L^r_0(B): = \lf\{f\in L^r(B):\ \pi_Bf = 0\r\}.
$$
Notice that $L^r(B)$ is the subspace
of $L^r(\rn)$ consisting
of all measurable functions vanishing outside $B$.
Thus, if
$f\in L^r_0(B)$,
then
$\frac{|Q|^{\frac{1}{r }}}{\|\chi_Q\|_{(E_\Phi^q)_t(\rn)}}\|f\|_{L^r(\rn)}^{-1}f$
is an $((E_\Phi^q)_t(\rn),\ r,\ d)$-atom, where $Q$ is a cube, $Q\supset B$ and the side length of $Q$
equals to 2 times radius of $B$.

Suppose now $L \in((HE_\Phi^q)_t(\rn))^* = ((HE_\Phi^q)_t^{r,d}(\rn))^*$. By Lemma \ref{atom},
we know that, for any $f\in L^r_0(B)$,
\begin{equation}\label{810}
|L(f)| \lesssim\frac{\|\chi_Q\|_{(E_\Phi^q)_t(\rn)}}{|Q|^{\frac{1}{r }}}
\|L\|_{((HE_\Phi^q)_t(\rn))^*}\|f\|_{L^r(\rn)}.
\end{equation}
Therefore, $L$ provides a bounded linear functional on $L^r_0(B)$.
Thus, by the Hahn-Banach theorem (see, for example, \cite[p.\,106, Theorem 1]{y}),
we know that there exists a linear functional
$L_B$, which extends $L$ to the whole space $L^r(B)$ without increasing its
norm.

If $r\in(1,\infty)$, by the duality $(L^r(B))^* = L^{r'}(B)$,
we find that there exists $h_B \in L^{r'}(B)\subset L^{1}(B)
$ such that, for any $f\in L^r_0(B)$,
$$
L(f)=L_B(f)=\int_{B}f(x)h_B(x)\,dx.
$$
For the case $r =\infty$, let $\widetilde{r}\in(1,\infty)$.
By Theorems \ref{atom ch}, we know that $L \in((HE_\Phi^q)_t^{\infty,d}(\rn))^*$
implies $L\in((HE_\Phi^q)_t^{\widetilde{r},d}(\rn))^*$
without changing the norm of $L$. Thus, there exists
$h_B \in L^{\widetilde{r}'}(B)\subset L^{1}(B)$
such that, for any $f\in L^\infty_0(B)$, $L(f)=\int_{B}f(x)h_B(x)\,dx$.
Altogether, we find that,
for any $r\in(1,\infty]$,
there exists $h_B \in L^{r'}(B)$ such that, for any $f\in L^r_0(B)$,
$L(f)=\int_{B}f(x)h_B(x)\,dx.$

Next we show that such $h_B$ is unique modulo $\cp_d(\rn)$.
Indeed, assume that  $h_B'$ is another element of $L^{r'}(B)$
such that $L(f)=\int_{B}f(x)h_B'(x)\,dx$ for any $f\in L^r_0(B)$.
Then, for any $f\in L^\infty(B)$, we have $f-\pi_B(f)\in L^\infty_0 (B)$ and
\begin{align*}
0&=\int_B[f(x)-\pi_B(f)(x)][h_B(x)-h_B'(x)]\,dx\\
&=\int_Bf(x)[h_B(x)-h_B'(x)]\,dx-\int_B\pi_B(f)(x)\pi_B(h_B - h_B')(x)\,dx\\
&=\int_Bf(x)[h_B(x)-h_B'(x)]\,dx-\int_Bf(x)\pi_B(h_B-h_B')(x)\,dx\\
&=\int_B f(x)[h_B(x)-h_B'(x)-\pi_B(h_B - h_B')(x)]\,dx.
\end{align*}
The arbitrariness of $f$ implies that
$h_B(x)-h_B'(x) = \pi_B(h_B - h_B')(x)$ for almost every $x\in B$.
Therefore, after changing values
of $h_B$ (or $h_B'$) on a set of measure zero, we have $h_B - h_B'\in \cp_d(\rn)$.
Thus, the function
$h_B$ is unique up to a polynomial of degree at most $d$
regardless of the exponent
$r\in(1,\infty]$.

For any $\rho\in\mathbb{N}$, let $g_\rho$ be the unique element of $L^{r'}(B(\vec 0_n,\rho))$ such that
$L(f) =\int_{B(\vec 0_n,\rho)}f(x)g_\rho(x)\,dx$ for any $f\in L^r_0(B(\vec 0_n,\rho))$.
The preceding arguments show that $g_\rho|_{B(\vec 0_n,\ell)} = g_\ell$
for any $\ell\in\{1, \ldots, \rho\}$. Therefore, we can define a locally $L^{r'}(\rn)$
function $g$ by setting
$g(x):= g_\rho(x)$ whenever $x\in B(\vec{0}_n,\rho)$.
If $f$ is a finite linear combination of
$((E_\Phi^q)_t(\rn),\ r,\ d)$-atoms, then
$L(f) =\int_{\rn}f(x)g(x)\,dx$.
By \eqref{810}, for any ball $B\subset\rn$, the norm of $g$ as a linear functional on $L^r_0(B)$
satisfies
\begin{equation}\label{811}
\|g\|_{(L^r_0(B))^*}
\leq\frac{\|\chi_Q\|_{(E_\Phi^q)_t(\rn)}}{|Q|^{\frac{1}{r }}}
\|L\|_{((HE_\Phi^q)_t(\rn))^*}.
\end{equation}
It is known (see \cite[p.\,52, (8.12)]{b}) that
\begin{equation}\label{812}
\|g\|_{(L^r_0(B))^*}=\inf_{P\in \cp_d(\rn)}\|g-P\|_{L^{r'}(B)}.
\end{equation}
Combining \eqref{811}, \eqref{812} and Lemma \ref{bbbb}, we have
$$
\|g\|_{\mathcal{L}_{\Phi,t}^{q,r',d}(\rn)}\leq\sup_{B\subset\rn}
\frac{|B|^{\frac{1}{r}}}{\|\chi_{B}\|_{(E_\Phi^q)_t(\rn)}}\|g\|_{(L^r_0(B))^*}
\leq\|L\|_{((HE_\Phi^q)_t^{r,d}(\rn))^*}.
$$
This finishes the proof of (ii) and hence of Theorem \ref{dual2}.
\end{proof}

\section{Applications\label{s6}}

In this section, we first establish a criterion on the boundedness of sublinear operators from $(HE_\Phi^q)_t(\rn)$
into a quasi-Banach space as an application of the finite atomic characterizations of $(HE_\Phi^q)_t(\rn)$
from Theorem \ref{finite}.
Then we clarify the relationship between the atomic space $(\mathfrak{C}_r^q)_t$ introduced in
\cite{ap} and the Orlicz-slice Hardy space $(HE_\Phi^q)_t(\rn)$.
As an application of the above boundedness criterion, we obtain the boundedness
of $\delta$-type Calder\'on-Zygmund operators
on $(HE_\Phi^q)_t(\rn)$.

\subsection{Boundedness of sublinear operators\label{s5.1}}

The main purpose of this section is to establish a criterion on
the boundedness of sublinear operators from $(HE_\Phi^q)_t(\rn)$
into a quasi-Banach space.

Recall that a complete vector space is called a \emph{quasi-Banach space} $\mathcal{B}$
if its quasi-norm $\|\cdot\|_{\mathcal{B}}$
satisfies
\begin{itemize}
\item[\rm(i)]
$\|f\|_{\mathcal{B}}=0$ if and only if $f$ is the zero element of $\mathcal{B}$;

\item[\rm(ii)]
there exists a positive constant $C\in[1,\infty)$ such that, for any $f,\ g\in\mathcal{B}$,
$$\|f + g\|_{\mathcal{B}} \leq C(\|f \|_{\mathcal{B}} + \|g\|_{\mathcal{B}}).$$
\end{itemize}
Obviously, when $C=1$, a quasi-Banach space $\mathcal{B}$ is just a Banach space.
Next we recall the notion of $\gamma$-quasi-Banach spaces (see, for example,
\cite{k}, \cite{yll}, \cite{yz} and \cite{yz2}).

\begin{definition}Let $\gamma\in(0,1]$. A quasi-Banach space $\mathcal{B}_\gamma$ with the quasi-norm $\|\cdot\|_{\mathcal{B}_\gamma}$
called a \emph{$\gamma$-quasi-Banach space} if there exists a positive constant $\kappa\in[1,\infty)$ such that,
for any $m\in\mathbb{N}$ and $\{f_j\}_{j=1}^{m}\subset\mathcal{B}_\gamma$,
$$
\lf\|\sum_{j=1}^{m}f_j\r\|_{\mathcal{B}_\gamma}^\gamma
\le\kappa\sum_{j=1}^{m}\lf\|f_j\r\|^\gamma_{\mathcal{B}_\gamma}.
$$
\end{definition}

For any given $\gamma$-quasi-Banach space $\mathcal{B}_\gamma$, with $\gamma\in(0,1]$, and a linear space
$\mathcal{V}$, an operator $T$ from $\mathcal{V}$ to $\mathcal{B}_\gamma$ is said to be
\emph{$\mathcal{B}_\gamma$-sublinear} if there exists a positive constant $\kappa\in[1,\infty)$
such that
\begin{enumerate}
\item[{\rm(i)}] for any $f,\ g\in\mathcal{V}$,
    $\|T(f)-T(g)\|_{\mathcal{B}_\gamma}\le\kappa\|T(f-g)\|_{\mathcal{B}_\gamma}$;

\item[{\rm(ii)}] for any $m\in\mathbb{N}$, $\{f_j\}_{j=1}^{m}\subset\mathcal{V}$
 and $\{\lambda_j\}_{j=1}^{m}\subset\mathbb{C}$,
$$
\lf\|T\lf(\sum_{j=1}^{m}\lambda_jf_j\r)\r\|_{\mathcal{B}_\gamma}^\gamma\le
\kappa\sum_{j=1}^{m}|\lambda_j|^{\gamma}\|T(f_j)\|_{\mathcal{B}_\gamma}^\gamma.
$$
\end{enumerate}

\begin{theorem}\label{suanzi1}
Let $t,\ q\in(0,\infty)$, $\gamma\in(0,1]$, $\Phi$ be an Orlicz function with positive lower type $p_{\Phi}^-$ and positive upper type $p_{\Phi}^+$ and $\mathcal{B}_\gamma$ a $\gamma$-quasi-Banach space. Let $r\in(\max\{1,\ q,\ p_{\Phi}^+\},\infty]$, $s\in (0,\min\{p^-_{\Phi},q,1\})$  and $d\in\mathbb{Z}_+$ satisfying
$d\ge\lfloor n(\frac{1}{s}-1)\rfloor$.
If either of the following two statements holds true:
\begin{enumerate}
\item[{\rm(i)}] $r\in(\max\{1,\ q,\ p_{\Phi}^+\},\infty)$ and
$T:\ (HE_\Phi^{q,r,d})_t^{\fin}(\rn)\rightarrow\mathcal{B}_\gamma$ is a $\mathcal{B}_\gamma$-sublinear
operator satisfying that there exists a positive constant $C_1$ such that, for any $f\in(HE_\Phi^{q,r,d})_t^{\fin}(\rn)$
\begin{equation}\label{61}
\|T(f)\|_{\mathcal{B}_\gamma}\le C_1\|f\|_{(HE_\Phi^{q,r,d})_t^{\fin}(\rn)};
\end{equation}
\item[{\rm(ii)}] $T:\ (HE_\Phi^{q,\infty,d})_t^{\fin}(\rn)\cap\mathcal{C}(\rn)\rightarrow\mathcal{B}_\gamma$
is a $\mathcal{B}_\gamma$-sublinear
operator satisfying that there exists a positive constant $C_2$ such that, for any $f\in(HE_\Phi^{q,\infty,d})_t^{\fin}(\rn)\cap\mathcal{C}(\rn)$
$$
\|T(f)\|_{\mathcal{B}_\gamma}\le C_2\|f\|_{(HE_\Phi^{q,\infty,d})_t^{\fin}(\rn)},
$$
\end{enumerate}
then $T$ uniquely extends to a bounded $\mathcal{B}_\gamma$-sublinear operator from $(HE_\Phi^q)_t(\rn)$ into $\mathcal{B}_\gamma$.
Moreover, there exists a positive constant $\widetilde{C}$ such that, for any $f\in(HE_\Phi^q)_t(\rn)$,
$$
\|T(f)\|_{\mathcal{B}_\gamma}\le\widetilde{C}\|f\|_{(HE_\Phi^q)_t(\rn)}.
$$
\end{theorem}

By Theorem \ref{suanzi1}, we easily obtain the following corollary, which is a variant of Meda et al.
\cite[Corollary 3.4]{msv} and Grafakos et al. \cite[Theorem 5.9]{gly} as well as Ky \cite[Theorem 3.5]{k}
(see also \cite[Theorem 1.6.9]{yll}), the details being omitted.

\begin{corollary}\label{suanzi}
Let $t\in(0,\infty)$, $q\in(0,1]$, $\gamma\in(0,1]$, $\Phi$ be an Orlicz function with positive lower type $p_{\Phi}^-$ and positive upper type $p_{\Phi}^+\in(0,1]$ and $\mathcal{B}_\gamma$ a $\gamma$-quasi-Banach space. Let $r\in(1,\infty]$, $s\in (0,\min\{p^-_{\Phi},q\})$  and $d\in\mathbb{Z}_+$ satisfying
$d\ge\lfloor n(\frac{1}{s}-1)\rfloor$.
If either of the following two statements holds true:
\begin{enumerate}
\item[{\rm(i)}] $r\in(1,\infty)$ and
$T:\ (HE_\Phi^{q,r,d})_t^{\fin}(\rn)\rightarrow\mathcal{B}_\gamma$ is a $\mathcal{B}_\gamma$-sublinear
operator satisfying
$$
A:=\sup\lf\{\|Ta\|_{\mathcal{B}_\gamma}:\ a\ is\ a\ ((E_\Phi^q)_t(\rn),\ r,\ d)\text{-atom}\r\}<\infty;
$$
\item[{\rm(ii)}] $T:\ (HE_\Phi^{q,\infty,d})_t^{\fin}(\rn)\bigcap\mathcal{C}(\rn)\rightarrow\mathcal{B}_\gamma$
is a $\mathcal{B}_\gamma$-sublinear
operator satisfying
$$
A:=\sup\lf\{\|Ta\|_{\mathcal{B}_\gamma}:\ a\ is\ a\ continuous\ ((E_\Phi^q)_t(\rn),\ \infty,\ d)\text{-atom}\r\}<\infty,
$$
\end{enumerate}
then $T$ uniquely extends to a bounded $\mathcal{B}_\gamma$-sublinear operator from $(HE_\Phi^q)_t(\rn)$ into $\mathcal{B}_\gamma$.
Moreover, there exists a positive constant $\widetilde{C}$ such that, for any $f\in(HE_\Phi^q)_t(\rn)$,
$$
\|T(f)\|_{\mathcal{B}_\gamma}\le\widetilde{C}\|f\|_{(HE_\Phi^q)_t(\rn)}.
$$
\end{corollary}

\begin{remark}\label{remax2}
Let $t\in(0,\infty)$, $q\in(0,1]$ and $\Phi(\tau):=\tau^q$ for any $\tau\in[0,\infty)$.
In this case, we have $(E_\Phi^q)_t(\rn)=L^q(\rn)$ and $(HE_\Phi^q)_t(\rn)=H^q(\rn)$
and Theorem \ref{suanzi} coincides with the well-known criterion on
the boundedness of sublinear operators from $H^q(\rn)$
into a quasi-Banach space except the case $r=1$ (see, for example, \cite[Theorem 3.5]{k},
\cite[Theorem 1.6.9]{yll} and \cite[Theorem 5.9]{gly} with $X:=\rn$).
Moreover, when $q=1$, Theorem \ref{suanzi} is just \cite[Corollary 3.4]{msv}.
\end{remark}

We now prove Theorem \ref{suanzi1}.

\begin{proof}[Proof of Theorem \ref{suanzi1}]
Suppose that assumption (i) holds true and $f\in(HE_\Phi^q)_t(\rn)$. Then, by the density of $(HE_\Phi^{q,r,d})_t^{\fin}(\rn)$ in $(HE_\Phi^q)_t(\rn)$, we know
that there exists a Cauchy sequence $\{f_k\}_{k\in\mathbb{N}}\subset(HE_\Phi^{q,r,d})_t^{\fin}(\rn)$ such that
$$
\lim_{k\rightarrow\infty}\|f_k-f\|_{(HE_\Phi^q)_t(\rn)}=0.
$$
By this, (\ref{61}) and Theorem \ref{finite}(i), we conclude that, for any $k,\ l\in\mathbb{N}$, as $k,\ l\rightarrow\infty$,
$$
\lf\|T(f_k)-T(f_l)\r\|_{\mathcal{B}_\gamma}\lesssim\|T(f_k-f_l)\|_{\mathcal{B}_\gamma}
\lesssim\|f_k-f_l\|_{(HE_\Phi^{q,r,d})_t^{\fin}(\rn)}\sim\|f_k-f_l\|_{(HE_\Phi^q)_t(\rn)}\rightarrow0,
$$
which implies that $\{T(f_k)\}_{k\in\mathbb{N}}$ is a Cauchy sequence in $\mathcal{B}_\gamma$. Therefore, by the completeness of $\mathcal{B}_\gamma$,
we know that there exists some $g\in\mathcal{B}_\gamma$ such that $g=\lim_{k\rightarrow\infty}T(f_k)$ in $\mathcal{B}_\gamma$. Then let $T(f):=g$.
From this, (\ref{61}) and Theorem \ref{finite}(i) again, it is easy to deduce that $T(f)$ is well defined and,
for any $f\in(HE_\Phi^q)_t(\rn)$,
\begin{align*}
\|T(f)\|_{\mathcal{B}_\gamma}^\gamma&\lesssim
\limsup_{k\rightarrow\infty}\lf[\|T(f)-T(f_k)\|_{\mathcal{B}_\gamma}^\gamma+\|T(f_k)\|_{\mathcal{B}_\gamma}^\gamma\r]
\lesssim
\limsup_{k\rightarrow\infty}\|T(f_k)\|_{\mathcal{B}_\gamma}^\gamma\\
&\lesssim\limsup_{k\rightarrow\infty}\|f_k\|_{(HE_\Phi^{q,r,d})_t^{\fin}(\rn)}^\gamma
\sim\lim_{k\rightarrow\infty}\|f_k\|_{(HE_\Phi^q)_t(\rn)}^\gamma\sim\|f\|_{(HE_\Phi^q)_t(\rn)}^\gamma,
\end{align*}
which completes the proof of (i).

Suppose that the assumption (ii) holds true. Similarly to the proof of (i), using Theorem \ref{finite}(ii), we also conclude that, for any $f\in(HE_\Phi^{q,\infty,d})_t^{\fin}(\rn)\cap\mathcal{C}{(\rn)}$,
$\|T(f)\|_{\mathcal{B}_\gamma}\lesssim\|f\|_{(HE_\Phi^q)_t(\rn)}$. To extend $T$ to the whole
$(HE_\Phi^q)_t(\rn)$, we only need to prove that $(HE_\Phi^{q,\infty,d})_t^{\fin}(\rn)\cap\mathcal{C}{(\rn)}$
is dense in $(HE_\Phi^q)_t(\rn)$. Observing that $(HE_\Phi^{q,\infty,d})_t^{\fin}(\rn)$ is dense in $(HE_\Phi^q)_t(\rn)$,
to show this, it suffices to prove that $(HE_\Phi^{q,\infty,d})_t^{\fin}(\rn)\cap\mathcal{C}{(\rn)}$ is dense
in $(HE_\Phi^{q,\infty,d})_t^{\fin}(\rn)$ in terms of the quasi-norm $\|\cdot\|_{(HE_\Phi^q)_t(\rn)}$.
Actually, we show that $(HE_\Phi^{q,\infty,d})_t^{\fin}(\rn)\cap\mathcal{C}^{\infty}{(\rn)}$
is dense in $(HE_\Phi^{q,\infty,d})_t^{\fin}(\rn)$.

To see this, let $f\in(HE_\Phi^{q,\infty,d})_t^{\fin}(\rn)$. Since $f$ is a finite linear combination
of functions with bounded supports, it follows that there exists $R\in(0,\infty)$ such that $\supp(f)\subset B(\vec{0}_n,R)$.
Take $\varphi\in\mathcal{S}(\rn)$ such that $\supp(\varphi)\subset B(\vec{0}_n,1)$ and
$\int_\rn\varphi(x)\,dx=1$. It is easy to see that $\supp(\varphi_\tau\ast f)\subset B(\vec{0}_n,2R)$ for
any $\tau\in(0,R)$ and $\varphi_\tau\ast f$ has vanishing moments up to order $d$, where
$\varphi_\tau(x):=\tau^{-n}\varphi(\tau^{-1}x)$ for any $x\in\rn$. Thus,
$\varphi_\tau\ast f\in (HE_\Phi^{q,\infty,d})_t^{\fin}(\rn)\cap\mathcal{C}^{\infty}{(\rn)}$.

Likewise, $\supp( f-\varphi_\tau\ast f)\subset B(\vec{0}_n,2R)$ for any
$\tau\in(0,R)$ and $\varphi_\tau\ast f$ has vanishing moments up to order $d$.
Moreover, taking any $\delta\in(1,\infty)$, we have
$$
\lf\|f-\varphi_\tau\ast f\r\|_{L^\delta(\rn)}\rightarrow0\quad\mathrm{as}\quad\tau\rightarrow0.
$$
Thus, $f-\varphi_\tau\ast f=c_\tau a_\tau$ for some $((E_\Phi^q)_t(\rn),\ \delta,\ d)$-atom $a_\tau$, and
some constant $c_\tau$ which satisfies that $c_\tau\rightarrow0$ as $\tau\rightarrow0$.
Thus, $\|f-\varphi_\tau\ast f\|_{(HE_\Phi^q)_t(\rn)}\rightarrow0$ as $\tau\rightarrow0$.
This finishes the proof of Theorem \ref{suanzi1}.
\end{proof}

\subsection{Boundedness of Calder\'on-Zygmund operators\label{s5.2}}

In \cite{ap}, Auscher and Prisuelos-Arribas obtained the boundedness on slice spaces
$(E_r^q)_t(\rn)$ of operators such as the
Hardy-Littlewood maximal operator, Calder\'on-Zygmund operators  etc. Based
on $(E_r^q)_t(\rn)$,  Auscher and Prisuelos-Arribas in \cite{ap} also introduced a
Hardy-type space $(\mathfrak{C}_r^q)_t(\rn)$
and proved the boundedness of Calder\'on-Zygmund operators on it.
In this section, we first obtain the relationship between the space $(\mathfrak{C}_r^q)_t(\rn)$
and the Orlicz-slice Hardy space $(HE_\Phi^q)_t(\rn)$.
Then, using the criterion for the boundedness of sublinear operators obtained in Theorem \ref{suanzi},
we establish the boundedness of $\delta$-type Calder\'on-Zygmund operators
from $(HE_\Phi^q)_t(\rn)$ to $(HE_\Phi^q)_t(\rn)$ [or to $(E_\Phi^q)_t(\rn)$]
with $\delta\in(0,1]$ and $\min\{p_{\Phi}^+,q\}\in(\frac{n}{n+\delta},1]$, respectively.

\begin{definition}\label{auscher}
Let $t\in(0,\infty)$, $r\in(1,\infty)$ and $q\in(\frac{n}{n+1},1]$. A function $a\in(E_r^q)_t(\rn)$ is called an
\emph{$(E_r^q)_t(\rn)$-atom} if it is supported on a ball $B$ of radius $\tau\in[t,\infty)$ and satisfies
$$
\|a\|_{L^r(\rn)}\le|B|^{\frac{1}{r}-\frac{1}{q}}.
$$
The \emph{Hardy-type space $(\mathfrak{C}_r^q)_t(\rn)$} is then defined to be the set of all measurable functions $f\in(E_r^q)_t(\rn)$
such that there exist a sequence of numbers, $\{\lambda_j\}_{j=1}^\infty\in \ell^q$, and a sequence of
$(E_r^q)_t(\rn)$-atoms, $\{a_j\}_{j=1}^\infty$, supported, respectively, on the balls
$\{B_j\}_{j=1}^\infty$, with $\int_{\rn}a_j(x)dx=0$ for any $j\in\mathbb{N}$
 so that $f=\sum_{j=1}^\infty\lambda_ja_j$ with convergence in
$(E_r^q)_t(\rn)$.
\end{definition}

\begin{proposition}\label{au}
Let $t\in(0,\infty)$, $r\in(1,\infty)$, $q\in(\frac{n}{n+1},1]$ and $s\in(0,q]$. Then
$$
(\mathfrak{C}_r^q)_t(\rn)\subset H^q(\rn)\subset H^q(\rn)\cup H^s(\rn)\subset(HE_s^q)_t(\rn),
$$
where the slice Hardy space $(HE_s^q)_t(\rn)$ is as in Definition \ref{dh}.
\end{proposition}

\begin{proof}
By the definition of the space $(\mathfrak{C}_r^q)_t(\rn)$, it is obvious that the Hardy-type space
$(\mathfrak{C}_r^q)_t(\rn)$
is the subspace of the classical real Hardy space $H^q(\rn)$.
Furthermore, by Proposition \ref{ggg}(i), we know that
$(\mathfrak{C}_r^q)_t(\rn)\subset H^q(\rn)\subset H^q(\rn)\cup H^s(\rn)\subset(HE_s^q)_t(\rn)$,
which completes the proof of Proposition \ref{au}.
\end{proof}

\begin{remark}
Let $t\in(0,\infty)$, $r\in(1,\infty)$, $q\in(\frac{1}{2},1]$, $s\in(0,q]$ and $n=1$.
When $q\in(\frac{1}{2},1)$, the difference $\delta_1-\delta_{-1}$ of Dirac masses
lies in $H^q(\mathbb{R})$
but not in $(\mathfrak{C}_r^q)_t(\mathbb{R})$, because $\delta_1-\delta_{-1}$ is only a distribution,
not a function (see also, for example, \cite[p.\,129]{em}). This shows that
$(\mathfrak{C}_r^q)_t(\mathbb{R})\subsetneqq(HE_s^q)_t(\mathbb{R})$. When $q=1$
and $s\in(0,1),$ let $d=\lfloor s^{-1}-1\rfloor,$ the $d$-order derivative
$(\delta_1-\delta_{-1})^{(d)}$ of $\delta_1-\delta_{-1}$ lies in $H^s(\rr)$,
but not in $(\mathfrak{C}_r^1)_t(\mathbb{R})$; thus, in this case, we also have
$(\mathfrak{C}_r^1)_t(\mathbb{R})\subsetneqq(HE_s^1)_t(\mathbb{R})$.
\end{remark}

\begin{definition}\label{con}
Let $\delta\in(0,1]$, a \emph{convolutional $\delta$-type
Calder\'on-Zygmund operator} $\mathcal{T}$ is a
linear operator, which is bounded on $L^2(\rn)$ with kernel $k\in\mathcal{S}'(\rn)$
coinciding with a locally integrable function on
$\rn\setminus\{\vec{0}_n\}$ and satisfying that there exists a positive constant $C$, independent of
$f$, $x$ and $y$, such that, for any $x,\ y\in\rn$ with $|x|\ge|2y|$,
$$
|k(x-y)-k(x)|\le C\frac{|y|^\delta}{|x|^{n+\delta}}
$$
and, for any $f\in L^2(\rn)$, $\mathcal{T}(f):=k\ast f$.
\end{definition}

\begin{definition}\label{noncon}
Let $\delta\in(0,1]$. A \emph{ non-convolutional $\delta$-type
Calder\'on-Zygmund operator} is a linear operator which is bounded on $L^2(\rn)$ and satisfies that,
for any $f\in L^2(\rn)$ with compact support and $x\notin\supp(f)$,
$$
\mathcal{T}(f)(x):=\int_{\rn}K(x,y)f(y)\,dy,
$$
where $K$ denotes a measurable function on $(\rn\times\rn)\setminus\{(x,x):\ x\in\rn\}$
satisfying that there exists a positive constant $C$ such that, for any $x,\  y,\  z\in\rn$,
$$\lf|K(x,y)-K(x,z)\r|\le C\frac{|y-z|^\delta}{|x-y|^{n+\delta}}\quad \mathrm{when}\ |x-y|>2|y-z|.$$
\end{definition}

\begin{lemma}\label{czsuan}
Let $t,\ q\in(0,\infty)$ and $\Phi$ be an Orlicz function with positive lower type $p_{\Phi}^-$ and positive upper type $p_{\Phi}^+$. Assume that  $r\in(\max\{1,\ q,\ p_{\Phi}^+\},\infty]$, $s\in(0,\min\{p^-_{\Phi},q,1\})$. Let $\{\lambda_k\}_{k\in\mathbb{N}}\subset[0,\infty)$
and $\{Q_k\}_{k\in\mathbb{N}}$ be a sequence of cubes. Then, for any sequence
$\{a_k\}_{k\in\mathbb{N}}\subset L^r(\rn)$ such that, for any $k\in\mathbb{N}$, $\supp(a_k)\subset Q_k$,
$$
\lf\|a_k\r\|_{L^r(\rn)}\le\frac{|Q_k|^{\frac{1}{r}}}{\|\chi_{Q_k}\|_{(E_\Phi^{q})_t(\rn)}}
$$
and
$$
\lf\|\lf\{\sum_{k\in\mathbb{N}}\lf[\frac{\lambda_k}{\|\chi_{Q_k}\|_{(E_\Phi^{q})_t(\rn)}}\r]^s
\chi_{Q_k}\r\}^{\frac{1}{s}}\r\|_{(E_\Phi^{q})_t(\rn)}<\infty,
$$
it holds true that
$$
\lf\|\sum_{k\in\mathbb{N}}\lambda_ka_k\r\|_{(E_\Phi^{q})_t(\rn)}
\lesssim\lf\|\lf\{\sum_{k\in\mathbb{N}}
\lf[\frac{\lambda_k}{\|\chi_{Q_k}\|_{(E_\Phi^{q})_t(\rn)}}\r]^s\chi_{Q_k}\r\}^{\frac{1}{s}}\r\|_{(E_\Phi^{q})_t(\rn)},
$$
where the implicit positive constant is independent of $\{\lambda_k\}_{k\in\mathbb{N}}$,
$\{a_k\}_{k\in\mathbb{N}}$ and $t$.
\end{lemma}

\begin{proof}
By Lemmas \ref{ballproof}, \ref{sconvex}, \ref{man28} and \ref{man37}, we know that
$(E_\Phi^q)_t(\rn)$ satisfies all the assumptions of \cite[Theorem 2.10]{ykds}.
As a simple corollary of \cite[Theorem 2.10]{ykds}, we immediately obtain the
desired conclusion of Lemma \ref{czsuan}, which completes the proof of Lemma \ref{czsuan}.
\end{proof}

Via borrowing some ideas from the proof of Yan et al. \cite[Theorem 7.4]{yyyz} and applying the criterion
established in Theorem \ref{suanzi1}, we obtain the boundedness of convolutional $\delta$-type
 Calder\'on-Zygmund operators from $(HE_\Phi^{q})_t(\rn)$
to itself or to $(E_\Phi^{q})_t(\rn)$
(see Theorem \ref{cz} below), which extends the corresponding results of Fefferman and Stein \cite[Theorem 12]{fs72}
to the present setting.

\begin{theorem}\label{cz}
Let $t,\ q\in(0,\infty)$, $\delta\in(0,1]$ and $\Phi$ be an Orlicz function with positive
lower type $p_{\Phi}^-$ satisfying $\min\{p_{\Phi}^-,\ q\}\in(\frac{n}{n+\delta},1]$ and
positive upper type $p_{\Phi}^+$.
\begin{enumerate}
\item[{\rm(i)}] If $\mathcal{T}$ is  a convolutional $\delta$-type
 Calder\'on-Zygmund operator as in Definition \ref{con}, then there exists a positive constant
 $C$ such that, for any $f\in (HE_\Phi^q)_t(\rn)$,
 $$
 \|\mathcal{T}(f)\|_{(E_\Phi^{q})_t(\rn)}\le C\|f\|_{(HE_\Phi^{q})_t(\rn)}.
 $$
\item[{\rm(ii)}]If $\mathcal{T}$ is a convolutional $\delta$-type
 Calder\'on-Zygmund operator as in Definition \ref{con}, then there exists a positive constant
 $C$ such that, for any $f\in (HE_\Phi^q)_t(\rn)$,
 $$
 \|\mathcal{T}(f)\|_{(HE_\Phi^{q})_t(\rn)}\le C\|f\|_{(HE_\Phi^{q})_t(\rn)},
 $$
 where the positive constant $C$ is independent of $f$ and $t$.
\end{enumerate}
\end{theorem}

\begin{proof}
By similarity, we only prove (ii). Let $\mathcal{T}$ be a convolutional $\delta$-type
Calder\'on-Zygmund operator as in Definition \ref{con}.
Let $r\in(\max\{1,\ q,\ p_{\Phi}^+\},\infty)$ and $f\in(HE_\Phi^{q,r,d})_t^{\fin}(\rn)$.
Then, without loss of generality, we
may assume that $\|f\|_{(HE_\Phi^{q})_t(\rn)}=1$. Thus, to prove (ii),
by Theorem \ref{suanzi1}(i), we only need to show that
\begin{equation}\label{h2}
\|\mathcal{T}(f)\|_{(HE_\Phi^{q})_t(\rn)}\lesssim1.
\end{equation}
Noticing that $f\in(HE_\Phi^{q})_t(\rn)\bigcap L^r(\rn)$,
by the proof of \cite[Theorem 3.7]{ykds}, we know that there exist a sequence of
$\{\lambda_j\}_{j=1}^{\infty}\subset[0,\infty)$ and $\{a_j\}_{j=1}^{\infty}$ of
$((E_\Phi^q)_t(\rn),\ r,\ 0)$-atoms
supported, respectively, on the cubes $\{Q_j\}_{j=1}^{\infty}:=\{Q(x_j,r_j)\}_{j=1}^{\infty}\subset\mathcal{Q}$ such that
$f=\sum_{j=1}^{\infty}\lambda_j a_j$
converges in $L^r(\rn)$ and
$$
\lf\|\lf\{\sum_{j=1}^{\infty} \lf[\frac{\lambda_j}
{\|\chi_{Q_j}\|_{(E_\Phi^q)_t(\rn)}} \r]^s\chi_{Q_j} \r\}
^{\frac{1}{s}}\r\|_{(E_\Phi^q)_t(\rn)}\lesssim\|f\|_{(HE_\Phi^q)_t(\rn)}\lesssim1.
$$
From the fact that
$\mathcal{T}$ is bounded on $L^r(\rn)$ (see, for example, \cite[Theorem 5.1]{d}),
we deduce that
\begin{equation*}
\mathcal{T}(f)=\sum_{j=1}^{\infty}\lambda_j \mathcal{T}(a_j)
\end{equation*}
converges in $L^r(\rn)$.
Using this and Theorem \ref{mdj}, we have
\begin{align*}
\lf\|\mathcal{T}(f)\r\|_{(HE_\Phi^{q})_t(\rn)}&\sim\lf\|M\lf(\mathcal{T}(f),\varphi\r)\r\|_{(E_\Phi^q)_t(\rn)}\\
&\lesssim\lf\|\sum_{j=1}^{\infty}\lambda_jM\lf(\mathcal{T}(a_j),\varphi\r)
\chi_{4\sqrt{n}Q_j}\r\|_{(E_\Phi^q)_t(\rn)}+
\lf\|\sum_{j=1}^{\infty}\lambda_jM\lf(\mathcal{T}(a_j),\varphi\r)\chi_{\rn\setminus 4\sqrt{n}Q_j}\r\|_{(E_\Phi^q)_t(\rn)}\\
&=:\mathrm{I}+\mathrm{II},
\end{align*}
where $\varphi\in\mathcal{S}(\rn)$ satisfies $\int_{\rn}\varphi(x)\,dx\neq0$
and $M(\mathcal{T}(f),\varphi)$ is as in Definition \ref{31}.

For $\mathrm{I}$, by the boundedness of $\mathcal{T}$ on $L^r(\rn)$ and the fact that $M(\mathcal{T}(a),\varphi)\lesssim\mathcal{M}(\mathcal{T}(a))$, we conclude that,
for any $j\in\mathbb{N}$,
$$
\lf\|M\lf(\mathcal{T}(a_j),\varphi\r)\chi_{4\sqrt{n}Q_j}\r\|_{L^r(\rn)}
\lesssim\lf\|\mathcal{M}\lf(\mathcal{T}(a_j)\r)\r\|_{L^r(\rn)}\lesssim\lf\|\mathcal{T}(a_j)\r\|_{L^r(\rn)}
\lesssim\lf\|a_j\r\|_{L^r(\rn)}\lesssim\frac{|Q_j|^{\frac{1}{r}}}{\|\chi_{Q_j}\|_{(E_\Phi^q)_t(\rn)}},
$$
which, combined with Lemma \ref{czsuan}, implies that
$$
\mathrm{I}
\lesssim\lf\|\lf\{\sum_{j=1}^{\infty}\lf[\lambda_jM(\mathcal{T}(a_j),\varphi)
\chi_{4\sqrt{n}Q_j}\r]^s\r\}^{\frac{1}{s}}\r\|_{(E_\Phi^q)_t(\rn)}
\lesssim\lf\|\lf\{\sum_{j=1}^{\infty} \lf[\frac{\lambda_j}
{\|\chi_{Q_j}\|_{(E_\Phi^q)_t(\rn)}} \r]^s\chi_{Q_j} \r\}
^{\frac{1}{s}}\r\|_{(E_\Phi^q)_t(\rn)}
\lesssim1.
$$
This is a desired estimate.

As for $\mathrm{II}$, for any $\tau\in(0,\infty)$, let $k^{(\tau)}:=k\ast\varphi_\tau$
with $\varphi_\tau(\cdot):=\frac{1}{\tau^n}\varphi(\frac{\cdot}{\tau})$. By the proof
of \cite[Theorem 7.4]{yyyz}, we find that
$k^{(\tau)}$ satisfies the same conditions as $k$.

Now, by the vanishing moment condition of $a_j$ and the H\"older inequality, we know that, for any
$x\notin 4\sqrt{n}Q_j$,
\begin{align*}
\lf|M\lf(\mathcal{T}(a_j),\varphi\r)(x)\r|
&=\sup_{\tau\in(0,\infty)}\lf|\varphi_\tau\ast (k \ast a_j)(x)\r|
=\sup_{\tau\in(0,\infty)}\lf|k^{(\tau)} \ast a_j(x)\r|\\
&=\sup_{\tau\in(0,\infty)}\lf|\int_{\rn} \lf[k^{(\tau)}(x-y)-k^{(\tau)}(x-x_j) \r]a_j(y)\,dy\r|\\
&\lesssim\int_{\rn} \frac{|y-x_j|^\delta}{|x-x_j|^{n+\delta}}\lf|a_j(y)\r|\,dy
\lesssim\frac{r_j^\delta}{|x-x_j|^{n+\delta}}\|a_j\|_{L^r(\rn)}\lf|Q_j\r|^{\frac{1}{r'}}\\
&\lesssim\frac{r_j^{n+\delta}}{|x-x_j|^{n+\delta}}\frac{1}{\|\chi_{Q_j}\|_{(E_\Phi^{q})_t(\rn)}}
\lesssim\lf[\mathcal{M}\lf(\chi_{Q_j}\r)(x)
\r]^{\frac{n+\delta}{n}}\frac{1}{\|\chi_{Q_j}\|_{(E_\Phi^{q})_t(\rn)}},
\end{align*}
which implies that, for any $x\notin 4\sqrt{n}Q_j$,
$$
\lf|M\lf(\mathcal{T}(a_j),\varphi\r)(x)\r|\chi_{\rn\setminus 4\sqrt{n}Q_j}(x)
\lesssim\lf[\mathcal{M}\lf(\chi_{Q_j}\r)(x)\r]^{\frac{n+\delta}{n}}\frac{1}{\|\chi_{Q_j}\|_{(E_\Phi^{q})_t(\rn)}}.
$$
Therefore, we have
\begin{align*}
\mathrm{II}
\lesssim\lf\|\sum_{j=1}^{\infty}\frac{\lambda_j}{\|\chi_{Q_j}\|_{(E_\Phi^{q})_t(\rn)}}
\lf[\mathcal{M}\lf(\chi_{Q_j}\r)\r]^{\frac{n+\delta}{n}}\r\|_{(E_\Phi^q)_t(\rn)}.
\end{align*}
Let $u:=\frac{n}{n+\delta}$ and $\Phi_u(\tau):=\Phi(\sqrt[u]{\tau})$. Since $\min\{p_{\Phi}^-,\ q\}\in(\frac{n}{n+\delta},1]$, it follows that $\Phi_u$ is of upper type
$\frac{p_{\Phi}^+}{u}$ and of lower type $\frac{p_{\Phi}^-}{u}$, and $\frac{p_{\Phi}^-}{u}$, $\frac{q}{u}\in(1,\infty)$.
By this and Theorem \ref{main}, we further conclude that
\begin{align}\label{gg}
\mathrm{II}
&\lesssim\lf\| \lf\{ \sum_{j=1}^{\infty} \frac{\lambda_j}{\|\chi_{Q_j}\|_{(E_\Phi^{q})_t(\rn)}}
\lf[\mathcal{M}(\chi_{Q_j})\r]^{\frac{1}{u}}  \r\}^{u} \r\|_{(E_{\Phi_u}^{q/u})_t(\rn)}^{\frac{1}{u}}
\lesssim\lf\| \lf\{ \sum_{j=1}^{\infty} \frac{\lambda_j\chi_{Q_j}}{\|\chi_{Q_j}\|_{(E_\Phi^{q})_t(\rn)}}
\r\}^{u} \r\|_{(E_{\Phi_u}^{q/u})_t(\rn)}^{\frac{1}{u}}\\ \noz
&\lesssim\lf\|\lf\{\sum_{j=1}^{\infty} \lf[\frac{\lambda_j}
{\|\chi_{Q_j}\|_{(E_\Phi^q)_t(\rn)}} \r]^s\chi_{Q_j} \r\}
^{\frac{1}{s}}\r\|_{(E_\Phi^q)_t(\rn)}
\lesssim1.
\end{align}

Combining the estimates for $\mathrm{I}$ and $\mathrm{II}$,
we obtain \eqref{h2},
which completes the proof of (ii) and hence of Theorem \ref{cz}.
\end{proof}

We recall the notion of $\beta$-order
Calder\'on-Zygmund operators as follows (see, for example, \cite{yyyz}).

\begin{definition}
For any given $\beta\in(0,\infty)\setminus\mathbb{N}$, a linear operator $\mathcal{T}$ is called a \emph{$\beta$-order
Calder\'on-Zygmund operator} if $\mathcal{T}$ is bounded on $L^2(\rn)$ and its kernel
$$
k:\ (\rn\times\rn)\setminus\lf\{(x,x):\ x\in\rn\r\}\rightarrow\mathbb{C}
$$
satisfies that there exists a positive constant $C$ such that, for any $\alpha\in\mathbb{Z}_+^n$ with $|\alpha|\le\lfloor\beta\rfloor$
and $x,\ y,\ z\in\rn$ with $|x-y|>2|y-z|$,
\begin{equation}\label{71}
\lf|\partial_x^\alpha k(x,y)-\partial_x^\alpha k(x,z)\r|\le C\frac{|y-z|^{\beta-\lfloor\beta\rfloor}}{|x-y|^{n+\beta}}
\end{equation}
and, for any $f\in L^2(\rn)$ having compact support and $x\notin \supp f$,
$$
\mathcal{T}(f)(x)=\int_{\supp f}k(x,y)f(y)\,dy.
$$
\end{definition}
Here and hereafter, for any $\alpha:=(\alpha_1, \ldots, \alpha_n)\in\zz_+^n$,
$\partial_x^{\alpha}:=(\frac{\partial}{\partial x_1})
^{\alpha_1}\cdots(\frac{\partial}{\partial x_n})^{\alpha_n}$.

Next, we establish the boundedness of the $\beta$-order
Calder\'on-Zygmund operator $\mathcal{T}$ from the Orlicz-slice Hardy space $(HE_\Phi^{q})_t(\rn)$
to itself (see Theorem \ref{cz2}) or to $(E_\Phi^{q})_t(\rn)$ (see Theorem \ref{cz3}).
Recall that, for any $l\in\mathbb{N}$, an operator $\mathcal{T}$ is said to have the
\emph{vanishing moment condition up to order $l$} if, for any $a\in L^2(\rn)$ with
compact support and satisfying that, for any $\gamma\in\zz_+^n$ with $|\gamma|\leq l$,
$\int_{\rn}x^\gamma a(x)\,dx=0$, it holds true that $\int_{\rn}x^{\gamma}\mathcal{T}(a)(x)\,dx=0$.

\begin{theorem}\label{cz2}
Let $t\in(0,\infty)$, $q\in(0,2)$, $\beta\in(0,\infty)\setminus\mathbb{N}$ and $\Phi$ be an Orlicz function with positive lower type $p_{\Phi}^-$ satisfying $\min\{p_{\Phi}^-,\ q\}\in(\frac{n}{n+\beta},\frac{n}{n+\lfloor\beta\rfloor}]$
and positive upper type $p_{\Phi}^+\in(0,2)$. Let $\mathcal{T}$ be a $\beta$-order
Calder\'on-Zygmund operator and have the vanishing moment condition up to order $\lfloor\beta\rfloor$.
Then $\mathcal{T}$ has a unique extension on $(HE_\Phi^{q})_t(\rn)$ and, for any $f\in(HE_\Phi^{q})_t(\rn)$,
$$
\lf\|\mathcal{T}(f)\r\|_{(HE_\Phi^{q})_t(\rn)}\le C\|f\|_{(HE_\Phi^{q})_t(\rn)},
$$
where $C$ is positive constant independent of $f$ and $t$.
\end{theorem}

\begin{proof}
Let $\{\lambda_j\}_{j=1}^{\infty}$ and $\{a_j\}_{j=1}^{\infty}$ be the same as
in the proof of Theorem \ref{cz}. By an argument similar to that used in the proof of Theorem \ref{cz}, we know that, to prove Theorem \ref{cz2},
it suffices to show that
\begin{equation}\label{yuan}
\lf\|\sum_{j=1}^{\infty}\lambda_jM\lf(\mathcal{T}(a_j),\varphi\r)\r\|_{(E_\Phi^q)_t(\rn)}\lesssim1,
\end{equation}
where $M(\mathcal{T}(a_j),\varphi)$ is as in Definition \ref{31}.

To this end, it is easy to see that
\begin{align*}
\lf\|\sum_{j=1}^{\infty}\lambda_jM\lf(\mathcal{T}(a_j),\varphi\r)\r\|_{(E_\Phi^q)_t(\rn)}
&\lesssim\lf\|\sum_{j=1}^{\infty}\lambda_jM\lf(\mathcal{T}(a_j),\varphi\r)
\chi_{4\sqrt{n}Q_j}\r\|_{(E_\Phi^q)_t(\rn)}\\
&\quad+\lf\|\sum_{j=1}^{\infty}\lambda_jM\lf(\mathcal{T}(a_j),\varphi\r)
\chi_{\rn\setminus 4\sqrt{n}Q_j}\r\|_{(E_\Phi^q)_t(\rn)}
=:\mathrm{I}+\mathrm{II},
\end{align*}
where, for any $j\in\mathbb{N}$, $Q_j:=Q(x_j,r_j)$ is the same as in the proof of Theorem \ref{cz}.

For $\mathrm{I}$, by an argument similar to that used in the proof
 of Theorem \ref{cz}, we conclude that $\mathrm{I}\lesssim1$.

Next, we deal with $\mathrm{II}$. To this end,
from the vanishing moment condition of $\mathcal{T}$ and the fact that $\lfloor\beta\rfloor\le n(\frac{1}{\min\{q,p_{\Phi}^-\}}-1)$ implies $\lfloor\beta\rfloor\le d$, it follows that, for any $j\in\mathbb{N}$,
 $\tau\in(0,\infty)$ and $x\notin4Q_j$,
\begin{align}\label{79}
\lf|\varphi_\tau\ast\mathcal{T}(a_j)(x)\r|
&=\frac{1}{\tau^n}\lf|\int_{\rn}
\varphi\lf(\frac{x-y}{\tau}\r)\mathcal{T}(a_j)(y)\,dy\r|\\\noz
&\le\frac{1}{\tau^n}\int_{\rn}\lf|\varphi\lf(\frac{x-y}{\tau}\r)-
\sum_{|\alpha|\le\lfloor\beta\rfloor}\frac{\partial^\alpha\varphi
(\frac{x-x_j}{\tau})}{\alpha!}\lf(\frac{y-x_j}{\tau}\r)^{\alpha}\r|
\lf|\mathcal{T}(a_j)(y)\r|\,dy\\\noz
&=\frac{1}{\tau^n}\lf(\int_{|y-x_j|<2r_j}+\int_{2r_j\le|y-x_j|<
\frac{|x-x_j|}{2}}+\int_{|y-x_j|\ge\frac{|x-x_j|}{2}}\r)\\\noz
&\quad\times\lf|\varphi\lf(\frac{x-y}{\tau}\r)-
\sum_{|\alpha|\le\lfloor\beta\rfloor}\frac{\partial^\alpha
\varphi(\frac{x-x_j}{\tau})}{\alpha!}\lf(\frac{y-x_j}{\tau}\r)^{\alpha}\r|
\lf|\mathcal{T}(a_j)(y)\r|\,dy\\\noz
&=:\mathrm{II}_1+\mathrm{II}_2+\mathrm{II}_3,
\end{align}
where $\varphi\in\mathcal{S}(\rn)$ satisfying $\int_{\rn}\varphi(x)\,dx\neq0$.

For $\mathrm{II}_1$, by the Taylor remainder theorem, the H\"older inequality and the fact that $\mathcal{T}$ is bounded on $L^2(\rn)$, similarly to the estimation of \cite[(7.9)]{yyyz},
we find that, for any $\tau\in(0,\infty)$ and $x\notin 4Q_j$,
\begin{align*}
\mathrm{II}_1&\lesssim\frac{1}{\tau^n}\int_{|y-x_j|<2r_j}
\frac{\tau^{n+\lfloor\beta\rfloor+1}}{|x-x_j|^{n+\lfloor\beta\rfloor+1}}
\frac{|y-x_j|^{\lfloor\beta\rfloor+1}}{\tau^{\lfloor\beta\rfloor+1}}\lf|\mathcal{T}(a_j)(y)\r|\,dy\\
&\lesssim\frac{r_j^{\lfloor\beta\rfloor+1}}{|x-x_j|^{n+\lfloor\beta\rfloor+1}}\lf\|
\mathcal{T}a_j\r\|_{L^2(\rn)}\lf|Q_j\r|^{\frac{1}{2}}
\lesssim\frac{r_j^{n+\lfloor\beta\rfloor+1}}{|x-x_j|^{n+\lfloor\beta\rfloor+1}}
\frac{1}{\|\chi_{Q_j}\|_{(E_\Phi^{q})_t(\rn)}}.
\end{align*}

For $\mathrm{II}_2$, by the Taylor remainder theorem, the vanishing moments of $a_j$, the fact that $\lfloor\beta\rfloor\le n(\frac{1}{\min\{q,p_{\Phi}^-\}}-1)\le d$,
(\ref{71}) and the H\"older inequality, similarly to the estimation of \cite[(7.10)]{yyyz}, we conclude that,
for any
$\tau\in(0,\infty)$ and $x\notin4Q_j$,
\begin{align*}
\mathrm{II}_2
&\lesssim\frac{1}{|x-x_j|^{n+\lfloor\beta\rfloor+1}}\int_{2r_j\le|y-x_j|<\frac{|x-x_j|}{2}}
|y-x_j|^{\lfloor\beta\rfloor+1}\int_{Q_j} |a_j(z)|
\frac{|z-x_j|^\beta}{|y-x_j|^{n+\beta}}\,dz\,dy\\
&\lesssim\frac{r_j^{\beta}}{|x-x_j|^{n+\lfloor\beta\rfloor+1}}  \int_{2r_j\le|y-x_j|<\frac{|x-x_j|}{2}}
\frac{1}{|y-x_j|^{n+\beta-\lfloor\beta\rfloor-1}}\,dy\|a_j\|_{L^2(\rn)}|Q_j|^{\frac{1}{2}}\\
&\lesssim\frac{r_j^{n+\beta}}{|x-x_j|^{n+\beta}}\frac{1}{\|\chi_{Q_j}\|_{(E_\Phi^{q})_t(\rn)}}.
\end{align*}

For $\mathrm{II}_3$, by the vanishing moments of $a_j$, the fact that $\lfloor\beta\rfloor\le n(\frac{1}{\min\{q,p_{\Phi}^-\}}-1)\le d$,
(\ref{71}) and the H\"older inequality,
similarly to the estimation of \cite[(7.11)]{yyyz}, we know that, for any
$\tau\in(0,\infty)$ and $x\notin4Q_j$,
\begin{align*}
\mathrm{II}_3
&\lesssim\int_{|y-x_j|\ge\frac{|x-x_j|}{2}}
|\varphi_\tau(x-y)|\int_{Q_j} |a_j(z)|\frac{|z-x_j|^\beta}{|y-x_j|^{n+\beta}}\,dz\,dy\\
&\quad +\int_{|y-x_j|\ge\frac{|x-x_j|}{2}}
\lf|\frac{1}{\tau^n}\sum_{|\alpha|\le\lfloor\beta\rfloor}
\frac{\partial^\alpha\varphi(\frac{x-x_j}{\tau})}{\alpha!}
\lf(\frac{y-x_j}{\tau}\r)^{\alpha}\r|
\int_{Q_j} |a_j(z)|\frac{|z-x_j|^\beta}{|y-x_j|^{n+\beta}}\,dz\,dy\\
&\lesssim\frac{|r_j|^\beta}{|x-x_j|^{n+\beta}}\|a_j\|_{L^2(\rn)}|Q_j|^{\frac{1}{2}}
\int_{|y-x_j|\ge\frac{|x-x_j|}{2}}|\varphi_\tau(x-y)|\,dy\\
&\quad +\sum_{|\alpha|\le\lfloor\beta\rfloor} |r_j|^\beta \|a_j\|_{L^2(\rn)} |Q|^{\frac{1}{2}}
\int_{|y-x_j|\ge\frac{|x-x_j|}{2}}
\frac{1}{\tau^n}\frac{\tau^{n+\alpha}}{|x-x_j|^{n+\alpha}}
\frac{|y-x_j|^{\alpha}}{\tau^\alpha}\frac{1}{|y-x_j|^{n+\beta}}\,dy\\
&\lesssim
\frac{1}{\|\chi_{Q_j}\|_{(E_\Phi^{q})_t(\rn)}}\frac{r_j^{n+\beta}}{|x-x_j|^{n+\beta}}.
\end{align*}

Combining (\ref{79}) and the estimates of $\mathrm{II}_1$, $\mathrm{II}_2$ and
$\mathrm{II}_3$, we conclude that, for any $x\notin4Q_j$,
\begin{align*}
M(\mathcal{T}(a_j),\varphi)(x)
&=\sup_{\tau\in(0,\infty)}|\varphi_\tau\ast \mathcal{T}(a_j)(x)|\\
&\lesssim\frac{r_j^{n+\beta}}{|x-x_j|^{n+\beta}}
\frac{1}{\|\chi_{Q_j}\|_{(E_\Phi^{q})_t(\rn)}}\lesssim\lf[\mathcal{M}(\chi_{Q_j})(x)\r]^{\frac{n+\delta}{n}}
\frac{1}{\|\chi_{Q_j}\|_{(E_\Phi^{q})_t(\rn)}},
\end{align*}
which further implies that, for any $x\in\rn$,
$$
M(\mathcal{T}(a_j),\varphi)(x)\chi_{4Q_j}(x)\lesssim\lf[\mathcal{M}(\chi_{Q_j})(x)\r]^{\frac{n+\delta}{n}}
\frac{1}{\|\chi_{Q_j}\|_{(E_\Phi^{q})_t(\rn)}}.
$$
Then, by an argument similar to that used in the proof of Theorem \ref{cz}, we know that
\eqref{yuan} holds true,
which completes the proof of Theorem \ref{cz2}.
\end{proof}

\begin{theorem}\label{cz3}
Let $t\in(0,\infty)$, $q\in(0,2)$, $\beta\in(0,\infty)\setminus\mathbb{N}$ and $\Phi$ be an Orlicz function with positive lower type $p_{\Phi}^-$ satisfying $\min\{p_{\Phi}^-,\ q\}\in(\frac{n}{n+\beta},\frac{n}{n+\lfloor\beta\rfloor}]$ and positive upper type
$p_{\Phi}^+\in(0,2)$.
Let $\mathcal{T}$ be a $\beta$-order
Calder\'on-Zygmund operator.
Then $\mathcal{T}$ has a unique extension from $(HE_\Phi^{q})_t(\rn)$ to $(E_\Phi^{q})_t(\rn)$
and, for any $f\in(HE_\Phi^{q})_t(\rn)$,
$$
\|\mathcal{T}(f)\|_{(E_\Phi^{q})_t(\rn)}\le C\|f\|_{(HE_\Phi^{q})_t(\rn)},
$$
where $C$ is positive constant independent of $f$ and $t$.
\end{theorem}

\begin{proof}
Let $\{\lambda_j\}_{j=1}^{\infty}$ and $\{a_j\}_{j=1}^{\infty}$ be the same as
in the proof of Theorem \ref{cz}. By an argument similar to that used in the proof of Theorem \ref{cz}, we know that, to prove Theorem \ref{cz3},
it suffices to show that
\begin{equation}\label{yuan2}
\lf\|\sum_{j=1}^{\infty}\lambda_j\mathcal{T}(a_j)\r\|_{(E_\Phi^q)_t(\rn)}\lesssim1.
\end{equation}
To this end, it is easy to see that
\begin{align*}
\lf\|\sum_{j=1}^{\infty}\lambda_j\mathcal{T}(a_j)\r\|_{(E_\Phi^q)_t(\rn)}
&\lesssim\lf\|\sum_{j=1}^{\infty}\lambda_j\mathcal{T}(a_j)
\chi_{4\sqrt{n}Q_j}\r\|_{(E_\Phi^q)_t(\rn)}\\
&\quad+\lf\|\sum_{j=1}^{\infty}\lambda_j\mathcal{T}(a_j)
\chi_{\rn\setminus 4\sqrt{n}Q_j}\r\|_{(E_\Phi^q)_t(\rn)}
=:\mathrm{I}+\mathrm{II},
\end{align*}
where, for any $j\in\mathbb{N}$, $Q_j:=Q(x_j,r_j)$ is the same as in the proof of Theorem \ref{cz}.

For $\mathrm{I}$, by the boundedness of $\mathcal{T}$ on $L^2(\rn)$, we conclude that,
for any $j\in\mathbb{N}$,
$$
\|\mathcal{T}(a_j)\|_{L^2(\rn)}
\lesssim\|a_j\|_{L^2(\rn)}\lesssim\frac{|Q_j|^{\frac{1}{2}}}{\|\chi_{Q_j}\|_{(E_\Phi^q)_t(\rn)}},
$$
which, together with Lemma \ref{czsuan}, implies that
$$
\mathrm{I}
\lesssim\lf\|\lf\{\sum_{j=1}^{\infty}\lf[\lambda_j\mathcal{T}(a_j)
\chi_{4\sqrt{n}Q_j}\r]^s\r\}^{\frac{1}{s}}\r\|_{(E_\Phi^q)_t(\rn)}
\lesssim\lf\|\lf\{\sum_{j=1}^{\infty} \lf[\frac{\lambda_j}
{\|\chi_{Q_j}\|_{(E_\Phi^q)_t(\rn)}} \r]^s\chi_{Q_j} \r\}
^{\frac{1}{s}}\r\|_{(E_\Phi^q)_t(\rn)}
\lesssim1.
$$
This is a desired estimate.

Next, we deal with $\mathrm{II}$. To this end,
from the Taylor remainder theorem,
vanishing moments of $a_j$, the fact that $\lfloor\beta\rfloor\le n(\frac{1}{\min\{q,p_{\Phi}^-\}}-1)$ implies $\lfloor\beta\rfloor\le d$ and the H\"older inequality, it follows that,
for any $j\in\mathbb{N}$ and $z\in Q_j$, there exists
$\xi(z)\in Q_j$ such that, for any
$x\notin4\sqrt{n}Q_j$,
\begin{align*}
|\mathcal{T}(a_j)(x)|
&\leq\int_{Q_j}|a(z)||k(x,z)|\,dz\\
&=\int_{Q_j}|a(z)|\lf|k(x,z)-\sum_{|\alpha|<\lfloor\beta\rfloor}\frac{\partial_x^\alpha k(x,x_j)
}{\alpha!}(z-x_j)^\alpha\r|\,dz\\
&\sim\int_{Q_j}|a(z)|\lf|\sum_{|\alpha|=\lfloor\beta\rfloor}\frac{\partial_x^\alpha
k(x,x_j)-\partial_x^\alpha k(x,\xi(z))
}{\alpha!}(z-x_j)^\alpha\r|\,dz\\
&\lesssim\int_{Q_j}|a(z)|\frac{r_j^\beta}{|x-x_j|^{n+\beta}}\,dz
\lesssim\frac{r_j^\beta}{|x-x_j|^{n+\beta}}\|a_j\|_{L^2(\rn)}|Q_j|^{\frac{1}{2}}\\
&\lesssim\frac{r_j^{n+\beta}}{|x-x_j|^{n+\beta}}\frac{1}{\|\chi_{Q_j}\|_{(E_\Phi^{q})_t(\rn)}}
\lesssim\lf[\mathcal{M}(\chi_{Q_j})(x)\r]^{\frac{n+\delta}{n}}
\frac{1}{\|\chi_{Q_j}\|_{(E_\Phi^{q})_t(\rn)}}.
\end{align*}
Then, by an argument similar to that used in the proof of Theorem \ref{cz}, we know that
\eqref{yuan2} holds true,
which completes the proof of Theorem \ref{cz3}.
\end{proof}

\begin{remark}
\begin{itemize}
\item[(i)] Let $t$, $\Phi$ be as in Theorem \ref{cz2}.
Notice that, when $\beta:=\delta\in(0,1)$, the operators $\mathcal{T}$
in Theorems \ref{cz2} and \ref{cz3} is just
a non-convolutional $\delta$-type Calder\'on-Zygmund operator.
Thus, the operators in Theorems \ref{cz2} and \ref{cz3} include the
non-convolutional $\delta$-type Clader\'on-Zygmund operators as special cases.
Observe that, differently from Theorem \ref{cz}, in Theorems \ref{cz2}
and \ref{cz3}, we have a restriction on the ranges of $q$ and $p_\Phi^+$,
namely, $q,\ p_\Phi^+\in (0, 2)$, which is caused by the fact that the $\beta$-order
Calder\'on-Zygmund operator is only known bounded on $L^r(\rn)$ for any $r\in (1,2]$
(see, for example, \cite[Theorem 5.10]{d}). Thus, by \cite[Theorem 5.10]{d} again,
if we further assume that the kernel $k$ of $\mathcal{T}$ satisfies
(5.11) of \cite[Theorem 5.10]{d}, we can then remove this restriction.

\item[(ii)] Let $t\in(0,\infty,),\ r\in(1,\infty),\ \delta\in(0,1]$ and
$q\in(\frac{n}{n+\delta},1].$
Recall that Auscher and Prisuelos-Arribas \cite[Proposition 8.4]{ap} proved that
the non-convolutional $\delta$-type Clader\'on-Zygmund operators
are bounded from $(\mathfrak C_r^q)_t(\rn)$ to $(E_r^q)_t(\rn)$ and from
$(\mathfrak C_r^q)_t(\rn)$ to $(\mathfrak C_r^q)_t(\rn)$.

In Theorems \ref{cz2} and \ref{cz3}, if let $\beta:=\delta\in(0,1],\ s\in(\frac{n}{n+\delta},q]$
and $\Phi(\tau):=\tau^s$ for any $\tau\in[0,\infty)$, then we know that
the non-convolutional $\delta$-type
Clader\'on-Zygmund operators are bounded from
$(HE_s^q)_t(\rn)$ to $(E_s^q)_t(\rn)$ and from $(HE_s^q)_t(\rn)$ to $(HE_s^q)_t(\rn).$
By Propositions \ref{au} and \ref{ggg}, we know that $(\mathfrak C_r^q)_t(\rn)\subsetneqq
(HE_s^q)_t(\rn)$ and $(E_r^q)_t(\rn)\subsetneqq (E_s^q)_t(\rn)$ and hence
\cite[Proposition 8.4]{ap} and
Theorems \ref{cz2} and \ref{cz3} in this article can not cover each other.

\item[(iii)]When $t,\ q\in(0,\infty)$ and $\Phi(\tau):=\tau^q$ for any $\tau\in[0,\infty)$,
$(HE_\Phi^q)_t(\rn)$ and $(E_\Phi^q)_t(\rn)$ respectively become the classical
Hardy space $H^q(\rn)$ and Lebesgue space $L^q(\rn)$.
In this case, we know that, if $\delta\in (0,1]$ and $q\in(\frac{n}{n+\delta}, 1]$,
then Theorems \ref{cz2} and \ref{cz3} and (ii) of this remark give the boundedness of the
classical $\delta$-type Clader\'on-Zygmund operator from $H^q(\rn)$ to
$L^q(\rn)$ and from $H^q(\rn)$ to itself,
which is well known (see, for example, \cite[Theorem 1.1]{am1},
\cite[p.\,115, Theorem 4]{em}, \cite[p.\,109, Theorem 4.1 and p.\,119, Theorem 4.5]{l}).
\end{itemize}
\end{remark}

\bigskip

\noindent  Yangyang Zhang, Dachun Yang (Corresponding author) and Wen Yuan

\smallskip

\noindent  Laboratory of Mathematics and Complex Systems
(Ministry of Education of China),
School of Mathematical Sciences, Beijing Normal University,
Beijing 100875, People's Republic of China

\smallskip

\noindent {\it E-mails}: \texttt{yangyzhang@mail.bnu.edu.cn} (Y. Zhang)

\noindent\phantom{{\it E-mails:}} \texttt{dcyang@bnu.edu.cn} (D. Yang)

\noindent\phantom{{\it E-mails:}} \texttt{wenyuan@bnu.edu.cn} (W. Yuan)

\bigskip

\noindent Songbai Wang

\smallskip

\noindent College of Mathematics and Statistics, Hubei Normal University,
Huangshi 435002, People's Republic of China

\smallskip

\noindent{\it E-mail}: \texttt{haiyansongbai@163.com}

\end{document}